\documentclass[reqno,11pt]{amsart}

\usepackage[english]{babel}
\usepackage[utf8]{inputenc}
\usepackage[T1]{fontenc}

\usepackage[margin=1in, letterpaper]{geometry}
\usepackage{amssymb,amsmath,amsfonts,amsthm}
\usepackage{bm,bbm,mathrsfs}
\usepackage[dvipsnames]{xcolor}
\usepackage{url}
\usepackage{enumitem}
\usepackage{mathtools}
\usepackage{placeins}
\usepackage{extarrows}
\usepackage{subfig}
\usepackage[font=small,labelfont=bf]{caption}
\usepackage{floatrow}
\usepackage{verbatim}
\usepackage{wrapfig}
\usepackage{tikz}
\usetikzlibrary{arrows,arrows.meta,cd,decorations.pathmorphing,backgrounds,positioning,fit,matrix}
\usepackage{xcolor,colortbl,graphics,graphicx}
\usepackage[colorlinks=true, allcolors=Blue, pdfencoding=auto]{hyperref}
\usepackage{aliascnt}
\usepackage[noabbrev,capitalize]{cleveref}

\newtheorem{theorem}{Theorem}[section]
\newaliascnt{lemma}{theorem}
\newtheorem{lemma}[lemma]{Lemma}
\aliascntresetthe{lemma}
\crefname{lemma}{Lemma}{Lemmas} 
\newaliascnt{proposition}{theorem}
\newtheorem{proposition}[proposition]{Proposition}
\aliascntresetthe{proposition}
\crefname{proposition}{Proposition}{Propositions}
\newaliascnt{corollary}{theorem}
\newtheorem{corollary}[corollary]{Corollary}
\aliascntresetthe{corollary}
\crefname{corollary}{Corollary}{Corollaries}
\newaliascnt{example}{theorem}
\newtheorem{example}[example]{Example}
\aliascntresetthe{example}
\crefname{example}{Example}{Examples}
\newaliascnt{question}{theorem}

\aliascntresetthe{question}
\crefname{question}{Question}{Questions}
\newaliascnt{conjecture}{theorem}
\newtheorem{conjecture}[conjecture]{Conjecture}
\aliascntresetthe{conjecture}
\crefname{conjecture}{Conjecture}{Conjectures}
\newaliascnt{assumption}{theorem}

\aliascntresetthe{assumption}
\crefname{assumption}{Assumption}{Assumptions}
\newaliascnt{definition}{theorem}
\newtheorem{definition}[definition]{Definition}
\aliascntresetthe{definition}
\crefname{definition}{Definition}{Definitions}
\newaliascnt{notation}{theorem}
\newtheorem{notation}[notation]{Notation}
\aliascntresetthe{notation}
\crefname{notation}{Notation}{Notations}
\theoremstyle{remark}
\newtheorem*{remark}{Remark}
\crefname{equation}{}{}
\numberwithin{equation}{section}
\allowdisplaybreaks

\newcommand{\Disc}{\mathrm{Disc}}
\newcommand{\Mult}{\mathrm{Mult}}

\newcommand{\bigboxtimes}{%
  \mathop{\mathpalette\bigboxtimesA{}}%
}
\newcommand{\bigboxtimesA}[2]{%
  \vcenter{\hbox{\scalebox{2}{$\displaystyle\boxtimes$}}}%
}
\newcommand{\Ind}{\mathrm{Ind}}
\newcommand{\Irr}{\textnormal{Irr}}
\newcommand{\sym}{\textnormal{sym}}
\newcommand{\psum}{\ \sideset{}{^*}\sum}
\newcommand{\Z}{\mathbb{Z}}
\newcommand{\R}{\mathbb{R}}

\newcommand{\C}{\mathbb{C}}

\renewcommand{\H}{\mathbb{H}}
\renewcommand{\P}{\mathbb{P}}

\newcommand{\mA}{\mathcal{A}}
\newcommand{\mB}{\mathcal{B}}

\newcommand{\mI}{\mathcal{I}}
\newcommand{\mJ}{\mathcal{J}}

\newcommand{\mL}{\mathcal{L}}

\newcommand{\mN}{\mathcal{N}}

\newcommand{\mS}{\mathcal{S}}
\newcommand{\mT}{\mathcal{T}}

\newcommand{\ma}{\mathfrak{a}}

\newcommand{\rzero}{\mathbf{0}}
\newcommand{\rone}{\mathbf{1}}
\newcommand{\one}{\mathbbm{1}}

\newcommand{\GL}{\textnormal{GL}}
\newcommand{\SL}{\textnormal{SL}}
\newcommand{\PGL}{\textnormal{PGL}}
\newcommand{\PSL}{\textnormal{PSL}}
\newcommand{\Id}{\textnormal{Id}}
\newcommand{\Tr}{\textnormal{Tr}}
\newcommand{\la}{\left\langle}
\newcommand{\ra}{\right\rangle}
\newcommand{\lf}{\left\lfloor}
\newcommand{\rf}{\right\rfloor}
\newcommand{\lc}{\left\lceil}
\newcommand{\rc}{\right\rceil}
\newcommand{\eps}{\varepsilon}

\renewcommand{\bar}{\overline}
\renewcommand{\hat}{\widehat}

\renewcommand{\mod}[1]{\ \textnormal{mod } #1}
\renewcommand{\pmod}[1]{\ (\textnormal{mod } #1)}

\newenvironment{psmall}
  {\left(\begin{smallmatrix}}
  {\end{smallmatrix}\right)}

\definecolor{myBlue}{rgb}{0, 0, 0.6}
\definecolor{myRed}{rgb}{0.6, 0, 0}
\newcommand{\fix}[1]{\textcolor{myBlue}{\large (#1)\normalsize}}

\title{Non-abelian amplification and bilinear forms with Kloosterman sums}
\author[Alexandru Pascadi]{Alexandru Pascadi}
\address{Mathematisches Institut, Endenicher Allee 60, 53115 Bonn, Germany}
\email{alexpascadi@gmail.com}

\begin{document}

\begin{abstract}
    We introduce a new method to bound bilinear (Type II) sums of Kloosterman sums with composite moduli $c$, using Fourier analysis on $\mathrm{SL}_2(\mathbb{Z}/c\mathbb{Z})$ and an amplification argument with non-abelian characters. For sums of length $\sqrt{c}$, our method produces a non-trivial bound for all moduli except near-primes, saving $c^{-1/12}$ for products of two primes of the same size. Combining this with previous results for prime moduli, we achieve savings beyond the P\'olya--Vinogradov range for all moduli. We give applications to moments of twisted cuspidal $L$-functions, and to large sieve inequalities for exceptional cusp forms with composite levels.
\end{abstract}
\maketitle

{
\setlength{\parskip}{0em}
\setcounter{tocdepth}{1}
\tableofcontents
}
\vspace{-1cm}

\section{Introduction} \label{sec:intro}

\subsection{Brief background}
There is by now a fairly comprehensive history of bounds for bilinear forms with Kloosterman sums and their applications \cite{blomer2015second,blomer2017moments,blomer2023second,fouvry2014algebraic,kowalski2017bilinear,kowalski2020stratification,deshouillers1982kloosterman,maynard2025primes,pascadi2026large,shparlinski2016cancellations,shparlinski2019sums,kerr2023bounds,xi2018ternary,wu2021arithmetic}. In the simplest form, the objects of interest are the sums
\begin{equation} \label{eq:kloosterman}
    \sum_{m \le M} \sum_{n \le N} \alpha_m \beta_n S(m, n; c),
    \qquad
    \text{where}
    \qquad
    S(m, n; c) := \sum_{x \in (\Z/c\Z)^\times} e\left(\frac{mx + n\bar{x}}{c}\right),
\end{equation}
for positive integers $c, M, N$ with $M, N \le c$ and complex sequences $(\alpha_m)$, $(\beta_n)$; here $e(t) := \exp(2\pi i t)$ and $x\bar{x} \equiv 1 \pmod{c}$. In this work, we are mainly concerned with the `Type II' setting where $(\alpha_m)$ and $(\beta_n)$ are arbitrary sequences, and we search for an upper bound in terms of their $\ell^2$ norms $\|\alpha\| := (\sum_m |\alpha_m|^2)^{1/2}$, $\|\beta\| := (\sum_n |\beta_n|^2)^{1/2}$. This is equivalent to bounding the operator norm, or the largest singular value, of the $M \times N$ matrix $(S(m,n;c))_{m\le M, n \le N}$.

For the Type II sums, it is in general necessary to incorporate a coprimality constraint $(m, n, c) = 1$. In practice, since $S(gm, gn; gc) = \tfrac{\phi(gc)}{\phi(c)}S(m,n;c)$, one can separately consider each value of $(m, n, c)$; therefore, in most bounds discussed henceforth, one can replace the restriction $(m, n, c) = 1$ with the assumption that $(\alpha_m)$ and $(\beta_n)$ are $1$-bounded (and the norms $\|\alpha\|$, $\|\beta\|$ with $\sqrt{M}$, $\sqrt{N}$).

There are two main `trivial' bounds to beat, packaged together into the following inequality with a slightly more general setup. For any (integer) intervals $\mI, \mJ \subset \Z$ with $|\mI| = M \le c$, $|\mJ| = N \le c$, any complex sequences $(\alpha_m)_{m \in \mI}$, $(\beta_n)_{n \in \mJ}$, and any $a \in (\Z/c\Z)^\times$, one has
\begin{equation} \label{eq:trivial-bound}
\mathop{\sum\sum}_{\substack{m \in \mI, n \in \mJ \\ (m, n, c) = 1}} \alpha_m \beta_n S(am, n; c) 
\ll 
\|\alpha\| \|\beta\| c^{o(1)} \min\left(c, \sqrt{MNc}\right).
\end{equation}
The term $\sqrt{MNc}$ comes from the algebro-geometric Weil bound $S(am, n; c) \ll c^{o(1)} \sqrt{(m, n, c)c}$, applied pointwise.
The other term of $c$ comes from Fourier analysis, and can be refined for unbalanced sums with $M < \sqrt{c} < N$ via the P\'olya--Vinogradov method (see \cite[Theorem 1.17]{fouvry2014algebraic}). The best one could hope for is the perfect-orthogonality bound 
$\|\alpha \| \|\beta\| c^{o(1)} \sqrt{(M+N)c}$, but making any improvement over \cref{eq:trivial-bound} for balanced sums ($M \approx N$) is notoriously difficult; a key range is $M \approx N \approx \sqrt{c}$, when the trivial bounds match. We note that while many applications \cite{kowalski2017bilinear,kowalski2020stratification} require improving the pointwise Weil bound when $M, N \ll \sqrt{c}$, some applications use larger values of $M, N \le c^{1-\eps}$ and require improving the Fourier-theoretic bounds; this is the case in our \cref{sec:exc-large-sieve}. 

An important improvement of \cref{eq:trivial-bound} when $M, N \approx \sqrt{c}$ was the celebrated breakthrough of Kowalski--Michel--Sawin \cite{kowalski2017bilinear}, which requires a prime modulus $c = p$, and which saves a factor of $p^{-1/64}$ when $M, N \asymp \sqrt{p}$; see also \cite{kowalski2020stratification} for their follow-up work, which outperforms the pointwise Weil bound for $MN$ as small as $p^{3/4+\eps}$. These results rely on a shift-by-$ab$ trick of Vinogradov and Karatsuba, a H\"older step, and deep inputs of $\ell$-adic cohomology; notably, the same bounds hold for more general algebraic trace functions, including hyper-Kloosterman sums. 

Closer to our methods is an approach of Shkredov \cite[(6)]{shkredov2021modular} for prime moduli $p$, which relies (in the Type II setting) on non-abelian Fourier analysis \cite[Lemma 22]{shkredov2018asymptotic} and expansion in $\SL_2(\Z/p\Z)$ \cite[Theorem 50]{shkredov2018asymptotic}; this beats \cref{eq:trivial-bound} in a range of the shape $M, N \in (p^{1/2-\delta}, p^{1-\eps})$ with a small but effective power saving, and for sequences $(\alpha_m)$, $(\beta_n)$ with more general additively-structured supports. We also mention some additive-combinatorial approaches: of Shparlinski--Zhang \cite{shparlinski2016cancellations} for smooth sequences, of the author \cite[\S 4]{pascadi2026large} for additively-structured sequences, and of Shkredov \cite[(7)]{shkredov2021modular} and Kerr--Shparlinski--Wu--Xi \cite{kerr2023bounds} for Type I bilinear forms (where only $(\alpha_m)$ is smooth). 

For moduli with a suitable factorization, the best Type II bounds so far have come from the $q$-van der Corput method \cite{fouvry2019lectures}, which relies on the twisted multiplicativity of Kloosterman sums, Cauchy--Schwarz, and a shifting trick; it was first applied in this setting by Blomer--Mili\'cevi\'c \cite{blomer2015second}. The $q$-van der Corput method can also be iterated, leading to strong results for smooth square-free moduli \cite{xi2018ternary,wu2021arithmetic}. Unfortunately, these arguments fail to handle certain types of composite moduli when $MN \approx c$, including squares of primes and products of two distinct primes of the same size. 

\subsection{Main results.}
In this work, we develop a new method to bound bilinear forms with Kloosterman sums for essentially all composite moduli. Like Shkredov's Type II result for prime moduli \cite[Theorem 4, (6)]{shkredov2021modular}, our results rely on non-abelian Fourier analysis; however, we use the normal subgroups of $\SL_2(\Z/c\Z)$ to our advantage, and we avoid relying on $L^2$-flattening to obtain quantitatively-good\footnote{Note that Shkredov's Type I bounds \cite[(7)]{shkredov2021modular} do not rely on $L^2$-flattening and achieve better power savings.} power savings over \cref{eq:trivial-bound} (up to $c^{-1/12}$; see \cref{ex:p2-pq}). Our key innovation is a new type of amplification argument with non-abelian characters, detailed in \cref{subsec:outline-amplif}, which may be of independent interest.

Combining our bounds with those of Kowalski--Michel--Sawin\footnote{One could also combine our bounds with the results of Shkredov \cite[Theorem 4]{shkredov2021modular} for prime moduli, to obtain a result like \cref{thm:bilinear-forms-general} which does not rely on algebraic geometry.} \cite{kowalski2017bilinear} (as well as Blomer--Mili\'cevi\'c \cite{blomer2015second} for an optimization), we obtain a non-trivial result for general moduli beyond the P\'olya--Vinogradov range, given in \cref{thm:MN-bilinear-forms-general}. We state a particular case of this result below, when $M, N \ll c^{1/2+o(1)}$.

\begin{theorem} \label{thm:bilinear-forms-general}
Let $c, M, N \in \Z_+$ with $M, N \ll c^{1/2+o(1)}$. Then for any complex sequences $(\alpha_m)_{m \le M}$, $(\beta_n)_{n \le N}$ and $a \in (\Z/c\Z)^\times$, one has
\[
    \mathop{\sum_{m=1}^M\sum_{n=1}^N}_{(m,n,c)=1} \alpha_m \beta_n S(am, n; c)
    \ll 
    \|\alpha\|\|\beta\| c^{1-\frac{1}{700}+o(1)}.
\]
Moreover, if $|\alpha_m| \le 1$ for all $m$ (so $\|\alpha\| \le \sqrt{M}$), then 
\[
    \mathop{\sum_{m=1}^M\sum_{n=1}^N}_{(n,c)=1} \alpha_m \beta_n S(am, n; c)
    \ll 
    \sqrt{M}\|\beta\| c^{1-\frac{1}{276}+o(1)}.
\]
\end{theorem}

Our main technical result leading to \cref{thm:bilinear-forms-general} is \cref{thm:MN-bilinear-forms-composite}, which considers a factorization of the modulus into three parts, $c = dd'e$ (but one can usually take $d' = 1$ or $e = 1$).
Below we state a particular case of \cref{thm:MN-bilinear-forms-composite}, focusing on the same range $M, N \ll c^{1/2+o(1)}$ as in \cref{thm:bilinear-forms-general}.

\begin{theorem} \label{thm:bilinear-forms-composite}
Let $c = dd'e$ for some $d, d', e \in \Z_+$ with $d' \mid d$ and $(d, e) = 1$, and let $f$ be the largest integer such that $f^2 \mid cd$. Let $\mI, \mJ \subset \Z$ be intervals with $|\mI|, |\mJ| \ll c^{1/2+o(1)}$. Then for any complex sequences $(\alpha_m)_{m \in \mI}$, $(\beta_n)_{n \in \mJ}$ and $a \in (\Z/c\Z)^\times$, one has
\[
\begin{aligned}
    \mathop{\sum\sum}_{\substack{m \in \mI, n \in \mJ \\ (m, n, c) = 1}} \alpha_m \beta_n S(am, n; c) 
    &\ll
    \|\alpha\| \|\beta\| c^{1+o(1)}
    \left(\frac{f}{\min(c,d^2)} \right)^{\frac{1}{6}}.
\end{aligned}
\]
\end{theorem}

Since $f \le \sqrt{cd}$, \cref{thm:bilinear-forms-composite} automatically gives a non-trivial result when $d \in (c^{1/3+\eps}, c^{1-\eps})$ for $\eps > 0$. Unless $c$ has a prime factor larger than $c^{1-\eps}$, one can always find a factorization $c = dd'e$ with $d$ in this range, $(d, e) = 1$, and $d' \mid d$, which makes the general result in \cref{thm:bilinear-forms-general} possible.

\begin{example} \label{ex:p2-pq}
Let $d \mid c$ such that $\tfrac{c}{d}$ is square-free. Then one can take $d' := (\tfrac{c}{d}, d)$, $e := \tfrac{c}{dd'}$, $f = d$ in \cref{thm:bilinear-forms-composite}, so the saving over the trivial bound is $\min(d, \tfrac{c}{d})^{-1/6}$. If $d \asymp \sqrt{c}$, this is roughly
\[
    c^{-\frac{1}{12}}.
\]
In particular, this saving is achieved if 
$c \in \{p^2, pq\}$, where $p$ and $q$ are distinct primes with $p \asymp q$. For the same values of $c$ and any $\eps > 0$, the more general \cref{thm:MN-bilinear-forms-composite} beats \cref{eq:trivial-bound} in the range 
\[
    M \asymp N \in [c^{\frac{5}{12} + \eps}, c^{\frac{5}{8}-\eps}].
\]
\end{example}

Notably, while the values $c \in \{p^2, pq\}$, $M, N \asymp \sqrt{c}$ give blind spots of the $q$-van der Corput method, they happen to give one of the best cases for our methods; this case has until now constituted the remaining barrier towards the application in \cref{thm:mom-l-fns}.

As a quick corollary of \cref{thm:bilinear-forms-composite}, we prove a trilinear-sum bound which includes a short averaging over $c$ with a given large divisor $q$; the point is that only a factorization of $q$ (rather than $c$) is assumed. Such sums arise in the spectral theory of automorphic forms, in particular in \cref{sec:exc-large-sieve}. Below is such a trilinear-sum bound, which is a particular case of \cref{cor:MN-bilinear-forms-avg-c}.

\begin{corollary} \label{cor:bilinear-forms-avg-c}
Let $C \ge \tfrac{1}{2}$, $q = dd'e$ for some $d, d', e \in \Z_+$ with $d' \mid d$ and $(d, e) = 1$, and let $f$ be the largest integer such that $f^2 \mid qd$. Let $\mI, \mJ \subset \Z_+$ be intervals of lengths $|\mI|, |\mJ| \ll C^{1/2+o(1)}$ with $\max(\mI \cup \mJ) \ll C^{O(1)}$. Then for any complex sequences $(\alpha_m)_{m \in \mI}$, $(\beta_n)_{n \in \mJ}$, one has
\[
\begin{aligned}
    \sum_{\substack{C < c \le 2C \\ q \mid c}} \left\vert
    \mathop{\sum\sum}_{\substack{m \in \mI, n \in \mJ \\ (m, n, q) = 1}} \alpha_m \beta_n S(m, n; c) \right\vert 
    &\ll
    \|\alpha\|\|\beta\|
    \frac{C^{2+o(1)}}{q}
    \left(\frac{f}{\min(C,d^2)+\min(q,d^2 \frac{C}{q})} \right)^{\frac{1}{6}}.
\end{aligned}
\]
\end{corollary}

\begin{remark}
Mili\'cevi\'c, Qin and Wu \cite{milicevic2025bilinear} have simultaneously and independently obtained results similar to our \cref{thm:bilinear-forms-general,thm:mom-l-fns} using substantially different methods. The two papers are complementary, each performing slightly better in different ranges and for different types of moduli, and both achieving power savings for bilinear sums of square-root length and general moduli. The methods in \cite{milicevic2025bilinear} (which use algebraic geometry and build on Kowalski--Michel--Sawin \cite{kowalski2017bilinear} and Blomer--Mili\'cevi\'c \cite{blomer2015second}) obtain better savings for general moduli and remove the dependency on the Ramanujan--Petersson conjecture in the application to moments of twisted cuspidal $L$-functions. Our methods (which use non-abelian Fourier analysis and are closer to the work of Shkredov \cite{shkredov2021modular,shkredov2018asymptotic}) perform better and in longer ranges for specific classes of moduli $c$ (see \cref{ex:p2-pq,ex:best-case-range}), can handle more general supports of the sequences $(\alpha_m), (\beta_n)$ (intervals or other additively-structured subsets of $\Z/c\Z$), and find an application to exceptional-spectrum large sieve inequalities (\cref{cor:large-sieve}).
\end{remark}

\emph{Potential extensions.} Finally, we note a few more-speculative generalizations of our approach:
\begin{itemize}[leftmargin=0.5cm]
    \item One could similarly study other exponential sums with $\SL_2(\Z/c\Z)$ or $\PGL_2(\Z/c\Z)$ structure, for example by replacing the inverse $\bar{x}$ from \cref{eq:kloosterman} with some other M\"obius transformation $F(x)$.
    \item There may be Archimedean analogues of our results using the representation theory of $\SL_2(\R)$ rather than $\SL_2(\Z/c\Z)$, perhaps replacing Kloosterman sums with Bessel functions. A joint (Kloosterman times Bessel) bilinear result could be useful in conjunction with trace formulae.
    \item Related methods for $\SL_3(\Z/c\Z)$ could be worth investigating as well, potentially in connection to the generalized Kloosterman sums which arise in the $\GL_3$ Kuznetsov formula \cite{bump1988poincare,blomer2013applications,blomer2017applications,blomer2019global}.
\end{itemize}

\subsection{Applications} \label{subsec:applications}

As our first application, we prove an asymptotic for the averaged second moment of modular $L$-functions twisted by primitive Dirichlet characters modulo $q$, where the modulus $q$ is arbitrary. Blomer--Mili\'cevi\'c \cite{blomer2017moments} established such an asymptotic for most moduli, more specifically whenever $q$ is not close to a prime or to a product of two primes of the same size. The missing ingredient in these cases has been precisely a power-saving bound for bilinear forms with Kloosterman sums modulo $q$, where both sums have length $\approx \sqrt{q}$. The case of prime moduli $q$ was established by Kowalski--Michel--Sawin \cite[Theorem 1.5]{kowalski2017bilinear}, and the remaining case can now be handled using our \cref{thm:bilinear-forms-general} (essentially in the setting from \cref{ex:p2-pq}).

To state this application, we introduce some quick notation as in \cite{blomer2015second}. Given $q \in \Z_+$, we write
\[
    \phi^*(q) := \sum_{d \mid q} \phi(d)\, \mu\left(\frac{q}{d}\right)
\]
for the number of primitive characters modulo $q$; this vanishes if and only if $q \equiv 2 \pmod{4}$, and is otherwise of size $\gg q^{1-o(1)}$. We write $\sum^*_{\chi \mod{q}}$ for a sum over all primitive characters modulo $q$.
Also following \cite{blomer2015second},
given an $L$-function $L(s)$, we write $L_q(s)$ for the product $\prod_{p \mid q} L_p(s)$ over all local factors at primes dividing $q$; thus for example $\zeta_q(s) = \prod_{p \mid q} (1 - p^{-s})^{-1}$.

\begin{theorem} \label{thm:mom-l-fns}
Let $f_1, f_2$ be fixed holomorphic cuspidal newforms for $\SL_2(\Z)$ with even weights $\kappa_1, \kappa_2$, and let $q \in \Z_+$. 
Provided that $\kappa_1 \equiv \kappa_2 \pmod{4}$, one has the asymptotic
\[
    \psum_{\chi \mod{q}} L(\tfrac{1}{2}, f_1 \otimes \chi) \bar{L(\tfrac{1}{2}, f_2 \otimes \chi)}
    =
    \frac{2\phi^*(q)}{\zeta(2)} M(f_1, f_2; q) + O_{f_1,f_2}\left(q^{1-\frac{1}{674}+o(1)}\right),
\]
with a main term of
\[
    M(f_1, f_2; q) := \begin{cases}
    P(1)\,L(1, \sym^2 f_1) \left(\log q + c(f_1) + \frac{P'(1)}{P(1)}\right), & f_1 = f_2,
    \\
    Q(1)\,L(1, f_1 \times f_2), & f_1 \neq f_2,
    \end{cases}
\]
where $c(f_1)$ is a constant depending only on $f_1$ and
\[
    P(s) := \frac{\zeta_q(2s)}{L_q(s,\sym^2 f_1)},
    \qquad\qquad 
    Q(s) := \frac{\zeta_q(2s)}{L_q(s,f_1 \times f_2)}.
\]
\end{theorem}

\begin{remark}
Similar results can be obtained for Maass cusp forms, with some care in removing the dependency on the Ramanujan conjecture; see also \cite{milicevic2025bilinear}. The case of (non-cuspidal) Eisenstein series reduces to the result of Young \cite{young2011fourth} on fourth moments of Dirichlet $L$-functions for prime moduli, extended to all moduli by Wu \cite{wu2023fourth}.
\end{remark}

Our second application concerns the exceptional spectrum of the hyperbolic Laplacian on $\Gamma_0(q)\backslash \H$, consisting of Maass cusp forms of level $q \in \Z_+$ with eigenvalues $\lambda < \tfrac{1}{4}$. Selberg's eigenvalue conjecture \cite{selberg1965estimation}, one of the central open problems in the theory of $\GL_2$ automorphic forms, states that this exceptional spectrum is empty. However, unconditionally, exceptional forms often produce the worst contribution in applications of the Kuznetsov trace formula \cite{deshouillers1982kloosterman,kuznetsov1980petersson} to analytic number theory problems \cite{drappeau2023one,topacogullari2018shifted,de2020niveau,pascadi2026large}, losing exponential factors in the parameter $\theta := (\tfrac{1}{4}-\lambda)^{1/2}$. The best known pointwise bound is Kim--Sarnak $\theta \le \tfrac{7}{64}$ \cite[Appendix 2]{kim2003functoriality}, but on-average results can also lead to savings in the $\theta$-aspect, sometimes enough to match the conditional results  \cite{pascadi2025exponents,grimmelt2025greatest}. Following Deshouillers--Iwaniec \cite{deshouillers1982kloosterman}, these on-average results often take the shape of large sieve inequalities for the Fourier coefficients of exceptional Maass forms, incorporating factors of $X^{2\theta}$. While improvements are now possible \cite{pascadi2026large} for exceptional-spectrum large sieve inequalities with special sequences $(\alpha_n)_{n \le N}$, the savings in the $\theta$-aspect for arbitrary sequences have been limited to $(\tfrac{q}{N})^{2\theta}$, due to Deshouillers--Iwaniec \cite[Theorem 5]{deshouillers1982kloosterman}. In fact, obtaining \emph{any} power saving for arbitrary sequences when $N \asymp q$ is as hard as proving Selberg's eigenvalue conjecture \cite[\S 2]{pascadi2026large}.

In \cref{thm:large-sieve-comp-lvl}, we overcome this barrier at $(\tfrac{q}{N})^{2\theta}$ if $q$ is suitably-composite and $N$ is not too large, using \cref{thm:MN-bilinear-forms-composite}. We note that in many applications \cite{bombieri1986primes,maynard2025primes,drappeau2023one,deshouillers1982power,deshouillers1984power}, the level $q$ is a product of two factors of similar sizes, and $N \in (\sqrt{q}, q)$. We state below a particular case of \cref{thm:large-sieve-comp-lvl}, for $N \approx \sqrt{q}$. We point the reader to \cref{sec:exc-large-sieve} for more background and notation.

\begin{corollary} \label{cor:large-sieve}
Let $q \in \Z_+$ have a divisor $d \asymp \sqrt{q}$ such that $\tfrac{q}{d}$ is square-free. Consider an orthonormal basis of Maass cusp forms for $\Gamma_0(q)$, with Laplacian eigenvalues $\lambda_j$ and Fourier coefficients $(\rho_j(n))_{n \in \Z}$ around $\infty$ (normalized as in \cite{deshouillers1982kloosterman,pascadi2026large}). Let $N \ll q^{\frac{1}{2}+o(1)}$, and $(\alpha_n)_{N < n \le 2N}$ be a complex sequence supported on $(n, q) = 1$. Then with $\theta_j := (\tfrac{1}{4}-\lambda_j)^{1/2} \le \tfrac{7}{64}$, one has
\begin{equation} 
    \sum_{\lambda_j < \frac{1}{4}}
    q^{\frac{6}{5} \theta_j} 
    \left\vert 
    \sum_{N < n \le 2N} \alpha_n\, \rho_j(n)
    \right\vert^2 
    \ll
    (qN)^{o(1)}
    \|\alpha\|^2.
\end{equation}
\end{corollary}
For reference, \cite[Theorem 5]{deshouillers1982kloosterman} of Deshouillers--Iwaniec would include a factor of $(\tfrac{q}{N})^{2\theta_j}$ (which is $q^{\theta_j}$ when $N = \sqrt{q}$) in the left-hand side, so \cref{cor:large-sieve} wins a factor of $q^{\theta/5}$ in this case. 

\begin{remark} It follows from the more general \cref{thm:large-sieve-comp-lvl} that one can relax the condition that $\tfrac{q}{d}$ is square-free when some averaging over levels $q \le Q$ with $d \mid q$ (and $d \asymp \sqrt{Q}$) is available. The sequence $(\alpha_n)$ inside the large sieve may depend on $q$ in this case, unlike in \cite[Theorem 6]{deshouillers1982kloosterman}.
\end{remark}

\subsection{Acknowledgements}
The author is deeply grateful to Valentin Blomer, James Maynard, Sary Drappeau, Philippe Michel, Emmanuel Kowalski, Ilya D.\ Shkredov, and the referees, for many helpful comments. This work was supported by the ERC Advanced Grant 101054336 and Germany’s Excellence Strategy grant EXC-2047/1-390685813. For a part of the duration of this project, the author was also supported by an EPSRC Scholarship, as well as a Campus France Scholarship.

\section{Outline} \label{sec:outline}

\subsection{Structure of the paper} \label{subsec:structure}
Our proof of \cref{thm:bilinear-forms-composite} has three main steps:
\begin{itemize}
    \item[I] (\emph{Fourier analysis}). In \cref{sec:rep-and-kloost} (particularly, \cref{prop:kloost-to-fourier}), we relate matrices of Kloosterman sums to the Fourier transform of certain functions at a special representation $\rho_c^\circ$ of $\SL_2(\Z/c\Z)$. This involves Fourier analysis on both abelian ($\Z/c\Z$, $\R$) and non-abelian ($\SL_2(\Z/c\Z)$) groups, and a M\"obius inversion process for representations of $\SL_2(\Z/c\Z)$. \vspace{0.1cm}
    \item[II] (\emph{Amplification}). In \cref{sec:amplif} (particularly, \cref{prop:fourier-to-counting}), we upper bound the spectral norm of the above Fourier coefficients by a weighted count of solutions to an equation in $\PSL_2(\Z/d\Z)$, with $d \mid c$. This is where our non-abelian amplification argument comes in. \vspace{0.1cm}
    \item[III] (\emph{Combinatorics}). In \cref{sec:counting} (particularly, \cref{prop:counting-6}), we analyze this counting problem modulo $d$ using elementary arguments, similar to \cite[\S 5]{shkredov2021modular}.
\end{itemize}

In \cref{sec:bil-forms}, we combine these ingredients to deduce \cref{thm:bilinear-forms-composite} and its variations. The applications to moments of twisted cuspidal $L$-functions and large sieve inequalities for exceptional cusp forms are handled in \cref{sec:mom-l-functions,sec:exc-large-sieve}, respectively.

For the rest of this section, we give a brief informal overview of our argument, ignoring various technical details. We will use the symbols `$\approx$', `$\lesssim$' for identities and inequalities that are `morally' true (and can be made rigorous with minor modifications, such as including $c^{o(1)}$ factors).

\subsection{First steps: Fourier analysis} \label{subsec:outline-Fourier}
Let us focus on the balanced case $M = N$. We begin by considering the $N \times N$ complex matrix
\[
    K := \left(S(m, n; c)\right)_{m, n \le N}.
\]
where $c, N \in \Z_+$ with $N \le c$. Our task is to bound the operator norm $\|K\|$ by less than $\min(c, N\sqrt{c})$, to beat \cref{eq:trivial-bound}. We extend $K$ to a $c \times c$ matrix, and multiply it on both sides by the unitary matrix $U := (\tfrac{1}{\sqrt{c}} e(\tfrac{xy}{c}))_{x, y \in \Z/c\Z}$, which preserves the norm and essentially amounts to taking a Fourier transform in the $m, n$ variables. Letting $H \approx \tfrac{c}{N}$, a truncated version of Poisson summation yields
\[
    \frac{1}{c}\, U^*KU \approx 
    \frac{1}{H^2} \left(\sum_{|h| \le H} \mT^{h} \right) \mS \left(\sum_{|h| \le H} \mT^{h}\right),
    \qquad 
    \text{where}
    \qquad 
    \begin{cases}
    \mT := (\one_{u = x+1})_{u, x \in \Z/c\Z}, \\
    \mS := (\one_{xy = -1})_{x, y \in \Z/c\Z}.
    \end{cases}
\]
By inserting a few more rows and columns, we can in fact work over the projective line $\P^1(\Z/c\Z)$ rather than $\Z/c\Z$. The matrices $\mT$ and $\mS$ then extend to $\rho_c(T)$ and $\rho_c(S)$, where $T$ and $S$ are the usual generators of $\SL_2(\Z/c\Z)$ (see \cref{eq:SL2-generators}), and 
\[
    \rho_c : \SL_2(\Z/c\Z) \to \left\{\text{Unitary maps of } L^2(\P^1(\Z/c\Z))\right\}
\]
is the $c^{1+o(1)}$-dimensional permutation representation corresponding to the action of $\SL_2(\Z/c\Z)$ on $\P^1(\Z/c\Z)$ by M\"obius transformations. It then remains to bound the spectral norm of the matrix
\[
    \frac{1}{H^2} \sum_{|h_1|, |h_2| \le H} \rho_c(T^{h_1} S T^{h_2})
\]
by less than $\min(1,N/\sqrt{c})$. In this form, our task is actually impossible: the matrix above decomposes as a direct sum corresponding to the irreducible representations inside $\rho_c$, one of which is the trivial representation -- and this contributes exactly one singular value of size $1$. Other small-dimensional subrepresentations of $\rho_c$ are also problematic for similar reasons.

This is where the coprimality constraint $(m,n,c)=1$ comes in. Incorporating this weight into $K$ and expanding it by M\"obius inversion essentially results in a `sifted' representation of $\SL_2(\Z/c\Z)$,
\[
    K^\circ := (S(m,n;c) \one_{(m,n,c)=1})_{m, n \le N}
    \qquad 
    \rightsquigarrow
    \qquad 
    \rho_c^\circ,
\]
where $\rho_c^\circ$ is obtained by removing from $\rho_c$ the contribution of all subrepresentations isomorphic to
\[
    \SL_2(\Z/c\Z) \xrightarrow{\text{Reduction mod } d} \SL_2(\Z/d\Z) \xrightarrow{\rho_d} \left\{ \text{Unitary maps of } L^2(\P^1(\Z/d\Z)) \right\},
\]
for $d \mid c$ (when $c$ is a prime, $\rho_c^\circ$ is simply the Steinberg representation).
Although $\rho_c^\circ$ is not irreducible in general, it has the key property that all of its $c^{o(1)}$ irreducible subrepresentations are large, of dimension $\gg c^{1-o(1)}$ (see \cref{prop:sifted-rep}).

\subsection{The key step: Amplification} \label{subsec:outline-amplif}
We are left to bound the spectral norm of the non-abelian Fourier coefficient
\[
    \hat{F}(\rho) = \sum_{g \in \SL_2(\Z/c\Z)} F(g) \rho(g),
    \qquad\quad 
    F := \frac{1}{H^2} \sum_{|h_1|, |h_2| \le H} \one_{T^{h_1}S T^{h_2}} : \SL_2(\Z/c\Z) \to \C,
\]
where $\rho$ is any irreducible subrepresentation of $\rho_c^\circ$.
A natural approach is to use the trace method, i.e., to bound the top singular value $\|\hat{F}(\rho_c^\circ)\|$ by an even moment of all singular values, and then to expand the latter as a trace; this brings in the character $\chi := \Tr\, \rho$.

One can then attempt to use non-abelian Fourier analysis by summing over all irreducible characters $\chi'$ of $\SL_2(\Z/c\Z)$. However, this sum must somehow amplify the contribution of $\chi' = \chi$ compared to other irreducible characters of $\SL_2(\Z/c\Z)$, especially the small-dimensional ones -- otherwise, our construction of $\rho_c^\circ$ by eliminating various subrepresentations from $\rho_c$ will have been useless.

If $\chi$ was an abelian character of $(\Z/c\Z)^\times$, i.e., a Dirichlet character, then following the ideas of Duke--Friedlander--Iwaniec \cite{duke1997bilinear}, one could weigh the sum by an amplifier of the shape
\[
    A_\chi(\chi') := \left\vert \sum_{\ell \in \mL} \bar{\chi'(\ell)} \chi(\ell) \right\vert^2
    =
    \sum_{\ell_1,\ell_2 \in \mL} \bar\chi'(\ell_1\ell_2^{-1}) \chi(\ell_1\ell_2^{-1}),
    \qquad\quad 
    \forall \chi' \in \hat{(\Z/c\Z)^\times},
\]
where $\mL$ is some set of positive integers (e.g., the primes in a dyadic interval). This $A(\chi')$ has size $\approx |\mL|^2$ when $\chi' = \chi$, and should typically obey square-root cancellation when $\chi' \neq \chi$.

Inspired by this, we construct a general amplifier for irreducible representations of a finite non-abelian group $G$. This is to the best of our knowledge the first instance of such a construction, and might find applications to other problems. We set
\begin{equation} \label{eq:sketch-non-abelian-amplifier}
    A_\chi(\chi') := \left\| \sum_{\ell \in \mL} \bar{\rho'(\ell)} \otimes \rho(\ell) \right\|_{S^2}^2
    = \sum_{\ell_1, \ell_2 \in \mL} \bar\chi'(\ell_1 \ell_2^{-1}) \chi(\ell_1 \ell_2^{-1}),
    \qquad\quad
    \forall \rho' \in \hat{G},\
    \chi' := \Tr\, \rho',
\end{equation}
where $\|\cdot\|_{S^2}$ denotes the Frobenius norm of a map (the $\ell^2$ norm of its singular values), and $\mL$ is a well-chosen subset of $G$. It is most convenient to pick $\mL$ to be a normal subgroup of $G$ (with this choice, the amplifier is actually related to the multiplicity of $\rho'$ in a certain induced representation depending on $\rho$; see the proof of \cref{prop:amplification}).
The normal subgroups of $G = \SL_2(\Z/c\Z)$ depend on the factorization of $c$, and we pick
\[
    \mL = \Gamma_c(d) := \ker\left(\SL_2(\Z/c\Z) \to \SL_2(\Z/d\Z) \right),
\]
for a suitable divisor $d$ of $c$. The result of this amplification argument for the sixth moment of singular values is a bound of the shape (see \cref{prop:amplification})
\begin{equation} \label{eq:outline-amplif}
    \|\hat{F}(\rho)\|^6 \lesssim \frac{c^3 H^{-6}}{\sum_{\ell \in \Gamma_c(d)} |\chi(\ell)|^2} \sum_{\substack{|h_1|, \ldots, |h_6| \le H \\ T^{h_1}S \cdots T^{h_6}S \in \Gamma_c(d)}} \chi(T^{h_1}S \cdots T^{h_6}S).
\end{equation}
To go any further, we need to know the typical size of the character $\chi$ on $\Gamma_c(d)$, based on the information that $\dim \chi \gg c^{1-o(1)}$. This is a non-trivial computation involving Clifford theory, and depends on the factorizations of $c$ and $d$; see \cref{lem:squared-char-primitive,lem:perm-rep-pointwise-ub}.

Let us now focus on the case when $c = p^2$ is the square of a prime and $N \approx H \approx p$; we naturally pick $d = p$. It turns out that $\chi$ typically has size $\approx p$ on $\Gamma_{p^2}(p)$, and roughly $\approx p^2$ at $\pm I \in \SL_2(\Z/p^2\Z)$. To beat \cref{eq:trivial-bound}, it essentially remains to bound
\begin{equation} \label{eq:overview-counting-pb}
    \sum_{|h_1|, \ldots, |h_6| \le p} \one_{T^{h_1} S \cdots T^{h_6} S \equiv \pm I \pmod{p^2}} 
    \stackrel{?}{<} p^3,
    \qquad
    \sum_{|h_1|, \ldots, |h_6| \le p} \one_{T^{h_1} S \cdots T^{h_6} S \equiv \pm I \pmod{p}} \stackrel{?}{<} p^4.
\end{equation}

\subsection{Final steps: Combinatorics} \label{subsec:outline-counting}
The estimates in \cref{eq:overview-counting-pb} amount to counting the number of solutions to the system of congruences
\[
    \begin{cases}
        1 - h_2 h_3 \equiv \mp(1 - h_5 h_6) \\
        h_1(1 - h_2 h_3) + h_3 \equiv \pm h_5 \\
        h_4(1 - h_2 h_3) + h_2 \equiv \pm h_6
    \end{cases}
    \pmod{p^2 \text{, respectively, } p},
\]
with $|h_1|, \ldots, |h_6| \le p$. Generically, one can expect each congruence to reduce the total number of solutions $p^6$ by the size of the modulus cubed -- but one must also account for certain diagonal solutions where some $h_i = 0$.
A careful but elementary analysis (which becomes more involved when the modulus $c$ is arbitrary) shows that these congruences have $\approx p^2$ solutions modulo $p^2$ and $\approx p^3$ solutions modulo $p$; see \cref{prop:counting-6}. Both of these counts are sharp, and save a factor of $p$ over the bounds required in \cref{eq:overview-counting-pb}. This saving is ultimately raised to the power $\tfrac{1}{6}$ in \cref{eq:outline-amplif} (since we considered a sixth moment of singular values), and putting everything together yields
\[
    \left\|\left(S(m, n; p^2) \one_{(m,n,p)=1}\right)_{m,n\le p} \right\| \lesssim p^{2-\frac{1}{6}},
\]
as in \cref{ex:p2-pq}. We note that for the other case $c = pq$ from \cref{ex:p2-pq}, one can simplify the amplification argument by noting that all irreducible characters of $\SL_2(\Z/pq\Z)$ are tensor products of irreducible characters of $\SL_2(\Z/p\Z)$ and $\SL_2(\Z/q\Z)$, but the end result is the same.

\subsection{Comments on prime moduli} \label{subsec:comments-prime}

When $c = p$ is a prime and $d = p$, the amplifier from \cref{subsec:outline-amplif} reduces to the `trivial' choice
\[
    A(\chi') = \bar\chi'(I) \chi(I) = \dim \chi' \dim \chi,
\]
since $\mL = \Gamma_c(c) = \{I\}$. In this setting, \cref{eq:outline-amplif} (with $6$ replaced by another even integer $q$) reads
\[
    \|\hat{F}(\rho)\|^q \lesssim p^2 H^{-q} \sum_{|h_1|, \ldots, |h_q| \le H} \one_{T^{h_1}S \cdots T^{h_q}S = I}. 
\]
This was observed, using a somewhat different language, by Shkredov \cite[proofs of Lemmas 22 and 53]{shkredov2018asymptotic}. Shkredov then relied on an $L^2$-flattening lemma \cite[Theorem 50]{shkredov2018asymptotic}, which stems from a result of Helfgott \cite{helfgott2008growth}, to bound the right-hand side above by $O(p^2 H^{-q} H^q p^{-3}) = O(p^{-1})$ for a large value of $q$ depending on $\tfrac{\log p}{\log H}$. This leads to a bound for bilinear (Type II) sums of Kloosterman sums with prime moduli \cite[Theorem 4, (6)]{shkredov2021modular}, with a power saving of $p^{-\delta}$, where $\delta \approx \tfrac{1}{q}$.

To obtain a quantitatively-better power saving for the Type II sums, competitive with \cite{kowalski2017bilinear,kowalski2020stratification}, one must use a smaller value of $q$; one would then need to solve a counting problem with few variables $h_1, \ldots, h_q$, as in \cref{subsec:outline-counting}. We do not know how to do this, but such an approach could in principle produce good results when, e.g., $q \in \{8, 10, 12\}$. In particular, assuming \cref{conj:counting} for $q = 8$, one could prove a non-trivial bound for bilinear sums of Kloosterman sums with prime moduli $p$ and sequences of lengths $M \asymp N > p^{3/8+o(1)}$; interestingly, the same limit at $p^{3/8+o(1)}$ appears in the results of Kowalski--Michel--Sawin \cite{kowalski2020stratification}, so our work reaffirms the difficulty of this barrier.

Alternatively, to obtain non-trivial results at prime moduli, it might be possible to use a different choice of subset $\mL \subset \SL_2(\Z/p\Z)$ in the construction of the amplifier from \cref{subsec:outline-counting}. Indeed, although a normal subgroup is the most natural choice for $\mL$, it is possible that another conjugation-invariant subset might produce a useful amplifier when normal subgroups are not available.

\section{Preliminaries} \label{sec:preliminaries}

\subsection{Analytic and arithmetic notation} \label{subsec:arithm-analytic-not}

We use the standard asymptotic notation from analytic number theory. Expressions of the form $f = O_\eps(g)$ can be substituted with `$|f| \le C_\eps g$ for some constant $C_\eps > 0$ which depends only on $\eps$'; in the absence of a subscript, the implied constant is absolute. We may write $f \ll_\eps g$ for $f = O_\eps(g)$, $f \asymp_\eps g$ for $f \ll_\eps g \ll_\eps f$, and $f = \Omega_\eps(g)$ for $f \gg_\eps g$.

Expressions of the form $o(g(x))$ can be substituted with `$\eta(x)g(x)$ for some bounded complex function $\eta$ with $\lim_{x \to \infty} \eta(x) = 0$'; here the parameter $x$ and its range are implicit. In particular, factors of $x^{o(1)}$ can be read as $x^{\eta(x)}$ for such a function $\eta$. In a bound of the shape $f(x) \ll x^{o(1)}$ for $x \ge 1$, the underlying function $\eta$ must be real for the inequality to hold, and can be taken to be positive without loss of generality. In fact, given fixed functions $f$ and $g$ with $g > 0$, we have the equivalence 
\[
    f(x, y) \ll x^{o(1)} g(x, y)
    \quad
    \iff 
    \quad 
    x^{-o(1)} f(x, y) \ll g(x, y)
    \quad 
    \iff
    \quad
    \forall \eps > 0 : \ 
    f(x, y) \ll_\eps x^\eps g(x, y).
\] 
With this notation, the divisor bound can be written as either $\sum_{d \mid c} 1 = c^{o(1)}$ or $\sum_{d \mid c} 1 \ll c^{o(1)}$. Note that there is no formal difference between $o(1)$ and $-o(1)$, but using signs in bounds like above helps facilitate the translation to $\eps$-statements.

We use the notation $\one_S$ for both indicator functions of sets $S$ and truth values ($0$ or $1$) of statements $S$; we also abbreviate $\one_x := \one_{\{x\}}$ for singletons. We write $n \sim N$ for the range $N < n \le 2N$, $\|\alpha\| := (\sum_n |\alpha_n|^2)^{1/2}$ for the $\ell^2$ norm of a sequence $(\alpha_n)_{n \in \mN}$ for some $\mN \subset \Z$ (or $\mN \subset \Z/c\Z$), and $e(t) := \exp(2\pi i t)$ for $t \in \R/\Z$. Given a positive integer $c$, we let $c\Z$ (resp., $c\Z_+$) be the sets of integers (resp., positive integers) divisible by $c$, and $\bar{x}$ be the inverse of $x$ modulo $c$ (here $c$ may be implied from context, e.g., in an exponential phase $e(\bar{x}/c)$). We use $\mu$ and $\phi$ be the M\"obius and Euler totient functions. Given $a, b \in \Z_+$, we write $(a, b)$ and $[a, b]$ for their greatest common divisor and lowest common multiple (and similarly for more positive integers), and $(a, b^\infty)$ for the greatest divisor of $a$ whose prime factors all divide $b$. We write $p^k \| c$ when a prime power exactly divides a positive integer, meaning that $p^k \mid c$ but $p^{k+1} \nmid c$. We call a positive integer $c$ \emph{square-full} if and only if all exponents in the prime factorization of $c$ are at least $2$ (i.e., there is no prime $p$ such that $p \| c$).

We will reserve the letter $\psi$ for functions on $\Z/c\Z$, and $\Phi, \Psi$ for functions on $\R$. We denote the Fourier transform of an $L^1$ function $\Phi : \R \to \C$ by
\begin{equation} \label{eq:Fourier-tr-R}
    \hat\Phi : \R \to \C,
    \qquad\qquad
    \hat\Phi(\xi) := \int_{-\infty}^\infty \Phi(t)\, e(-t\xi)\, dt.
\end{equation}

In particular, if $\Psi(t) := \Phi(At) e(Bt)$ for some $A > 0$ and $B \in \R$, then a change of variables yields
\begin{equation} \label{eq:scaled-Fourier-tr}
\begin{aligned}
    \hat\Psi(\xi) &= \int_{-\infty}^\infty \Phi(At)\, e(-t(\xi-B))\, dt 
    =
    \frac{1}{A}
    \hat\Phi\left(\frac{\xi-B}{A}\right).
\end{aligned}
\end{equation}
If $\Phi$ is a Schwartz function, then so is $\hat{\Phi}$, and the Poisson summation identity reads
\begin{equation} \label{eq:Poisson}
    \sum_{n \in \Z} \Phi(n) = \sum_{k \in \Z} \hat\Phi(k).
\end{equation}

Given positive integers $m, n$ and a commutative ring $R$, we write $R^{m \times n}$ for the ring of $m \times n$ matrices with entries from $R$. Given a finite set $X$, we denote by $L^2(X)$ the Hilbert space of all complex-valued functions on $X$, equipped with the standard inner product $\la f, g \ra := \sum_{x \in X} f(x) \bar{g(x)}$.
Given a linear map $M$ between finite-dimensional complex Hilbert spaces, we write its operator norm as
\begin{equation} \label{eq:operator-norm}
    \|M\| := \sup_{\|\vec{v}\| = 1} \|M\vec{v}\| = \sup_{\|\vec{v}\| = \|\vec{w}\| = 1} |\vec{w}^T M \vec{v}|.
\end{equation}
For $q \in [1,\infty]$, we also define $\|M\|_{S^q}$ as the $\ell^q$ norm of singular values of $M$. In particular, we have
\begin{equation} \label{eq:schatten}
    \|M\|_{S^\infty} = \|M\| \qquad\quad \text{and} \qquad\quad
    \|M\|_{S^q}^q = \Tr\left((MM^*)^{q/2}\right)^{\frac{1}{q}} \text{ for } q \in 2\Z_+,
\end{equation}
where $M^*$ denotes the adjoint of $M$. The same notation applies to complex matrices $M \in \C^{m \times n}$, viewed as maps $\C^n \to \C^m$.
We record the following fact about projections and operator norms.
\begin{lemma} \label{lem:norm-proj}
Let $V$ be a finite-dimensional complex Hilbert space, $W \subset V$ be a subspace, and $P_W : V \to V$ be the orthogonal projection onto $W$. Suppose that $W$ is an invariant subspace of a linear map $M : V \to V$ (i.e., the restriction $M\vert_W : W \to W$ is well-defined). Then
$\|M\vert_W\| = \|MP_W\|$.
\end{lemma}

\begin{proof}
By definition, $\|M\vert_W\| = \sup_{\vec{w} \in W, \|\vec{w}\| \le 1} \|M \vec{w}\|$ and $\|MP_W\| = \sup_{\vec{v} \in V, \|\vec{v}\| \le 1} \|MP_W \vec{v}\|$.
Since $P_W \vec{v} \in W$ with $\|P_W \vec{v}\| \le \|\vec{v}\| \le 1$ for all $\vec{v} \in V$ with $\|\vec{v}\| \le 1$, we have $\|M P_W\| \le \|M\vert_W\|$. On the other hand, for each $\vec{w} \in W$ with $\|\vec{w}\| \le 1$, we have $P_W \vec{w} = \vec{w}$, so $\|M\vert_W\| \le \|MP_W\|$.
\end{proof}

\subsection{Bounds for Kloosterman sums}
We now recall the Ramanujan and Weil bounds for Kloosterman sums, as well as some results of Kowalski--Michel--Sawin \cite{kowalski2017bilinear} and Blomer--Mili\'cevi\'c \cite{blomer2015second}.

\begin{lemma}[Ramanujan bound] \label{lem:ram}
For $c \in \Z_+$ and $n \in \Z$, one has
\[
    |S(0, n; c)| \le (n, c).
\]
\end{lemma}
\begin{proof}
This is a classical result which follows from M\"obius inversion.
\end{proof}

\begin{lemma}[Weil bound] \label{lem:weil}
For $c \in \Z_+$ and $m, n \in \Z$, one has
\[
    S(m, n; c) \ll c^{o(1)} \sqrt{(m, n, c) c}.
\]
\end{lemma}

\begin{proof}
This is \cite[Corollary 11.12]{iwaniec2021analytic} followed by the divisor bound.
\end{proof}

For the sake of completeness, we give a quick proof of the bounds from \cref{eq:trivial-bound}.

\begin{proof}[Proof of \cref{eq:trivial-bound}]
The second bound implicit in \cref{eq:trivial-bound}, with a term of $\sqrt{MNc}$, follows immediately from \cref{lem:weil} and Cauchy--Schwarz. For the first bound implicit in \cref{eq:trivial-bound}, we eliminate the constraint $(m, n, c) = 1$ by M\"obius inversion and use the identity $S(dm, dn; c) = \tfrac{\phi(c)}{\phi(c/d)} S(m, n; \tfrac{c}{d})$ to write
\[
    \mathop{\sum\sum}_{\substack{m \in \mI, n \in \mJ \\ (m, n, c) = 1}} \alpha_m \beta_n S(am, n; c) 
    \ll
    c^{o(1)} \max_{d \mid c} d \left\vert 
    \sum_{dm \in \mI} \alpha_{dm} \sum_{dn \in \mJ} \beta_{dn} S(am, n; \tfrac{c}{d}) \right\vert.
\]
Now apply Cauchy--Schwarz in the sum over $m$, and complete the sum over $m \pmod{\tfrac{c}{d}}$ to get
\[
    d \left\vert 
    \sum_{dm \in \mI} \alpha_{dm} \sum_{dn \in \mJ} \beta_{dn} S(am, n; \tfrac{c}{d}) \right\vert
    \le
    \|\alpha\| \left(d^2\sum_{m \pmod{\frac{c}{d}}} \left\vert \sum_{dn \in \mJ} \beta_{dn} S(am, n; \tfrac{c}{d}) \right\vert^2 \right)^{\frac{1}{2}}.
\]
Expanding the square and the Kloosterman sums, then performing the sum over $m$, one reaches
\[
    d^2\sum_{m \pmod{\frac{c}{d}}} \left\vert \sum_{dn \in \mJ} \beta_{dn} S(am, n; \tfrac{c}{d}) \right\vert^2
    =
    dc
    \sum_{x \in (\Z/\frac{c}{d}\Z)^\times}
    \left\vert 
    \sum_{dn \in \mJ} \beta_{dn} e\left(\frac{nx}{c/d}\right)\right\vert^2.
\]
Finally, complete the sum over $x \pmod{\tfrac{c}{d}}$, expand the square, and perform the sum over $x$ to obtain
\[
    dc
     \sum_{x \in (\Z/\frac{c}{d}\Z)^\times}
    \left\vert 
    \sum_{dn \in \mJ} \beta_{dn} e\left(\frac{nx}{c/d}\right)\right\vert^2
    \le
    c^2 \|\beta\|^2.
\]
Putting these bounds together completes our proof.
\end{proof}

\begin{theorem}[Kowalski--Michel--Sawin \cite{kowalski2017bilinear}] \label{thm:kms}
Let $p$ be a prime and $M, N \in \Z$ be such that $1 \le N \le M \le p-1$. Then for any complex sequences $(\alpha_m)_{m \le M}$, $(\beta_n)_{n \le N}$ and any $a \in (\Z/p\Z)^\times$, 
\[
    \sum_{m=1}^M \sum_{n=1}^N \alpha_m \beta_n S(am, n; p) 
    \ll
    \|\alpha\| \|\beta\| p^{o(1)} \sqrt{MNp} \left(N^{-\frac{1}{2}} + (MN)^{-\frac{3}{16}} p^{\frac{11}{64}}\right).
\]
\end{theorem}

\begin{proof}
This is
\cite[Theorem 1.1]{kowalski2017bilinear} with $k = 2$ and $M, N$ swapped, except for an additional assumption in loc.\ cit.\ that $p^{1/4} < MN < p^{5/4}$; but this assumption can be removed using the bound \cref{eq:trivial-bound}. Indeed, if $MN \le p^{1/4}$, then
\[
    \sqrt{MNp} \cdot (MN)^{-\frac{3}{16}} \cdot p^{\frac{11}{64}} \ge \sqrt{MNp} \cdot p^{\frac{11}{64} - \frac{3}{64}} 
    >
    \sqrt{MNp},
\]
so the second bound in \cref{eq:trivial-bound} is better. Similarly, if $MN \ge p^{5/4}$, then
\[
    \sqrt{MNp} \cdot (MN)^{-\frac{3}{16}} \cdot p^{\frac{11}{64}} \ge 
    p^{\frac{5}{4}(\frac{1}{2}-\frac{3}{16})} \cdot p^{\frac{1}{2}+\frac{11}{64}}
    =
    p^{\frac{25}{64} + \frac{43}{64}} > p,
\]
so the first bound in \cref{eq:trivial-bound} is better.
\end{proof}

\begin{theorem}[Blomer--Mili\'cevi\'c \cite{blomer2015second}] \label{thm:bm}
Let $c, d, M, N \in \Z_+$ such that $d \mid c$ and $d$ is odd. Then for any complex sequences $(\alpha_m)_{m \le M}$ and $(\beta_n)_{n \le N}$ such that $|\alpha_m| \le 1$ for all $m$, and any $a \in (\Z/c\Z)^\times$, one has
\[
    \mathop{\sum_{m=1}^M\sum_{n=1}^N}_{(n,c)=1} \alpha_m \beta_n S(am, n; c)
    \ll 
    \sqrt{M}\|\beta\| (MNc)^{\frac{1}{2} + o(1)}
    \left(\frac{c^{1/2}}{d^{1/2}M^{1/2}} + \frac{1}{d^{1/4}} + \frac{d^{1/4}}{N^{1/2}}\right).
\]
\end{theorem}
\begin{proof}
Dyadically summing instances of \cite[Theorem 5]{blomer2015second} with $(q, r, s, M, K, \lambda(k))$ in loc.\ cit.\ replaced by $(c, c, \tfrac{c}{d}, N, M, \alpha_m)$, one obtains the bound\footnote{\cite[Theorem 5]{blomer2015second} does not include an $a$-scalar inside the Kloosterman sum, but it holds in this slightly more general form with the same proof, and it is in fact applied this way in \cite[p.\,471, after (4.2)]{blomer2015second}.}
\[
    \sum_{\substack{n \le N \\ (n, c) = 1}} \left\vert \sum_{m \le M} \alpha_m S(am, n; c) \right\vert^2
    \ll 
    (cMN)^{o(1)} M^2 N c \left(\frac{c}{d M} + \frac{1}{\sqrt{d}} + \frac{\sqrt{d}}{N}\right).
\]
The desired bound now follows from Cauchy--Schwarz in the shape
\[
    \left\vert \mathop{\sum_{m=1}^M\sum_{n=1}^N}_{(n,c)=1} \alpha_m \beta_n S(am, n; c) \right\vert^2
    \le
    \|\beta\|^2
    \sum_{\substack{n \le N \\ (n, c) = 1}}
    \left\vert \sum_{m \le M} \alpha_m S(am, n; c) \right\vert^2.
\]
(Since $(\beta_n)$ can be chosen to attain equality in this Cauchy--Schwarz step, \cref{thm:bm} is in fact a restatement of \cite[Theorem 5]{blomer2015second}.)
\end{proof}

\subsection{Fourier analysis on finite groups}

Here we recall some general facts and notation from representation theory of finite groups; we point the reader to \cite{serre1977linear,fulton1991representation,terras1999fourier,isaacs2006character} for more background.
Let $G$ be a finite group with identity element $e$. A (unitary) representation of $G$ is a homomorphism 
\[
    \rho : G \to U(V),
\]
where $V$ is a finite-dimensional complex Hilbert space and $U(V)$ is the set of unitary transformations of $V$. In particular, $\rho(e) = \Id_V$ is the identity transformation on $V$. We write
\[
    \dim \rho := \dim V
\]
for the dimension of $\rho$. We say that two representations $\rho_1 : G \to U(V_1)$, $\rho_2 : G \to U(V_2)$ are \emph{isomorphic}, written $\rho_1 \cong \rho_2$, iff there is an invertible linear map $M : V_1 \to V_2$ such that $M \circ \rho_1(g) = \rho_2(g) \circ M$ for all $g \in G$. Since we work with unitary representations, this actually implies that there is a unitary map $U : V_1 \to V_2$ such that $U \circ \rho_1(g) = \rho_2(g) \circ U$ for all $g \in G$.

\begin{example}
We write $\rzero : G \to U(\{0\})$ for the \emph{zero representation} given by $\rzero(g) = \Id_{\{0\}}\ \forall g \in G$, and $\rone : G \to U(\C)$ for the \emph{trivial representation} given by $\rone(g) = \Id_\C\ \forall g \in G$.
Any action of $G$ on a finite set $X$ gives rise to a \emph{permutation representation} $\rho : G \to U(L^2(X))$, defined by $(\rho(g)f)(x) := f(g^{-1}x)$ for $g \in G$, $x \in X$. The \emph{regular representation} $R_G$ is the permutation representation induced by the action by left-multiplication on $X = G$, so $\dim R_G = |G|$.
\end{example} 

Given two representations $\rho_1 : G \to U(V_1)$ and $\rho_2 : G \to U(V_2)$, we write $\rho_1 \oplus \rho_2 : G \to U(V_1 \oplus V_2)$ and $\rho_1 \otimes \rho_2 : G \to U(V_1 \otimes V_2)$ for their direct sum and tensor product. The operations $\oplus$ and $\otimes$ have identity elements $\rzero$ and $\rone$ respectively (up to isomorphism). Given $\rho : G \to U(V)$, we write
\[
    \rho^{\oplus m} := \underbrace{\rho \oplus \cdots \oplus \rho}_{m \text{ times}}
\]
for all nonnegative integers $m$; when $m = 0$, we interpret this as the zero representation $\rzero$. We use a similar notation for repeated direct sums of linear maps or matrices.

An \emph{invariant subspace} $W$ of a representation $\rho : G \to U(V)$ is a subspace of $V$ such that $\rho(g)W \subset W$ for all $g \in G$. For such $W$, we define $\rho\vert_W : G \to U(W)$ by $\rho\vert_W(g) := \rho(g)$ for all $g \in G$, which is automatically unitary, and we say that $\rho\vert_W$ is a \emph{subrepresentation} of $\rho$. One can decompose $\rho \cong \rho_W \oplus \rho_{W^\perp}$; conversely, if $\rho \cong \rho_1 \oplus \rho_2$, then $\rho_1$ and $\rho_2$ are isomorphic to subrepresentations of $\rho$. An important class of invariant subspaces of a representation $\rho : G \to U(V)$ are the fixed-point spaces associated to the normal subgroups of $G$,
\[
    V^N := \left\{ v \in V : \rho(n)v = v,\ \forall n \in N \right\},
    \qquad\qquad 
    \text{for } N \triangleleft G.
\]
\begin{lemma} \label{lem:orth-proj-formula}
Given a representation $\rho : G \to U(V)$ and a normal subgroup $N \triangleleft G$, the orthogonal projection $P : V \to V$ onto $V^N$ can be expressed as
\[
    P = \frac{1}{|N|} \sum_{n \in N} \rho(n).
\]
Moreover, $P$ commutes with $\rho(g)$ for any $g \in G$.
\end{lemma}

\begin{proof}
Let $T : V \to V$ be given by $T := \frac{1}{|N|} \sum_{n \in N} \rho(n)$; we will show that $T = P$. The fact that $N$ is a subgroup quickly implies that $T$ is self-adjoint and that $T^2 = T$, so $T$ is an orthogonal projection. Moreover, one has $\rho(n)T = T$ for any $n \in N$, so any $Tv \in T(V)$ has $\rho(n)Tv = Tv$, which shows $T(V) \subset V^N$. Conversely, if $v \in V^N$, so $\rho(n)v = v$ for all $n \in N$, then clearly $v = Tv$, which shows $V^N \subset T(V)$. Thus $T = P$ is the orthogonal projection onto $V^N$. The claim about commutativity follows immediately from the formula and the normality of $N$.
\end{proof}

We say that a representation of $G$ is \emph{irreducible} iff it is nonzero and has no nonzero subrepresentation other than itself. We write $\hat{G}$ for a complete set of irreducible representations of $G$ up to isomorphism, which always includes the trivial representation $\rone$. If $G$ is abelian, then all irreducible representations in $\hat{G}$ are $1$-dimensional and form a group (isomorphic to $G$) under tensor product.

Any representation $\rho$ of $G$ has a decomposition into irreducible representations of the shape
\begin{equation} \label{eq:rep-decomposition}
    \rho \cong \bigoplus_{\rho' \in \hat{G}} {\rho'} ^{\, \oplus \Mult(\rho', \rho)},
\end{equation}
where the \emph{multiplicities} $\Mult(\rho', \rho)$ are uniquely determined. In particular, $\Mult(\rho', R_G) = \dim \rho'$.

Given two finite groups $G_1, G_2$ and representations $\rho_1 : G_1 \to U(V_1)$ and $\rho_2 : G_2 \to U(V_2)$, we write $\rho_1 \boxtimes \rho_2 : G_1 \times G_2 \to U(V_1 \otimes V_2)$ for the representation of $G_1 \times G_2$ given by
\[
    (\rho_1 \boxtimes \rho_2)(g_1, g_2) := \rho_1(g_1) \otimes \rho_2(g_2),
    \qquad\quad 
    g_1 \in G_1,\ g_2 \in G_2.
\]
The elements of $\hat{G_1 \times G_2}$ are (up to isomorphism) precisely those of the form $\rho_1 \boxtimes \rho_2$ where $\rho_1 \in \hat G_1$ and $\rho_2 \in \hat G_2$ \cite[\S 3.2]{serre1977linear}.

\begin{notation} \label{not:box-product} 
If $G_1, G_2, \rho_1, \rho_2$ are as above, and $G_{1,2}$ is a group isomorphic to $G_1 \times G_2$ by a fixed implicit map (such as \cref{eq:CRT-isom-groups}), we also use the notation $\rho_1 \boxtimes \rho_2$ to describe representations of $G_{1,2}$.
\end{notation}

A \emph{character} $\chi : G \to \C$ is any function of the form $\chi(g) = \Tr\, \rho(g)$, where $\rho$ is a representation of $G$; note that characters are constant on conjugacy classes, that $\chi(e) = \dim \rho$ and $\chi(g^{-1}) = \bar\chi(g)$, and that isomorphic representations induce the same character. If $\rho_1, \rho_2$ are two representations of $G$ with characters $\chi_1, \chi_2$, then $\Tr(\rho_1 \oplus \rho_2) = \chi_1 + \chi_2$ and $\Tr(\rho_1 \otimes \rho_2) = \chi_1\chi_2$. If $\rho_1, \rho_2$ are representations of $G_1, G_2$ with characters $\chi_1, \chi_2$ (respectively), then $\Tr(\rho_1 \boxtimes \rho_2)(g_1,g_2) = \chi_1(g_1) \chi_2(g_2)$. We say that $\chi$ is irreducible if and only if $\rho$ is, and write $\Irr(G)$ for the set of all irreducible characters of $G$. The character table of $G$ satisfies the following orthogonality relations.
\begin{lemma}[Character orthogonality]
One has
\begin{align} \label{eq:ortho-chars-sum-g}
    \sum_{g \in G} \chi_1(g) \bar\chi_2(g) &= |G| \one_{\chi_1 = \chi_2},
    \qquad\qquad 
    \chi_1, \chi_2 \in \Irr(G), 
    \\
    \label{eq:ortho-chars-sum-chi}
    \sum_{\chi \in \Irr(G)} \chi(g_1) \bar\chi(g_2) &= 
    \begin{cases} 
    \frac{|G|}{|C|}, & g_1, g_2 \text{ belong to the same conjugacy class $C$ of $G$,} \\ 
    0, & g_1, g_2 \in G \text{ are not conjugate.}
    \end{cases}
\end{align}
\end{lemma}
\begin{proof}
The first relation is \cite[Theorem 2.12]{fulton1991representation}. Since characters are constant on conjugacy classes, we may pick a system of representatives $\{g_C\}$ for the conjugacy classes of $G$ and write \cref{eq:ortho-chars-sum-g} as
\[
    \sum_{\substack{C \text{ conjugacy} \\ \text{ class of $G$}}} \frac{|C|}{|G|} \chi_1(g_C) \bar\chi_2(g_C) = \one_{\chi_1 = \chi_2},
    \qquad\qquad 
    \chi_1,\chi_2 \in \Irr(G).
\]
By \cite[Proposition 2.30]{fulton1991representation}, the number of conjugacy classes of $G$ equals $|\Irr(G)| = |\hat{G}|$, so the values $\sqrt{|C|/|G|}\chi(g_C)$ form a unitary matrix indexed by $\chi, C$. The orthonormality of the columns of this matrix is precisely \cref{eq:ortho-chars-sum-chi}.
\end{proof}

It follows from \cref{eq:ortho-chars-sum-g,eq:rep-decomposition} that for any representations $\rho, \rho'$ of $G$ where $\rho'$ is irreducible, with $\chi = \Tr\, \rho$ and $\chi' = \Tr\, \rho'$, one has
\begin{equation} \label{eq:multiplicity-as-sum}
    \frac{1}{|G|} \sum_{g \in G} \chi(g) \bar\chi'(g) = \Mult(\rho', \rho).
\end{equation}
Moreover, by summing over $\rho'$ with weights $\Mult(\rho', \rho)$, one has
\begin{equation} \label{eq:sum-square-char}
    \frac{1}{|G|} \sum_{g \in G} |\chi(g)|^2 = \sum_{\rho' \in \hat{G}} \Mult(\rho', \rho)^2.
\end{equation}

We may restrict a representation $\rho : G \to U(V)$ and its character $\chi = \Tr\, \rho$ to a subgroup $H \le G$, to obtain a representation of $\rho\vert_H : H \to U(V)$ with character $\chi\vert_H = \Tr\rho\vert_H$. If $\rho$ is irreducible, $\rho\vert_H$ is not necessarily irreducible. When $H = N$ is a normal subgroup, the structure of $\rho\vert_N$ can be better understood using Clifford theory \cite{clifford1937representations}. We recall that for $N \triangleleft G$, $G$ acts on $\hat{N}$ by conjugation,
\begin{equation} \label{eq:G-conjugation-N-hat}
    G \times \hat{N} \ni (g, \sigma) \mapsto g \cdot \sigma \in \hat{N},
    \qquad\quad
    (g \cdot \sigma)(n) := \sigma(g^{-1}ng)\ \ \forall n \in N.
\end{equation}

\begin{lemma}[Clifford] \label{lem:clifford}
Let $G$ be a group, $N \triangleleft G$ be a normal subgroup, and $\rho \in \hat{G}$ be an irreducible representation. Then there exist positive integers $L, m, d$ with $\dim \rho = Lmd$, and non-isomorphic irreducible representations $\sigma_1, \ldots, \sigma_L \in \hat{N}$ of dimension $d$, such that
\[
    \rho\vert_N \cong \bigoplus_{\ell = 1}^L \sigma_\ell^{\oplus m}.
\]
Moreover, $\{\sigma_1, \ldots, \sigma_L\}$ form an orbit of the action \cref{eq:G-conjugation-N-hat} of $G$ by conjugation on $\hat{N}$.
\end{lemma}

\begin{proof}
See, e.g., \cite[Theorems 6.2 and 6.5]{isaacs2006character}.
\end{proof}

Conversely, given a subgroup $H \le G$ and a representation $\rho : H \to U(W)$ of $H$, we can construct an \emph{induced representation}
\[
    \Ind_H^G(\rho) : G \to U(V),
\]
which acts by translation on the space
\[
    V = \left\{ \text{Functions } \phi : G \to W \text{ such that } \phi(gh^{-1}) = \rho(h) \phi(g),\ \forall h \in H, g \in G\right\}.
\]
If $\chi = \Tr\, \rho$, then the character of the induced representation is given by \cite[(3.18)]{fulton1991representation}
\begin{equation} \label{eq:char-ind}
    \Ind_H^G(\chi)(g) := \Tr\, \Ind_H^G(\rho)(g) = 
    \frac{1}{|H|} \sum_{\substack{x \in G \\ x^{-1}g x \in H}} \chi(x^{-1} g x).
\end{equation}
\begin{lemma}[Frobenius reciprocity] \label{lem:frob-rec}
Let $G, H, \rho$ be as above, and $\psi$ be any character of $G$. Then
\[
    \frac{1}{|G|} \sum_{g \in G} \Ind_H^G(\chi)(g) \cdot \psi(g) =
    \frac{1}{|H|} \sum_{h \in H} \chi(h) \cdot \psi(h).
\]
\end{lemma}
\begin{proof}
See, e.g., \cite[Corollary 3.20]{fulton1991representation}.
\end{proof}

Given a function $F : G \to \C$ and a (not necessarily irreducible) representation $\rho : G \to U(V)$, we define the \emph{Fourier coefficient} $\hat{F}(\rho) : V \to V$ by
\begin{equation} \label{eq:Fourier-tr-nonab}
    \hat{F}(\rho) := \sum_{g \in G} F(g) \rho(g).
\end{equation}
This obeys $\hat{F_1 * F_2}(\rho) = \hat{F_1}(\rho) \hat{F_2}(\rho)$, where $(F_1 * F_2)(g) := \sum_{g_1 g_2 = g} F_1(g) F_2(g)$ denotes the convolution of two functions $F_1, F_2 : G \to \C$. In particular, if $G = \Z/c\Z$, the irreducible representations of $G$ are all $1$-dimensional and of the shape $\rho_a(g) := e(\tfrac{ag}{c})$ for $a, g \in \Z/c\Z$. In this case, we write 
\begin{equation} \label{eq:Fourier-tr-ab}
    \hat{F}(a) := \hat{F}(\rho_{-a}) = \sum_{g \in \Z/c\Z} F(g)\, e\left(-\frac{ag}{c}\right).
\end{equation}

\begin{lemma} \label{lem:break-down-schatten}
Let $F : G \to \C$, $\rho : G \to U(V)$ be a representation, and $q \in [1, \infty)$. Then one has
\[
    \|\hat{F}(\rho)\|_{S^q}^q = \sum_{\rho' \in \hat{G}} \Mult(\rho', 
    \rho) \|\hat{F}(\rho')\|_{S^q}^q,
    \qquad\qquad 
    \|\hat{F}(\rho)\| = \max_{\substack{\rho' \in \hat{G} \\ \Mult(\rho',\rho)>0}} \|\hat{F}(\rho')\|.
\]
\end{lemma}

\begin{proof}
By \cref{eq:rep-decomposition}, there exists a unitary map $U$ (from $V$ to the direct sum of $\rho$'s irreducible invariant subspaces) such that for any $g \in G$,
\[
    U\rho(g)U^* = \bigoplus_{\rho' \in \hat{G}} {\rho'(g)\oplus}^{\Mult(\rho',\rho)}.
\]
But then, by \cref{eq:Fourier-tr-nonab}, we have
\[
\begin{aligned}
    U\hat{F}(\rho)U^* = \sum_{g \in G} F(g) U\rho(g)U^*
    &=
    \sum_{g \in G} F(g) \bigoplus_{\rho' \in \hat{G}} \rho'(g)^{\oplus \Mult(\rho',\rho)}
    \\
    &=
    \bigoplus_{\rho' \in \hat{G}} \left(\sum_{g \in G} F(g) \rho'(g)\right)^{\oplus \Mult(\rho',\rho)}
    =
    \bigoplus_{\rho' \in \hat{G}} \hat{F}(\rho') ^{\oplus \Mult(\rho',\rho)},
\end{aligned}
\]
and the conclusion follows from the fact that the multiset of singular values of a direct sum of matrices is the union of the multisets of singular values of those matrices.
\end{proof}

\subsection{Standard facts about \texorpdfstring{$\SL_2$}{SL2}}
Let $c \in \Z_+$. Recall the special linear groups $\SL_2(\Z)$ and $\SL_2(\Z/c\Z)$ of matrices in $\Z^{2 \times 2}$ (resp., $(\Z/c\Z)^{2 \times 2}$) with determinant $1$, and the projective special linear groups,
\begin{equation} \label{eq:def-PSL2}
    \PSL_2(\Z) := \SL_2(\Z)/Z(\SL_2(\Z)), \qquad \PSL_2(\Z/c\Z) := \SL_2(\Z/c\Z)/Z(\SL_2(\Z/c\Z)),
\end{equation}
where $Z(G)$ denotes the center of a group $G$. The centers here are explicitly given by 
\begin{equation} \label{eq:centers} 
    Z(\SL_2(\Z)) = \{\pm I\}, \qquad\qquad Z(\SL_2(\Z/c\Z)) = \{\gamma I : \gamma \in \Z/c\Z, \gamma^2 = 1\},
\end{equation}
where 
\begin{equation} \label{eq:center-bound}
    |Z(\SL_2(\Z/c\Z))| \ll c^{o(1)},
\end{equation}
by reducing to a local computation.
When the group $\SL_2(\Z)$, $\PSL_2(\Z)$, $\SL_2(\Z/c\Z)$ or $\PSL_2(\Z/c\Z)$ is understood from context, we write
\begin{equation} \label{eq:SL2-generators}
    I := \begin{pmatrix} 1 & 0 \\ 0 & 1 \end{pmatrix},
    \qquad
    T := \begin{pmatrix} 1 & 1 \\ 0 & 1 \end{pmatrix},
    \qquad
    S := \begin{pmatrix} 0 & -1 \\ 1 & 0 \end{pmatrix},
\end{equation}
which satisfy the relations $-S^2 = -(ST)^3 = I$, and in the case of $\SL_2(\Z/c\Z)$ or $\PSL_2(\Z/c\Z)$, also $T^c = I$. Note that $T$ and $S$ generate $\SL_2(\Z)$. 

\begin{notation}[Projective line] \label{not:proj-line}
For $c \in \Z_+$, we recall the projective line
\[
    \P^1(\Z/c\Z) := \left\{(x, y) \in (\Z/c\Z)^2: \not\exists\, d > 1 \text{ s.t. } (x, y) \in (d\Z/c\Z)^2 \right\}/_\sim,
\]
where $\sim$ is the equivalence relation generated by $(x, y) \sim (\alpha x, \alpha y)$ for $\alpha \in (\Z/c\Z)^\times$. We write the equivalence class of $(x, y)$ as $[x : y]$, and we will typically use the letters $u, v$ to denote projective points in $\P^1(\Z/c\Z)$, reserving $x, y$ for elements of $\Z/c\Z$. Note that for any $[x : y] \in \P^1(\Z/c\Z)$, there exist $a, b \in \Z/c\Z$ such that $ax + by \equiv 1 \pmod{c}$.
For $d \mid c$, we write the natural map $\P^1(\Z/c\Z) \to \P^1(\Z/d\Z)$ which reduces both entries modulo $d$ as $u \mapsto u \mod d$. 
\end{notation}

The group $\PSL_2(\Z/c\Z)$ (and, through it, $\SL_2(\Z/c\Z)$) acts on $\P^1(\Z/c\Z)$ by 
\begin{equation} \label{eq:action-proj}
    \begin{pmatrix}
        m & n \\ p & q
    \end{pmatrix}
    [x : y]
    :=
    [mx + ny : px + qy].
\end{equation}
One can think of $\P^1(\Z/c\Z)$ as $\Z/c\Z$ with a few additional `points at infinity', which must be included to obtain a well-defined action of $\SL_2(\Z/c\Z)$. 
In particular, one can embed $\Z/c\Z \subset \P^1(\Z/c\Z)$ by $x \mapsto [x : 1]$, and via this embedding, the generators from \cref{eq:SL2-generators} act on elements of $\Z/c\Z$ by
\[
    Tx = x+1, \qquad\qquad 
    Sy = -\bar{y}, \qquad\qquad 
    \text{for } x \in \Z/c\Z,\ y \in (\Z/c\Z)^\times.
\]
We now briefly go over a few well-known facts about the subgroups and representations of $\SL_2(\Z/c\Z)$.

\begin{notation}[Reduction mod $d$] \label{not:red-mod-d}
Given a positive integer $d$ with $d \mid c$, we denote by
\[
    \pi_{c, d} : \SL_2(\Z/c\Z) \to \SL_2(\Z/d\Z)
\]
the natural epimorphism which `reads' the entries of $g \in \SL_2(\Z/c\Z)$ modulo $d$. We write
\[
    \Gamma_c(d) := \ker \pi_{c, d}
\]
for the congruence subgroup given by the kernel of this map (consisting of matrices of the form $I + dA$, where one may view the entries of $A$ as elements of $\Z/\tfrac{c}{d}\Z$).
\end{notation}

\begin{lemma} \label{lem:action-transitive}
$\SL_2(\Z/c\Z)$ acts transitively on $\P^1(\Z/c\Z)$ (i.e., there is only one orbit). In fact, for $d \mid c$, there is a bijection between $\P^1(\Z/d\Z)$ and the orbits of $\P^1(\Z/c\Z)$ under $\Gamma_c(d)$,
\begin{equation} \label{eq:orbits-bijection}
\begin{array}{rcl}
\Gamma_c(d) \backslash \P^1(\Z/c\Z) &\longrightarrow& \P^1(\Z/d\Z), \\
\Gamma_c(d) \cdot u &\longmapsto& u \mod d,
\end{array}
\end{equation}
and all orbits in $\Gamma_c(d) \backslash \P^1(\Z/c\Z)$ have size $|\P^1(\Z/c\Z)|/|\P^1(\Z/d\Z)|$.
\end{lemma}

\begin{proof}
For any $[x : y] \in \P^1(\Z/c\Z)$, there exist $a, b \in \Z/c\Z$ with $ax + by \equiv 1 \pmod{c}$, so $[x : y] = \begin{psmall} x & -b \\ y & a \end{psmall} [1 : 0] \in \SL_2(\Z/c\Z) \cdot [1 : 0]$. Thus the action of $\SL_2(\Z/c\Z)$ on $\P^1(\Z/c\Z)$ is transitive.

The map in \cref{eq:orbits-bijection} is well-defined since $(I+dA)u \pmod{d} = u \pmod{d}$ for any $I + dA \in \Gamma_c(d)$. It is surjective since the original map $\P^1(\Z/c\Z) \to \P^1(\Z/d\Z)$ is surjective. To show that \cref{eq:orbits-bijection} is also injective, suppose $u \pmod{d} = v \pmod{d}$ for some $u, v \in \P^1(\Z/c\Z)$, and we aim to show that $\Gamma_c(d) \cdot u = \Gamma_c(d) \cdot v$. By the transitivity of the action of $\SL_2(\Z/c\Z)$, we can find $g \in \SL_2(\Z/c\Z)$ such that $gv = [1 : 0] \in \P^1(\Z/c\Z)$, so
\[
    (gu) \mod d = (gv) \mod d = [1 : 0] \in \P^1(\Z/d\Z).
\]
Write $gu = [xd+1 : yd]$ for some $x, y \in \Z/c\Z$. Since $gu \in \P^1(\Z/c\Z)$, we have $1 = (xd+1, yd, c) = (xd+1, yd^2, c)$, so there exist $a, b \in \Z/c\Z$ with $a(xd+1) + byd^2 \equiv 1 \pmod{c}$, and in particular $a \equiv 1 \pmod{d}$. Then,
\[
    gu = \begin{pmatrix}
        xd+1 & -bd \\ 
        yd & a 
    \end{pmatrix}
    [1 : 0]
    \in 
    \Gamma_c(d) \cdot gv
    =
    g\Gamma_c(d) \cdot v,
\]
where the last equality is due to the normality of $\Gamma_c(d)$. Hence $u \in \Gamma_c(d) \cdot v$, as we wanted.

Finally, all orbits in $\Gamma_c(d)\backslash\P^1(\Z/c\Z)$ have the same size due to the normality of $\Gamma_c(d)$ (which again implies $|\Gamma_c(d) \cdot gu| = |g \Gamma_c(d) \cdot u| = |\Gamma_c(d) \cdot u|$ for all $u \in \P^1(\Z/c\Z)$ and $g \in \SL_2(\Z/c\Z)$) and the transitivity of the action of $\SL_2(\Z/c\Z)$. But there are $|\P^1(\Z/d\Z)|$ such orbits due to the bijection in \cref{eq:orbits-bijection}, so each orbit must have size $|\P^1(\Z/c\Z)|/|\P^1(\Z/d\Z)|$.
\end{proof}

\begin{notation}[Specified isomorphisms] \label{not:specified-isomorphisms}
Recall that the Chinese remainder theorem gives a standard isomorphism of rings $\Z/c\Z \cong \prod_{p^k\| c} \Z/p^k\Z$.
To describe certain sets and groups depending on $c$ in terms of the prime factorization of $c$, we use the notation `$\cong$' to refer to the specific bijections and isomorphisms that are compatible with the Chinese remainder theorem. In particular, we have a bijection of sets
\begin{equation} \label{eq:CRT-bij-P1}
    \P^1(\Z/c\Z) \cong \prod_{p^k \| c} \P^1(\Z/p^k\Z),
\end{equation}
and group isomorphisms
\begin{equation} \label{eq:CRT-isom-groups}
    \SL_2(\Z/c\Z) \cong \prod_{p^k \| c} \SL_2(\Z/p^k\Z),
    \qquad\qquad 
    \Gamma_c(d) \cong \prod_{\substack{p^k \| c \\ p^j \| d}} \Gamma_{p^k}(p^j),
\end{equation}
for $d \mid c$ (in the products above, it is understood that only primes which divide $c$ are included, so $k \ge 1$, but we allow $j = 0$). The isomorphisms in \cref{eq:CRT-isom-groups} also correspond to combining the maps $\pi_{c,p^k}$ from \cref{not:red-mod-d} for $p^k \| c$.
\end{notation}

Note that for a prime power $p^k$ with $k \ge 1$, each point in $\P^1(\Z/p^k\Z)$ can be written uniquely as either $[x : 1]$ with $x \in \Z/p^k\Z$ or as $[1 : y]$ with $y \in p\Z/p^k\Z$; thus $|\P^1(\Z/p^k\Z)| = p^k + p^{k-1}$. It follows from \cref{eq:CRT-bij-P1} that
\begin{equation} \label{eq:P1-size}
    |\P^1(\Z/c\Z)| = c\prod_{\text{prime } p \mid c} \left(1 + \frac{1}{p}\right) \ll c^{1+o(1)}.
\end{equation}
Similaly, since $|\SL_2(\Z/p^k\Z)| = p^{3k}(1 - \tfrac{1}{p^2})$ for $k \ge 1$, it follows from \cref{eq:CRT-isom-groups} that
\begin{equation} \label{eq:size-nd}
    |\SL_2(\Z/c\Z)| = c^3 \prod_{\text{prime } p | c} \left(1 - \frac{1}{p^2}\right)  \asymp c^3
    \qquad \Rightarrow 
    \qquad 
    |\Gamma_c(d)| = \frac{|\SL_2(\Z/c\Z)|}{|\SL_2(\Z/d\Z)|} \asymp \frac{c^3}{d^3},
\end{equation}
and that the irreducible representations of $\SL_2(\Z/c\Z)$ can be parametrized as
\begin{equation} \label{eq:CRT-irrep}
    \hat\SL_2(\Z/c\Z) = \left\{ \bigboxtimes_{p^k \| c} \rho_{p,k} :
    \rho_{p,k} \in \hat \SL_2(\Z/p^k\Z) \right\}.
\end{equation}

Now let $p$ be a prime and $k \in \Z_+$, and let us focus on understanding the structure of $\SL_2(\Z/p^k\Z)$. 

\begin{lemma} \label{lem:abelian-subgroups}
For any integer $j \in [\tfrac{k}{2}, k]$, the normal subgroup $N := \Gamma_{p^k}(p^j)$ is abelian, so all of its irreducible representations are $1$-dimensional. Moreover, writing $R := \Z/p^{k-j}\Z$ and $RI = \{rI : r \in R\}$, one has group isomorphisms
\[
    \left(\left\{ A \in R^{2 \times 2} : \Tr(A) = 0\right\}, +\right)
    \cong (N, \cdot),
    \qquad\qquad 
    (R^{2\times 2} / RI, +)
    \cong 
    (\hat{N}, \cdot),
\]
given explicitly by the maps $A \mapsto I + p^j A$ and $B+RI \mapsto \sigma_B$, where 
\begin{equation} \label{eq:sigma-b}
    \sigma_B(I + p^j A) := e\left(\frac{\Tr(AB)}{p^{k-j}}\right),
    \qquad 
    \text{ for }
    A, B \in R^{2 \times 2},\ \Tr(A) = 0.
\end{equation}
These isomorphisms preserve the action of $\SL_2(\Z/p^k\Z)$ by conjugation, recalling \cref{eq:G-conjugation-N-hat}. In particular, for $B \in R^{2 \times 2}$ and $g \in \SL_2(\Z/p^k\Z)$, we have $\sigma_{gBg^{-1}} = g \cdot \sigma_B$.
\end{lemma}

\begin{remark}
If $p$ is odd, there is also an isomorphism $\left\{ A \in R^{2 \times 2} : \Tr(A) = 0\right\} \cong R^{2\times 2} / RI$ by $A \mapsto A + RI$, which preserves the action of $\SL_2(\Z/p^k\Z)$ by conjugation. When $p = 2$ and $k-j \ge 1$, the same map fails to be injective (consider $A + 2^{k-j-1}I$) or surjective (consider odd-trace matrices).
\end{remark}

\begin{proof}[Proof of \cref{lem:abelian-subgroups}]
The isomorphism $(N, \cdot) \cong \left(\left\{ A \in R^{2 \times 2} : \Tr(A) = 0\right\}, +\right)$ follows by noting that
\[
\begin{aligned}
    N &= \left\{ I + p^j A : A \in (\Z/p^{k-j}\Z)^{2 \times 2}, \ \det(I + p^j A) \equiv 1 \pmod{p^k}\right\}
    \\
    &= 
    \left\{ I + p^j A : A \in R^{2 \times 2},\ \Tr(A) = 0\right\},
\end{aligned}
\]
and that $(I + p^j A) (I + p^j B) = I + p^j (A + B)$ in $\SL_2(\Z/p^k\Z)$, for $j \ge \tfrac{k}{2}$. The compatibility of this isomorphism with $\SL_2(\Z/p^k)$-conjugation is immediate from
\[
    g(I + p^jA)g^{-1} = I + p^j g A g^{-1},
\]
for $A \in R^{2 \times 2}$ with $\Tr(A) = 0$ and $g \in \SL_2(\Z/p^k\Z)$.

Now for $B \in R^{2 \times 2}$, we have a homomorphism $\{A \in R^{2 \times 2} : \Tr(A) = 0\} \to \{z \in \C : |z| = 1\}$ by
\[
    A \mapsto e\left(\frac{\Tr(AB)}{p^{k-j}}\right).
\]
This only depends on the class $B + RI \in R^{2 \times 2}/RI$, since $\Tr(A(B+rI)) = \Tr(AB)$ for $r \in R$. Moreover, two different classes $B+RI, B'+RI \in R^{2 \times 2}/RI$ are seen to induce different homomorphisms by considering $A \in \{\begin{psmall} 0 & 1 \\ 0 & 0\end{psmall}, \begin{psmall} 0 & 0 \\ 1 & 0\end{psmall}, 
\begin{psmall} 1 & 0 \\ 0 & -1\end{psmall}
\}$. Since $|\{A \in R^{2 \times 2} : \Tr(A) = 0\}| = |R|^3 = |R^{2 \times 2}/RI|$, all homomorphisms $\{A \in R^{2 \times 2} : \Tr(A) = 0\} \to \{z \in \C : |z| = 1\}$ arise this way.
The Pontryagin dual of $\left\{A \in R^{2 \times 2},\ \Tr(A) = 0\right\}$ is therefore naturally identified with $R^{2 \times 2}/RI$.
Combining this with the isomorphism $\left(\left\{ A \in R^{2 \times 2} : \Tr(A) = 0\right\}, +\right) \cong (N, \cdot)$ by $A \mapsto I + p^jA$ leads to \cref{eq:sigma-b}. 

Finally, the resulting isomorphism $(R^{2 \times 2}/RI, +) \cong (\hat{N}, \cdot)$ by $B \mapsto \sigma_B$ is compatible with the action of $\SL_2(\Z/p^k\Z)$ by conjugation, since for $g \in \SL_2(\Z/p^k\Z)$ and $A, B \in R^{2 \times 2}$ with $\Tr(A) = 0$, we have
\[
\begin{aligned}
    \sigma_{gBg^{-1}}(I + p^jA) 
    &= e\left(\frac{\Tr(AgBg^{-1})}{p^{k-j}}\right)
    \\
    &=
    e\left(\frac{\Tr(g^{-1}AgB)}{p^{k-j}}\right)
    =
    \sigma_B(g(I + p^jA)g^{-1})
    =
    (g \cdot \sigma_B)(I + p^jA).
\end{aligned}
\]
This completes our proof.
\end{proof}

\begin{definition}[Primitive representations] \label{def:primitive}
A representation $\rho : \SL_2(\Z/p^k\Z) \to U(V)$ is called \emph{primitive} iff its kernel does not contain $\Gamma_{p^k}(p^{k-1})$. Equivalently (by the first isomorphism theorem), $\rho$ cannot be factored as $\rho' \circ \pi_{p^k, p^{k-1}}$ for some representation $\rho'$ of $\SL_2(\Z/p^{k-1}\Z)$. A primitive (resp., non-primitive) character is one associated to a primitive (resp., non-primitive) representation.
\end{definition}

Thus the primitive irreducible representations of $\SL_2(\Z/p^k\Z)$ are `new' at level $p^k$, much like primitive Dirichlet characters or newforms in the theory of automorphic representations. We can easily isolate the `maximal' non-primitive component of a representation using the following lemma.

\begin{lemma} \label{lem:prim-split}
Let $\rho : \SL_2(\Z/p^k\Z) \to U(V)$ be a representation and 
\[
    W := V^{\Gamma_{p^k}(p^{k-1})} = \{v \in V : \rho(n) v = v,\ \forall n \in \Gamma_{p^k}(p^{k-1}) \}.
\]
Then $\rho \vert_W$ is non-primitive, and $\rho \vert_{W^\perp}$ is isomorphic to a direct sum of primitive irreducible representations.
\end{lemma}

\begin{proof}
Recall that $\rho\vert_{W}$ and $\rho\vert_{W^\perp}$ are well-defined since $\Gamma_{p^k}(p^{k-1}) \triangleleft G$. By definition, $\rho\vert_{W}(n) = \Id_{W}$ for all $n \in \Gamma_{p^k}(p^{k-1})$, so the kernel of $\rho\vert_{W}$ includes $\Gamma_{p^k}(p^{k-1})$, i.e., $\rho\vert_{W}$ is non-primitive.

Now let $\rho\vert_{V_0}$ be any irreducible subrepresentation of $\rho\vert_{W^\perp}$, where $V_0 \subset W^\perp$. Since $V_0 \neq \{0\}$ and $V_0 \cap W = \{0\}$, we can find some $v \in V_0 \setminus W$, and thus some $n \in \Gamma_{p^k}(p^{k-1})$ such that $\rho(n)v \neq v$. But then $\rho\vert_{V_0}(n) \neq \Id_{V_0}$, so the kernel of $\rho\vert_{V_0}$ does not contain $\Gamma_{p^k}(p^{k-1})$, i.e., $\rho_{V_0}$ is primitive. 
\end{proof}

The classification of the (primitive) irreducible representations of $\SL_2(\Z/p^k\Z)$ has been the topic of numerous works \cite{kloosterman1946behaviourI,kloosterman1946behaviourII,shalika2004representation,tanaka1967irreducible,kutzko1973characters,nobs1976irreduziblenI,nobs1976irreduziblenII}; we remark in particular two papers of Kloosterman on this topic \cite{kloosterman1946behaviourI,kloosterman1946behaviourII}, with an approach based on theta series. The following preparatory lemma is very similar to \cite[Lemma 3]{kloosterman1946behaviourI}, but it includes the case $p = 2$.

\begin{lemma} \label{lem:quadr-congruence-count}
Let $p$ be a prime, $k \in \Z_+$, and $a, b, c \in \Z$. Then the number of solutions in $x, y \pmod{p^k}$ to the congruence
\begin{equation} \label{eq:quadr-congruence}
    a x^2 + b xy + c y^2 \equiv 1 \pmod{p^k}
\end{equation}
is $O(p^k)$ (where the implied constant is independent of $p, k, a, b, c$).
\end{lemma}

\begin{proof}
We may assume that $p \nmid (a, b, c)$, since otherwise \cref{eq:quadr-congruence} has no solutions.

First, we observe a Hensel-type lifting property: if $x, y$ give a solution to \cref{eq:quadr-congruence} with $p \nmid (2ax+by, 2cy+bx)$, then this solution modulo $p^k$ has exactly $p$ lifts to a solution $x', y'$ modulo $p^{k+1}$. Indeed, let $x, y \in \{0, \ldots, p^k-1\}$ such that $ax^2 + bxy + cy^2 = 1 + p^kt$ for some $t \in \Z$. Then we may write $x' = x + p^kr$, $y' = y + p^ks$, and solve the congruence
\[
    a(x')^2 + bx'y' + c(y')^2 \equiv 1 \pmod{p^{k+1}}
    \quad 
    \iff 
    \quad 
    t + (2ax+by)r + (2cy+bx)s \equiv 0 \pmod{p}
\]
in $r, s \pmod{p}$ in exactly $p$ ways. We will use this lifting property in the first two cases below.

\textbf{Case 1:} $p$ is odd.
Consider the solutions in $x, y \pmod{p}$ to $ax^2 + bxy + cy^2 \equiv 1 \pmod{p}$. If $p \nmid a$, then each value of $y \pmod{p}$ gives a quadratic congruence in $x \pmod{p}$, leading to a total of $O(p)$ solutions; if $p \nmid c$, the symmetric argument applies. If $p \mid (a, c)$ but $p \nmid b$, the congruence becomes $bxy \equiv 1 \pmod{p}$, which also has $O(p)$ solutions in $x, y \pmod{p}$.

Moreover, each of the $O(p)$ solutions in $x, y \pmod{p}$ to $ax^2 + bxy + cy^2 \equiv 1 \pmod{p}$ must satisfy $p \nmid (2ax+by, 2cy+bx)$, since otherwise $p \mid (2ax+by)x + (2cy+bx)y = 2(ax^2 + bxy + cy^2)$. The lemma then follows inductively from the lifting property.

\textbf{Case 2:} $p = 2$ and $b$ is odd. 
Then there are $O(1)$ solutions to $ax^2 + bxy + cy^2 \equiv 1 \pmod{2}$, and each of these solutions satisfies $2 \nmid (x, y)$ and $2 \nmid b$, so $2 \nmid (2ax + by, 2cy + bx)$. The lemma follows inductively from the lifting property.

\textbf{Case 3:} $p = 2$ and $b$ is even, say $b = 2b_0$. Then one of $a$ and $c$ must be odd, so by symmetry let us assume that $a$ is odd. Then the linear change of variables $x \mapsto x - \bar{a} b_0 y$ brings us to the congruence 
\[
    ax^2 + dy^2 \equiv 1 \pmod{2^k},
\]
where $d \equiv c-\bar{a}b_0^2 \pmod{2^k}$. 
If $d$ is even, then $x$ is an odd solution to the quadratic congruence $x^2 \equiv \bar{a}(1 - dy^2) \pmod{2^k}$, and there are only $O(1)$ such solutions for each choice of $y$ (indeed, for any two such solutions $x, x'$ we have $2^k \mid (x-x')(x+x')$, but one of $x-x', x+x'$ is a multiple of $4$ plus $2$). This gives a total of $O(2^k)$ solutions. 

We are left with the case that $a, d$ are both odd. Then exactly one of $x, y$ must be even, say $y$ by symmetry. Once again, for each choice of even $y$, there are $O(1)$ odd solutions in $x$ to the quadratic congruence $x^2 \equiv \bar{a}(1 - dy^2) \pmod{2^k}$, leading to a total of $O(2^k)$ solutions.
\end{proof}

The following result about primitive representations of $\SL_2(\Z/p^k\Z)$ will suffice for our purposes.

\begin{proposition} \label{prop:prim-rep-large} 
Any primitive irreducible representation $\rho$ of $\SL_2(\Z/p^k\Z)$ has $\dim \rho \gg p^k$ (where the implied constant is independent of both $p$ and $k$). 
\end{proposition}

\begin{proof}
This is a classical result if $k = 1$, going back to Frobenius (see, e.g., \cite{bourgain2008uniform}).
For $k \ge 2$, the possible dimensions of the irreducible representations of $\SL_2(\Z/p^k\Z)$ were determined by Nobs--Wolfart \cite{nobs1976irreduziblenII} (who refer to primitive representations in the sense of our \cref{def:primitive} as having `level' $k$ \cite[Definition 2]{nobs1976irreduziblenI}), building on the work of Kloosterman \cite{kloosterman1946behaviourI,kloosterman1946behaviourII}; for odd primes $p$, the irreducible representations had been classified by Shalika \cite[\S 4.3]{shalika2004representation}, Tanaka \cite{tanaka1967irreducible}, and Kutzko \cite{kutzko1973characters}. One can explicitly verify the bound $\dim \rho \gg p^k$ in the tables of Nobs--Wolfart \cite[p.\,525]{nobs1976irreduziblenII}.

There is also a more direct proof of the lower bound via Clifford theory and \cref{lem:quadr-congruence-count},
due to Bourgain--Gamburd \cite[Lemma 7.1]{bourgain2008expansion} (Bourgain--Gamburd assume $p$ is odd, but the argument can be adapted to cover the case $p = 2$). We include a variant of this proof here, assuming that $k$ is even; the case of odd $k$ is similar. 
We apply \cref{lem:clifford} with $G := \SL_2(\Z/p^k\Z)$ and $N := \Gamma_{p^k}(p^{k/2})$, to decompose $\rho\vert_N$ into irreducible representations $\sigma_1, \ldots, \sigma_L \in \hat{N}$, forming an orbit under $G$-conjugation. By \cref{lem:abelian-subgroups} with $j = \tfrac{k}{2}$ and $R = \Z/p^{k/2}\Z$, $\sigma_1, \ldots, \sigma_L \in \hat{N}$ are $1$-dimensional and correspond to elements $B_1 + RI, \ldots, B_L + RI \in R^{2 \times 2}/RI$, forming an orbit under $G$-conjugation; equivalently, they form an orbit under $\SL_2(R)$-conjugation. Moreover, the primitivity condition that $\ker \rho$ does not contain $\Gamma_{p^k}(p^{k-1})$ implies that $B_1, \ldots, B_L \not\in pR^{2 \times 2} + RI$: indeed, if some $B_j \in pR^{2 \times 2} + RI$, then all conjugates $B_1, \ldots, B_L \in pR^{2 \times 2} + RI$, and then $\sigma_1, \ldots, \sigma_L$ would be trivial on $\Gamma_{p^k}(p^{k-1})$ by the explicit description in \cref{eq:sigma-b}. It follows that
\[
    \dim \rho \ge L = \frac{|\SL_2(R)|}{|C_{\SL_2(R)}(B_1+RI)|} \gg \frac{p^{3k/2}}{|C_{\SL_2(R)}(B_1+RI)|},
\]
where $C_{\SL_2(R)}(B_1+RI)$ is the centralizer of $B_1+RI$ in $\SL_2(R)$. It thus remains to bound 
\begin{equation} \label{eq:centralizer-bound}
    |C_{\SL_2(R)}(B+RI)| \ll p^{k/2},
\end{equation} 
for $B \in R^{2 \times 2} \setminus (pR^{2 \times 2} + RI)$. We may take $B = \begin{psmall} a & b \\ c & 0 \end{psmall}$ without loss of generality, so that the primitivity condition $B \not\in pR^{2\times 2} + RI$ becomes $(a, b, c, p) = 1$. Then for any $\begin{psmall} x & y \\ z & t \end{psmall} \in C_{\SL_2(R)}(B+RI)$ we have 
    $\begin{psmall}
        x & y \\ z & t
    \end{psmall}
    \begin{psmall}
        a & b \\ c & 0
    \end{psmall}
    \begin{psmall}
        x & y \\ z & t
    \end{psmall}^{-1}
    =
    \begin{psmall}
        a+r & b \\ c & r
    \end{psmall}
    $
for some $r \in R$; taking traces gives $2r = 0$. Now let 
\[
    q := 
    \begin{cases} 
    p^{k/2}, & p \neq 2, \\ 
    2^{(k/2)-1}, & p = 2,
    \end{cases} 
\] 
so that $r \equiv 0 \pmod{q}$. We therefore have
    $\begin{psmall}
        x & y \\ z & t
    \end{psmall}
    \begin{psmall}
        a & b \\ c & 0
    \end{psmall}
    \equiv 
    \begin{psmall}
        a & b \\ c & 0
    \end{psmall} 
    \begin{psmall}
        x & y \\ z & t
    \end{psmall}
    \pmod{q}$,
which leads to the system of congruences
\begin{equation} \label{eq:system-R}
    bz \equiv cy, \qquad
    b(x-t) \equiv (a+r)y, \qquad 
    c(x-t) \equiv (a-r)z, \qquad
    xt - yz \equiv 1
    \quad 
    \pmod{q}.
\end{equation}
It suffices to solve this system in $x, y, z, t \pmod{q}$, since every such solution lifts to $O(1)$ solutions in $x, y, z, t \pmod{p^{k/2}}$.
If $(a, p) = 1$ (respectively, $(b, p) = 1$), then \cref{eq:system-R} determines $y, z$ in terms of $x, t$ (respectively, $z, t$ in terms of $x, y$), and implies the congruence
\[
    xt - \bar{a}^2bc (x-t)^2 \equiv 1 \pmod{q},
    \qquad 
    \text{respectively}
    \qquad 
    x(x-\bar{b}ay) - y\bar{b}cy \equiv 1 \pmod{q}.
\]
The case $(c, p) = 1$ is similar. \cref{lem:quadr-congruence-count} applies to all three cases, giving a total of $O(q) = O(p^{k/2})$ solutions; this establishes \cref{eq:centralizer-bound}, thus completing our proof.
\end{proof}

\section{Representations and Kloosterman matrices} \label{sec:rep-and-kloost}

Here we connect matrices of Kloosterman sums modulo $c$ to Fourier analysis on $\SL_2(\Z/c\Z)$.

\subsection{The relevant representations}
When digesting the notation below, the reader should keep in mind the informal outline from \cref{subsec:outline-Fourier}. We will first define the simpler representations $(\rho_c, V_c)$ which are connected to matrices of Kloosterman sums $S(m, n; c)$, and then the more relevant subrepresentations $(\rho^\circ_c, V_c^\circ)$ which correspond to adding the restriction $(m, n, c) = 1$. In fact, the subspace $V^\circ_c \subset V_c$ will be constructed by sifting out `old' subspaces isomorphic to $V_d$ for $d \mid c$.

\begin{definition}[Permutation representations of the projective action] \label{def:perm-rep}
For $c \in \Z_+$, we denote the permutation representation of $\SL_2(\Z/c\Z)$ associated to the action \cref{eq:action-proj} on $\P^1(\Z/c\Z)$ by
\begin{equation} \label{eq:action}
    \rho_c : \SL_2(\Z/c\Z) \to U(V_c),
    \qquad\quad 
    V_c := L^2(\P^1(\Z/c\Z)),
\end{equation}
and its character by $\chi_c := \Tr\, \rho_c$. Hence $V_c$ contains functions $f : \P^1(\Z/c\Z) \to \C$, and $(\rho_c(g)f)(u) = f(g^{-1}u)$ for $g \in \SL_2(\Z/c\Z)$, $u \in \P^1(\Z/c\Z)$. In particular, for $u \in \P^1(\Z/c\Z)$, one has $\rho_c(g) \one_u = \one_{gu}$.
\end{definition}

\begin{definition}[Invariant subspaces by congruence subgroups]\label{def:invt-subspace}
For $c, d \in \Z_+$ with $d \mid c$, we denote
\[
    V_c(d) := V_c^{\Gamma_c(d)} = \left\{ f \in V_c : \rho_c(n)f = f \quad \forall n \in \Gamma_c(d) \right\} \subset V_c.
\]
In particular, $V_c(c) = V_c$.
Thus $V_c(d)$ is the space of complex-valued functions on $\P^1(\Z/c\Z)$ which are constant on orbits of $\Gamma_c(d)$, so \cref{lem:action-transitive} gives a specific isomorphism
\begin{equation} \label{eq:invariant-orbits}
    V_c(d) \cong L^2(\Gamma_c(d)\backslash\P^1(\Z/c\Z)) \cong L^2(\P^1(\Z/d\Z)) = V_d.
\end{equation}
\end{definition}

\begin{lemma} \label{lem:invt-subspace}
Let $c, d \in \Z_+$ with $d \mid c$. Using \cref{not:red-mod-d}, we have
\[
    \rho_c \vert_{V_c(d)} \cong \rho_d \circ \pi_{c,d}.
\]
\end{lemma}

\begin{proof}
Let $\Phi : V_d \to V_c(d)$ be the invertible linear map from \cref{eq:invariant-orbits}, which relies on the bijection from \cref{eq:orbits-bijection}. Then for any $g \in \SL_2(\Z/c\Z)$, one can easily check that $\rho_c(g) \vert_{V_c(d)} \circ \Phi = \Phi \circ \rho_d(\pi_{c,d}(g))$: both maps take the basis vector $\one_u \in V_d = L^2(\P^1(\Z/d\Z))$ to the $L^2$-normalized function in $V_c(d)$ which is only nonzero on the orbit $g\Gamma_c(d) \cdot u = \Gamma_c(d) \cdot gu$.
\end{proof}

In light of \cref{lem:invt-subspace}, we will need to remove the contribution of `old' representations $(\rho_d, V_d)$ to $(\rho_c, V_c)$. To this end, it will be helpful to adopt the following convention for tensor products, which has a similar spirit to \cref{not:specified-isomorphisms}. 

\begin{notation}[Specified tensor products] \label{not:specified-tensor}
We identify the spaces $V_c$ and $\bigotimes_{p^k \| c} V_{p^k}$
via the bijection specified in \cref{eq:CRT-bij-P1} (given by the Chinese remainder theorem), where $\SL_2(\Z/c\Z)$ acts on $\bigotimes_{p^k \| c} V_{p^k}$ via the isomorphism specified in \cref{eq:CRT-isom-groups}. With this convention and \cref{not:box-product}, we have
\begin{equation} \label{eq:CRT-perm-rep}
    V_c = \bigotimes_{p^k \| c} V_{p^k}, \qquad \qquad 
    \rho_c = \bigboxtimes_{p^k \| c} \rho_{p^k}.
\end{equation}
and, more generally, for $d \mid c$,
\begin{equation} \label{eq:CRT-red-rep}
    V_c(d) = \bigotimes_{\substack{p^k \| c \\ p^j \| d}} V_{p^k}(p^j), 
    \qquad\qquad
    \rho_c\vert_{V_c(d)} = \bigboxtimes_{\substack{p^k \| c \\ p^j \| d}} \rho_{p^k}\vert_{V_{p^k}(p^j)}.
\end{equation}
\end{notation}
Finally, we can define the representations $(\rho_c^\circ, V_c^\circ)$.

\begin{definition}[Sifted representations] \label{def:sift-rep}
For a prime power $p^k$, we let $V_{p^k}^\circ := V_{p^k}(p^{k-1})^\perp \subset V_{p^k}$ be the orthogonal complement of $V_{p^k}(p^{k-1})$ inside $V_{p^k}$ (which is an invariant subspace of $\rho_{p^k}$). For $c \in \Z_+$, we define
\[
    V_c^\circ := \bigotimes_{p^k \| c} V_{p^k}^\circ, \qquad\qquad 
    \rho_c^\circ := \rho_c \vert_{V_c^\circ},
    \qquad\qquad 
    \chi_c^\circ := \Tr\, \rho_c^\circ.
\]
\end{definition}

\begin{remark}
When $c = p$ is a prime, $\rho_p^\circ$ is known as the Steinberg representation of $\SL_2(\Z/p\Z)$, which has dimension $p$ and arises by removing the trivial representation from $\rho_p$; this is well-known to be irreducible. But for $k \ge 2$, $\rho_{p^k}^\circ$ has dimension $(p^2-1)p^{k-2}$ in light of \cref{eq:CRT-sift-rep} below; if $p$ is odd, the table of primitive irreducible representations from \cite[p.\ 525]{nobs1976irreduziblenII} forces $\rho_{p^k}^\circ$ to split as a direct sum of two irreducible representations of dimension $\tfrac{1}{2} (p^2-1)p^{k-2}$.
\end{remark}

\begin{proposition}[Decomposition of sifted representations] \label{prop:sifted-rep}
For any $c \in \Z_+$, one has
\begin{equation} \label{eq:CRT-sift-rep}
    \rho_c^\circ = \bigboxtimes_{p^k \| c} \rho_{p^k}^\circ,
    \qquad\qquad 
    \dim \rho_c^\circ = c \prod_{\substack{p^k \| c \\ k \ge 2}} \left(1 - \frac{1}{p^2}\right).
\end{equation}
Moreover, each $\rho_{p^k}^\circ$ is isomorphic to a (nonempty) direct sum of primitive irreducible representations, and $\rho_c^\circ$ is isomorphic to a direct sum of $\ll c^{o(1)}$ irreducible representations of dimensions $\gg c^{1-o(1)}$.
\end{proposition}

\begin{proof}
The factorization of $\rho_c^\circ$ in \cref{eq:CRT-sift-rep} follows immediately from \cref{eq:CRT-perm-rep,def:sift-rep}, and this implies $\dim \rho_c^\circ = \prod_{p^k \| c} \dim \rho_{p^k}^\circ$. The dimension of each $\rho_{p^k}^\circ$ is simply $\dim V_{p^k}(p^{k-1})^\perp = \dim V_{p^k} - \dim V_{p^k}(p^{k-1}) = |\P^1(\Z/p^k\Z)| - |\P^1(\Z/p^{k-1}\Z)|$ by \cref{eq:invariant-orbits}, which evaluates to $p^k - \one_{k \ge 2}p^{k-2}$ by \cref{eq:P1-size}. In particular, $\rho_{p^k}^\circ \neq \rzero$.
The fact that $\rho_{p^k}^\circ$ is isomorphic to a direct sum of primitive irreducible representations is precisely the content of \cref{lem:prim-split}, wherein $W = V_{p^k}(p^{k-1})$ and $W^\perp = V_{p^k}^\circ$.

Now write each $\rho_{p^k}^\circ$ as a direct sum of primitive irreducible representations of $\SL_2(\Z/p^k\Z)$, and expand the tensor product in \cref{eq:CRT-sift-rep}. This expresses $\rho_c^\circ$ as a direct sum of representations (potentially with repetitions) of the shape
\[
    \rho = \bigboxtimes_{p^k \| c} \rho_{p,k},
\]
which are irreducible, and have dimensions $\gg c^{1-o(1)}$ by \cref{prop:prim-rep-large} and the divisor bound. Since $\dim \rho_c^\circ \le c$, the number of these representations is at most $c^{o(1)}$.
\end{proof}

We now briefly analyze the orthogonal projections onto invariant subspaces of $V_c$. It will turn out that the projection onto $V_c^\circ$ can be obtained by a M\"obius-inversion-type process.

\begin{notation}[Special projections]\label{not:projections}
For $c, d \in \Z_+$ with $d \mid c$, we let $P_c(d), P_c^\circ : V_c \to V_c$ be the orthogonal projections onto $V_c(d)$, respectively $V_c^\circ$. In particular, $P_c(c)$ is the identity map on $V_c$.
\end{notation}

Recalling that $V_c(d) = V_c^{\Gamma_c(d)}$, it follows directly from \cref{eq:CRT-red-rep} and \cref{lem:orth-proj-formula} that
\begin{equation} \label{eq:red-proj}
    \bigotimes_{\substack{p^k \| c \\ p^j \| d}} P_{p^k}(p^j) = P_c(d) = \frac{1}{|\Gamma_c(d)|} \sum_{n \in \Gamma_c(d)} \rho_c(n),
\end{equation}
and that $P_c(d)$ commutes with $\rho_c(g)$ for any $g \in \SL_2(\Z/c\Z)$. We will also need the following lemmas.

\begin{lemma} \label{lem:red-proj}
For $c, d\in \Z_+$ with $d \mid c$, the matrix representation of $P_c(d)$ map with respect to the standard basis of $V_c = L^2(\P^1(\Z/c\Z))$ has entries
\begin{equation} \label{eq:red-proj-entries}
    P_c(d)_{u,v} = \frac{|\P^1(\Z/d\Z)|}{|\P^1(\Z/c\Z)|} \one_{u \in \Gamma_c(d) \cdot v},
    \qquad\qquad u, v \in \P^1(\Z/c\Z).
\end{equation}
\end{lemma}

\begin{proof}
It follows from \cref{eq:red-proj} and \cref{def:perm-rep} that
\begin{equation} \label{eq:pass-to-stab}
    P_c(d)_{u,v} = \frac{1}{|\Gamma_c(d)|} \sum_{n \in \Gamma_c(d)} \one_{u = nv} = \frac{|\Gamma_c(d)_u|}{|\Gamma_c(d)|} \one_{u \in \Gamma_c(d)\cdot v},
\end{equation}
where $|\Gamma_c(d)_u|$ is the stabilizer of $u$ inside $\Gamma_c(d)$ (indeed, once $u = n_0v$ for some $n_0 \in \Gamma_c(d)$, all other solutions to $u = nv$ satisfy $n_0n^{-1}u = u$, so $n \in \Gamma_c(d)_u n_0$). But by the orbit-stabilizer theorem, $|\Gamma_c(d)|/|\Gamma_c(d)_u|$ is just the size of the orbit of $u$ in $\Gamma_c(d)\backslash \P^1(\Z/c\Z)$, which is $|\P^1(\Z/c\Z)|/|\P^1(\Z/d\Z)|$ by \cref{lem:action-transitive}.
\end{proof}

\begin{lemma} \label{lem:sift-proj}
For $c \in \Z_+$, one has
\[
    P_c^\circ = \bigotimes_{p^k \| c} P_{p^k}^\circ = \sum_{d \mid c} \mu\left(\frac{c}{d}\right) P_c(d).
\]
\end{lemma}

\begin{proof}
The factorization as a tensor product follows immediately from \cref{def:sift-rep,not:projections}. Now for a prime power $p^k$, recall that $P_{p^k}(p^k)$ is the identity map on $V_{p^k}$ and $P_{p^k}(p^{k-1})$ is the orthogonal projection onto $V_{p^k}(p^{k-1})$, so the orthogonal projection onto $V_{p^k}^\circ = V_{p^k}(p^{k-1})^\perp$ can be written as
\[
    P_{p^k}^\circ = P_{p^k}(p^k) - P_{p^k}(p^{k-1}).
\]
It follows from this and \cref{eq:red-proj} that
\[
\begin{aligned}
    \bigotimes_{p^k \| c} P_{p^k}^\circ
    &=
    \bigotimes_{p^k \| c} \left(P_{p^k}(p^k) - P_{p^k}(p^{k-1})\right) 
    \\
    &=
    \sum_{d \mid c} \mu\left(\frac{c}{d}\right) \bigotimes_{\substack{p^k \| c \\ p^j \| d}}
    P_{p^k}(p^j)
    \quad =
    \sum_{d \mid c} \mu\left(\frac{c}{d}\right) P_c(d),
\end{aligned}
\]
as claimed.
\end{proof}

\subsection{The Kloosterman matrix}
Here we finally relate the abstract discussion in the preceding subsections to the classical Kloosterman sums.

\begin{proposition}[From Kloosterman matrices to Fourier coefficients] \label{prop:kloost-to-fourier}
Let $c \in \Z_+$ and write $c = c_1c_2$ where $c_1$ is square-free, $c_2$ is square-full, and $(c_1, c_2) = 1$.
Let $\psi_1, \psi_2 : \Z/c\Z \to \C$ be any functions, and $K_c^{\psi_1,\psi_2} \in \C^{\Z/c\Z \times \Z/c\Z}$ be the $c \times c$ complex matrix with entries
\begin{equation} \label{eq:kl-mat}
    (K_c^{\psi_1,\psi_2})_{m,n} := \psi_1(m) \psi_2(n) S(m, n; c) \nu_{(m,n,c_1)} \one_{(m,n,c_2)=1},
\end{equation}
where $\nu_d := \prod_{\text{prime }p \mid d} \tfrac{-1}{p^2-1}$ for $d \in \Z_+$.
Consider the function $F_c^{\psi_1,\psi_2} : \SL_2(\Z/c\Z) \to \C$ given by
\begin{equation} \label{eq:func-psi12}
    F_c^{\psi_1,\psi_2} := \frac{1}{c^2} \sum_{h_1, h_2 \in \Z/c\Z} \hat\psi_1(h_1) \hat\psi_2(h_2)\, \one_{T^{h_1} S T^{h_2}},
\end{equation}
where $T$ and $S$ are as in \cref{eq:SL2-generators}. Then one has the inequality of operator norms
\[
    \|K_c^{\psi_1,\psi_2}\| \le c \|\hat F_c^{\psi_1,\psi_2}(\rho_c^\circ)\|.
\]
\end{proposition}

\begin{remark}
In $\hat\psi_1$ and $\hat\psi_2$, the Fourier transform is taken over $\Z/c\Z$, as in \cref{eq:Fourier-tr-ab}. In $\hat F^{\psi_1,\psi_2}_c$, the Fourier transform is taken over the non-abelian group $\SL_2(\Z/c\Z)$, as in \cref{eq:Fourier-tr-nonab}.
\end{remark}

\begin{proof}[Proof of \cref{prop:kloost-to-fourier}]
Let $U_c$ be the unitary $c \times c$ matrix with entries $(U_c)_{u,v} = c^{-1/2} e(\tfrac{uv}{c})$. By expanding the Kloosterman sums, we have that for any $u, v \in \Z/c\Z$,
\begin{equation} \label{eq:after-unitary}
    (U_c^* K_c^{\psi_1,\psi_2} U_c)_{u,v} 
    =
    \frac{1}{c} \sum_{\substack{m, n \in \Z/c\Z \\ x \in (\Z/c\Z)^\times}}  
    \nu_{(m,n,c_1)}\one_{(m,n,c_2) = 1} 
    \psi_1(m)\, e\left(\frac{m(x - u)}{c}\right) \psi_2(n)\, e\left(\frac{n(\bar{x}+v)}{c}\right).
\end{equation}
Now consider the multiplicative function $w : \Z_+ \to \R$ given by
\begin{equation} \label{eq:w-function}
    w(d) := \prod_{\text{prime }p \mid d} \left(1 - \frac{1}{p^2}\right).
\end{equation}
We then have
\[
\begin{aligned}
    \sum_{d \mid (m, n, c)} \mu(d) \frac{w(c/d)}{w(c)}
    &=
    \prod_{\substack{\text{prime }p \mid (m, n, c) \\ p^k \| c}} \left(1 - \frac{w(p^{k-1})}{w(p^k)}\right)
    \\
    &=
    \prod_{\text{prime }p \mid (m, n, c_1)} \left(1 - \frac{1}{1-\frac{1}{p^2}}\right)
    \cdot 
    \prod_{\text{prime }p \mid (m, n, c_2)} \left(1 - 1\right)
    =
    \nu_{(m,n,c_1)}
    \one_{(m,n,c_2) = 1},
\end{aligned}
\]
so by Fourier analysis,
\[
\begin{aligned}
    \nu_{(m,n,c_1)}
    \one_{(m,n,c_2) = 1}
    = \sum_{d \mid c} \mu(d) \frac{w(c/d)}{w(c)} \one_{d \mid m} \one_{d \mid n}
    = \sum_{d \mid c} \mu(d)\frac{w(c/d)}{d^2w(c)} \sum_{a, b \in \Z/d\Z} e\left(\frac{am}{d}\right) e\left(\frac{bm}{d}\right).
\end{aligned}
\]
We now plug this into \cref{eq:after-unitary} and evaluate the sums over $m, n$ to obtain
\[
\begin{aligned}
    (U_c^* K_c^{\psi_1,\psi_2} U_c)_{u,v} 
    &=
    \frac{1}{c} \sum_{d \mid c} \mu(d) \frac{w(c/d)}{d^2w(c)}
    \sum_{x \in (\Z/c\Z)^\times}
    \sum_{a, b \in \Z/d\Z}
    \hat\psi_1\left(-x + u - \frac{ac}{d}\right) \hat\psi_2\left(-\bar x - v - \frac{bc}{d}\right)
    \\
    &=
    \frac{1}{c} \sum_{d \mid c} \mu(d) \frac{w(c/d)}{d^2w(c)}
    \sum_{x \in (\Z/c\Z)^\times}
    \sum_{h_1, h_2 \in \Z/c\Z}
    \hat\psi_1(h_1)
    \hat\psi_2(h_2)
    \one_{\substack{x \equiv u-h_1 \pmod{\frac{c}{d}} \\ -\bar{x} \equiv v+h_2 \pmod{\frac{c}{d}}}},
\end{aligned}
\]
where we substituted $h_1 := -x + u - \tfrac{ac}{d}$, $h_2 = -\bar{x} - v - \tfrac{bc}{d}$. Switching divisors $d \mapsto \tfrac{c}{d}$ and swapping sums, we reach
\[
    (U_c^* K_c^{\psi_1,\psi_2} U_c)_{u,v} 
    =
    \frac{1}{c} 
    \sum_{h_1, h_2 \in \Z/c\Z}
    \hat\psi_1(h_1)
    \hat\psi_2(h_2)
    \sum_{d \mid c} \mu\left(\frac{c}{d}\right) \frac{d^2w(d)}{c^2w(c)}
    \sum_{x \in (\Z/c\Z)^\times}
    \one_{\substack{x \equiv u-h_1 \pmod{d} \\ -\bar{x} \equiv v+h_2 \pmod{d}}}.
\]
The inner sum over $x$ evaluates to $\tfrac{\phi(c)}{\phi(d)}$ if $(u-h_1)(v+h_2) \equiv -1 \pmod{d}$, and vanishes otherwise. Since $d^2w(d)\phi(d)^{-1} = |\P^1(\Z/d\Z)|$ by \cref{eq:w-function,eq:P1-size}, we find that
\begin{equation} \label{eq:kl-times-unitary}
    (U_c^* K_c^{\psi_1,\psi_2} U_c)_{u,v} =
    \frac{1}{c} \sum_{h_1, h_2 \in \Z/c\Z}
    \hat\psi_1(h_1)
    \hat\psi_2(h_2)
    \sum_{d \mid c} \mu\left(\frac{c}{d}\right)
    \frac{|\P^1(\Z/d\Z)|}{|\P^1(\Z/c\Z)|}
    \one_{(u-h_1)(v+h_2) \equiv -1 \pmod{d}}.
\end{equation}
Let us keep this in mind. Separately, by \cref{eq:func-psi12,def:sift-rep}, we have
\[
\begin{aligned}
    \hat F_c^{\psi_1,\psi_2}(\rho_c^\circ)
    &=
    \frac{1}{c^2} \sum_{h_1,h_2 \in \Z/c\Z} \hat\psi_1(h_1) \hat\psi_2(h_2) \rho_c^\circ(T^{h_1} S T^{h_2})
    \\
    &=
    \Bigg(\frac{1}{c^2} \sum_{h_1,h_2 \in \Z/c\Z} \hat\psi_1(h_1) \hat\psi_2(h_2) \rho_c(T^{h_1} S T^{h_2}) \Bigg)\Bigg\vert_{V_c^\circ},
\end{aligned}
\]
and thus by \cref{lem:norm-proj},
\begin{equation} \label{eq:norm-equality-2}
    \|\hat F_c^{\psi_1,\psi_2}(\rho_c^\circ)\| = \|M_c^{\psi_1,\psi_2}\|,
\end{equation}
where $M_c^{\psi_1,\psi_2} : V_c \to V_c$ is the map
\[
\begin{aligned}
    M_c^{\psi_1,\psi_2}
    &= \frac{1}{c^2} \sum_{h_1,h_2 \in \Z/c\Z} \hat\psi_1(h_1) \hat\psi_2(h_2) \rho_c(T^{h_1} S T^{h_2}) P_c^\circ.
\end{aligned}
\]
By \cref{lem:sift-proj} and the commutativity of $P_c(d)$ with $\rho_c(g)$ for any $g \in \SL_2(\Z/c\Z)$, we can further write
\[
\begin{aligned}
    M_c^{\psi_1,\psi_2} 
    &= \frac{1}{c^2} \sum_{h_1,h_2 \in \Z/c\Z} \hat\psi_1(h_1) \hat\psi_2(h_2)
     \rho_c(T^{h_1} S T^{h_2}) \sum_{d \mid c} \mu\left(\frac{c}{d}\right) P_c(d)
    \\
    &= \frac{1}{c^2} \sum_{h_1,h_2 \in \Z/c\Z} \hat\psi_1(h_1) \hat\psi_2(h_2)
    \sum_{d \mid c} \mu\left(\frac{c}{d}\right) \rho_c(T^{h_1}) P_c(d) \rho_c(S T^{h_2}).
\end{aligned}
\]
By \cref{def:perm-rep,eq:red-proj-entries}, we can represent this map as a matrix in $\C^{\P^1(\Z/c\Z) \times \P^1(\Z/c\Z)}$ with entries
\begin{equation} \label{eq:kl-fourier-side}
    (M_c^{\psi_1,\psi_2})_{u,v}
    =
    \frac{1}{c^2} \sum_{h_1,h_2 \in \Z/c\Z} \hat\psi_1(h_1) \hat\psi_2(h_2)
    \sum_{d \mid c} \mu\left(\frac{c}{d}\right) \frac{|\P^1(\Z/d\Z)|}{|\P^1(\Z/c\Z)|}
    \one_{T^{-h_1}u \in \Gamma_c(d) \cdot ST^{h_2}v}
\end{equation}
for $u, v \in \P^1(\Z/c\Z)$; compare this to \cref{eq:kl-times-unitary}. We will show that restricting the matrix $c M_c^{\psi_1,\psi_2}$ to those rows and columns indexed by $u, v \in \Z/c\Z \subset \P^1(\Z/c\Z)$ (by the canonical embedding $x \mapsto [x : 1]$) yields precisely the matrix $U_c^* K_c^{\psi_1,\psi_2} U_c$. Indeed, using the notation above, if $u, v \in \Z/c\Z$, then $T^{-h_1} u = u-h_1 =: x \in \Z/c\Z$, $T^{h_2} v = v + h_2 =: y \in \Z/c\Z$, and we have $x \in \Gamma_c(d) \cdot Sy$ if and only if the equation
\[
    g
    \begin{pmatrix} x \\ 1 \end{pmatrix} 
    =
    \alpha
    \begin{pmatrix} -1 \\ y \end{pmatrix} 
\]
has solutions in $g \in \Gamma_c(d)$ and $\alpha \in (\Z/c\Z)^\times$. The existence of such solutions implies that $\begin{psmall} x \\ 1 \end{psmall} \equiv \begin{psmall} -\alpha \\ \alpha y \end{psmall} \pmod{d}$, so $xy \equiv -1 \pmod{d}$. On the other hand, if $xy \equiv -1 \pmod{d}$, then we can pick $\alpha \in (\Z/c\Z)^\times$ with $\alpha \equiv -x \pmod{d}$ and $\bar\alpha \equiv y \pmod{d}$, and $g = \begin{psmall} -\alpha & -1 \\ \alpha y & y-\bar\alpha \end{psmall} \begin{psmall} x & -1 \\ 1 & 0 \end{psmall}^{-1}$ to obtain a solution (note that $\begin{psmall} -\alpha & -1 \\ \alpha y & y-\bar\alpha \end{psmall} \equiv \begin{psmall} x & -1 \\ 1 & 0 \end{psmall} \pmod{d}$, so $g \in \Gamma_c(d)$).
It follows that for $u, v \in \Z/c\Z \subset \P^1(\Z/c\Z)$, one has
\[
    \one_{T^{-h_1}u \in \Gamma_c(d) \cdot ST^{h_2}v} = \one_{(u-h_1)(v+h_2) \equiv -1 \pmod{d}},
\]
and then by comparing \cref{eq:kl-times-unitary,eq:kl-fourier-side}, we find that
\[
    (U_c^* K_c^{\psi_1,\psi_2} U_c)_{u,v} = c(M_c^{\psi_1,\psi_2})_{u,v},
    \qquad\qquad 
    u, v \in \Z/c\Z \subset \P^1(\Z/c\Z).
\]
Since removing some rows and columns of a matrix can only decrease its spectral norm, we conclude that
\[
    \|K_c^{\psi_1,\psi_2}\| = \|U_c^* K_c^{\psi_1,\psi_2} U_c\| \le c\|M_c^{\psi_1,\psi_2}\|,
\]
which, together with \cref{eq:norm-equality-2}, completes our proof.
\end{proof}

\begin{corollary} \label{cor:kloosterman-ap}
Let $c \in \Z_+$ and write $c = c_1c_2$ where $c_1$ is square-free, $c_2$ is square-full, and $(c_1, c_2) = 1$.
Let $M, N \in \Z_+$ with $1 \le M, N \le c$, $a \in (\Z/c\Z)^\times$, and $\mI, \mJ \subset \Z$ be intervals of lengths $|\mI| = M$, $|\mJ| = N$. Let $K_{c,a}^{\mI,\mJ} \in \C^{\mI \times \mJ}$ be the $M \times N$ matrix indexed by $m \in \mI$ and $n \in \mJ$, with entries
\begin{equation} \label{eq:kl-mat-ap}
    (K_{c,a}^{\mI,\mJ})_{m,n} := S(am, n; c) \nu_{(m,n,c_1)} \one_{(m,n,c_2)=1},
\end{equation}
where $\nu_d := \prod_{\text{prime }p \mid d} \tfrac{-1}{p^2-1}$ for $d \in \Z_+$.
Let $\eps > 0$ and $H_1 := c^{1+\eps} M^{-1}$, $H_2 := c^{1+\eps} N^{-1}$. Then there exist absolutely-bounded complex numbers $z_h, w_h \ll 1$ such that for the function $F_{c,a}^{H_1,H_2} : \SL_2(\Z/c\Z) \to \C$ given by
\begin{equation} \label{eq:func-H12}
    F_{c,a}^{H_1,H_2} := \frac{1}{H_1H_2} \sum_{\substack{|h_1| \le H_1 \\ |h_2| \le H_2}} z_{h_1} w_{h_2} \one_{T^{\bar{a}h_1} S T^{h_2}},
\end{equation}
one has
\begin{equation} \label{eq:kl-bound-intervals}
    \|K_{c,a}^{\mI,\mJ}\| \le c^{1+2\eps}\|\hat F_{c,a}^{H_1,H_2} (\rho_c^\circ)\| + O_\eps(c^{-100}).
\end{equation}
\end{corollary}

\begin{remark}
Given \cref{eq:func-H12}, one can apply the triangle inequality for the operator norm to obtain
\[
\begin{aligned}
    \|\hat F_{c,a}^{H_1,H_2}(\rho_c^\circ)\| 
    &=
    \Bigg\|\frac{1}{H_1H_2} \sum_{\substack{|h_1| \le H_1 \\ |h_2| \le H_2}} \alpha_{h_1} \beta_{h_2}\, \rho_c^\circ(T^{ah_1} S T^{h_2}) \Bigg\| 
    \\
    &\le 
    \frac{1}{H_1H_2} \sum_{\substack{|h_1| \le H_1 \\ |h_2| \le H_2}} |\alpha_{h_1} \beta_{h_2}| \cdot \|\rho_c^\circ(T^{ah_1} S T^{h_2})\|
    \ll 1,
\end{aligned}
\]
since $\|\rho_c^\circ(T^{ah_1} S T^{h_2})\| = 1$ (as the norm of a unitary map). Our task in the later sections will be to establish some power-saving cancellation in the sum over $h_1, h_2$.
\end{remark}

\begin{proof}[Proof of \cref{cor:kloosterman-ap}]
Let us write $[M] := \{1, \ldots, M\}$, $[N] := \{1, \ldots, N\}$, and $\mI = [M] + r$, $\mJ = [N] + s$ for some $r, s \in \Z$. Since $M, N \le c$, we may identify $\mI$, $\mJ$ with their images in $\Z/c\Z$. Let $\Phi : \R \to \C$ be a smooth function supported in $[-1, 2]$, such that $\Phi \ge \one_{[0,1]}$ and $\Phi^{(j)} \ll_j 1$ for $j \ge 0$, and define $\psi_1, \psi_2 : \Z/c\Z \to \C$ by
\begin{equation} \label{eq:psi12-def}
    \psi_1(m) := \sum_{\substack{m' \in \Z \\ a(m'+r) \equiv m \pmod{c}}} \Phi\left(\frac{m'}{M}\right),
    \qquad\qquad 
    \psi_2(n) := \sum_{\substack{n' \in \Z \\ n'+s \equiv n \pmod{c}}} \Phi\left(\frac{n'}{N}\right).
\end{equation}
Since $\Phi \ge \one_{[0,1]}$, we have $\psi_1 \ge \one_{a\mI}$ and $\psi_2 \ge \one_{\mJ}$ (viewing these as functions on $\Z/c\Z$). But scaling a row or a column of a matrix by a constant in $[0, 1]$ can only decrease its spectral norm, so with the notation from \cref{eq:kl-mat} we get
\[
    \|K_c^{\psi_1,\psi_2}\| \ge \|K_{c,a}^{\mI,\mJ}\|.
\]
From \cref{prop:kloost-to-fourier}, it then follows that
\begin{equation} \label{eq:kl-norm-intermediate}
    \|K_{c,a}^{\mI,\mJ}\| \le c\|\hat{F}_c^{\psi_1,\psi_2}(\rho_c^\circ)\|,
\end{equation}
and it remains to compute $F_c^{\psi_1,\psi_2}$. For $h \in \Z/c\Z$, we obtain from \cref{eq:psi12-def}, \cref{eq:Fourier-tr-ab}, \cref{eq:Poisson}, and \cref{eq:scaled-Fourier-tr} that
\[
\begin{aligned}
    \hat\psi_1(h) = \sum_{m \in \Z/c\Z} \psi_1(m)\, e\left(-\frac{hm}{c}\right)
    &=
    \sum_{m' \in \Z} \Phi\left(\frac{m'}{M}\right) e\left(-\frac{ah(m'+r)}{c}\right)
    \\
    &=
    M e\left(-\frac{rah}{c}\right) \sum_{k \in \Z} \hat\Phi\left(M\Big(k+\frac{ah}{c}\Big)\right)
    \\
    &=
    M e\left(-\frac{rah}{c}\right) \sum_{\substack{h' \in \Z \\ h' \equiv ah \pmod{c}}} \hat\Phi\left(\frac{h'}{c/M}\right),
\end{aligned}
\]
and similarly for $\hat\psi_2(h)$ (with $1$ in place of $a$). Plugging this into \cref{eq:func-psi12}, we conclude that 
\[
\begin{aligned}
    F_c^{\psi_1,\psi_2} 
    &= 
    \frac{MN}{c^2} \sum_{h_1, h_2 \in \Z/c\Z} \sum_{\substack{h_1', h_2' \in \Z \\ h_1' \equiv ah_1 \pmod{c} \\ h_2' \equiv h_2 \pmod{c}}} \hat\Phi\left(\frac{h_1'}{c/M}\right) \hat\Phi\left(\frac{h_2'}{c/N}\right) e\left(\frac{-rah_1-sh_2}{c}\right) \one_{T^{h_1} S T^{h_2}}
    \\
    &= 
    \frac{c^{2\eps}}{H_1 H_2} \sum_{h_1', h_2' \in \Z} \hat\Phi\left(\frac{h_1'}{c/M}\right) \hat\Phi\left(\frac{h_2'}{c/N}\right) e\left(\frac{-rah_1'-sh_2'}{c}\right) \one_{T^{\bar{a} h_1'} S T^{h_2'}}.
\end{aligned}
\]
Using the Schwarz decay of $\hat\Phi$, we can discard the contribution of the terms with $|h_1'| > H_1$ or $|h_2'| > H_2$ to $\|\hat{F}_c^{\psi_1,\psi_2}(\rho_c^\circ)\|$, up to an error of $O_\eps(c^{-100})$. Choosing
\[
    z_h := \hat\Phi\left(\frac{h}{c/M}\right) e\left(-\frac{rah}{c}\right),
    \qquad\qquad 
    w_h := \hat\Phi\left(\frac{h}{c/N}\right) e\left(-\frac{sh}{c}\right)
\]
concludes our proof in light of \cref{eq:kl-norm-intermediate} and \cref{eq:func-H12}.
\end{proof}

\section{The amplification argument} \label{sec:amplif}
Recall that \cref{prop:kloost-to-fourier,cor:kloosterman-ap} reduce the problem of bounding bilinear forms with Kloosterman sums to that of bounding Fourier coefficients of certain functions on $\SL_2(\Z/c\Z)$ at a certain representation. One can then reduce to irreducible subrepresentations via \cref{lem:break-down-schatten}. 
To use Fourier analysis, we will pass to a sum over all irreducible representations of $\SL_2(\Z/c\Z)$. To avoid a critical loss, we insert an amplifier weight in this sum, as outlined in \cref{subsec:outline-amplif}. Making this work in a non-abelian setting is a key step in our argument.

\subsection{Introducing the amplifier} We recall the notation for Schatten norms from \cref{subsec:arithm-analytic-not}.

\begin{proposition}[Non-abelian amplification] \label{prop:amplification}
Let $G$ be a finite group, $N \triangleleft G$, $F : G \to \C$, $\rho \in \hat{G}$, $\chi := \Tr\, \rho$, and $q$ be an even positive integer. Then one has
\[
    \|\hat{F}(\rho)\|_{S^q}^q \le \frac{|G|}{\sum_{n \in N} |\chi(n)|^2} \sum_{\substack{g_1, \ldots, g_q \in G \\ g_1 \cdots g_q \in N}} F(g_1) \bar F(g_2^{-1}) \cdots F(g_{q-1}) \bar F(g_q^{-1}) \chi(g_1 \cdots g_q).
\]
\end{proposition}

\begin{remark}
In comparison, expanding $\|\hat{F}(\rho)\|_{S^q}^q$ by \cref{eq:Fourier-tr-nonab,eq:schatten} yields
\[
    \|\hat{F}(\rho)\|_{S^q}^q =  \sum_{g_1, \ldots, g_q \in G} F(g_1) \bar F(g_2^{-1}) \cdots F(g_{q-1}) \bar F(g_q^{-1}) \chi(g_1 \cdots g_q).
\]
The upper bound from \cref{prop:amplification} replaces $\chi$ with a function proportional to $\chi \one_N$, normalized so that equality is attained when $F = \bar{\chi}$. The requirement that $\rho$ be irreducible is crucial.
\end{remark}


\begin{proof}[Proof of \cref{prop:amplification}]
We view $\rho$ as a subrepresentation of a suitable representation $R$, and therefore $\hat{F}(\rho)$ as a term (with multiplicity) in a direct sum which equals $\hat{F}(R)$. 
Starting from $\rho : G \to U(V)$, we consider the restricted representation $\rho\vert_N : N \to U(V)$, and then the induced representation
\[
    R := \Ind_N^G(\rho \vert_N)
\]
of $G$. 
By \cref{eq:multiplicity-as-sum} and Frobenius reciprocity (\cref{lem:frob-rec}), for any irreducible representation $\rho'$ of $G$, we have (writing $\chi' := \Tr\, \rho'$)
\begin{equation} \label{eq:mult-in-induced}
    \Mult(\rho', R) = \frac{1}{|G|} \sum_{g \in G} \Tr R(g) \bar\chi'(g) = \frac{1}{|N|} \sum_{n \in N} \chi(n) \bar\chi'(n).
\end{equation}
Note that up to a normalizing factor, this has the shape of the amplifier anticipated in \cref{eq:sketch-non-abelian-amplifier} with $\mL = N$. Using \cref{eq:mult-in-induced} with $\rho' = \rho$ and \cref{lem:break-down-schatten}, we obtain
\[
    \frac{1}{|N|} \left(\sum_{n \in N} |\chi(n)|^2\right) \|\hat{F}(\rho)\|_{S^q}^q
    \le 
    \sum_{\rho' \in \hat{G}} \Mult(\rho', R) \|\hat{F}(\rho')\|_{S^q}^q
    =
    \|\hat{F}(R)\|_{S^q}^q.
\]
We then use \cref{eq:Fourier-tr-nonab,eq:schatten} to expand
\[
\begin{aligned}
    \|\hat{F}(R)\|_{S^q}^q 
    &= 
    \Tr \left(\left(\sum_{g_1 \in G} F(g_1) R(g_1) \sum_{g_2 \in G} \bar{F}(g_2^{-1}) R(g_2)\right)^{q/2}\right)^{\frac{1}{q}}
    \\
    &=
    \sum_{g_1, \ldots, g_q \in G} F(g_1) \bar F(g_2^{-1}) \cdots F(g_{q-1}) \bar F(g_q^{-1}) \Tr R(g_1\cdots g_q).
\end{aligned}
\]
Finally, we use the character formula \cref{eq:char-ind} to write
\[
    \Tr R(g) = \frac{1}{|N|} \sum_{\substack{x \in G \\ x^{-1} g x \in N}} \chi(x^{-1} g x) = 
    \frac{|G|}{|N|} \chi(g) \one_N(g),
    \qquad 
    \forall g \in G,
\]
where the last equality uses the normality of $N$ and the fact that $\chi$ is a character of $G$.
\end{proof}

\begin{remark}
In the particular case when $N = \{e\}$ is the trivial subgroup, the induced representation $R$ in the proof above is (isomorphic to) the regular representation $R_G$. In this case, the conclusion of \cref{prop:amplification} simply reads
\[
    \|\hat{F}(\rho)\|_{S^q}^q \le \frac{1}{\dim \rho} \|\hat{F}(R_G)\|_{S^q}^q,
\]
and similar ideas appear in \cite{shkredov2018asymptotic,moshchevitin2020pacific}, as well as \cite[Proof of Proposition 8.2]{green2025quadratic}.
\end{remark}

\begin{remark}
One can of course rephrase the proof of \cref{prop:amplification} without the language of induced representations and Frobenius reciprocity, by directly inserting the amplifier from the right-hand side of \cref{eq:mult-in-induced} (which is nonnegative by \cref{eq:sketch-non-abelian-amplifier}), and evaluating the sum over $\rho'$ via \cref{eq:ortho-chars-sum-chi}. This might be helpful for variations of our argument, using other choices of the set $\mL$ in \cref{eq:sketch-non-abelian-amplifier}.
\end{remark}

\subsection{Bounds for non-abelian characters}
We now return to the setting when $G = \SL_2(\Z/c\Z)$ and $N = \Gamma_c(d)$ for some $d \mid c$. The goal is to pass from the spectral norm $\|\hat F_{c,a}^{H_1,H_2}(\rho_c^\circ)\|$ in \cref{eq:kl-bound-intervals} to a count of solutions to a certain equation in $\PSL_2(\Z/d\Z)$. After applying \cref{prop:amplification}, we will need an upper bound for $\chi(g_1\cdots g_q)$, and a lower bound for the denominator $\sum_{n \in N} |\chi(n)|^2$. For the specific characters $\chi_c$ and $\chi^\circ_c$ from \cref{sec:rep-and-kloost}, such bounds are given in the following results.

\begin{lemma} \label{lem:perm-rep-pointwise-ub}
Let $c \in \Z_+$, $g \in \SL_2(\Z/c\Z)$, and $d$ be the largest divisor of $c$ such that 
\[
    g \in Z(\SL_2(\Z/c\Z)) \cdot \Gamma_c(d).
\] 
Let $f \le \sqrt{cd}$ be the largest positive integer such that $f^2 \mid cd$.
Then one has
\begin{equation} \label{eq:char-pwise-bound}
    \chi_c(g) \ll c^{o(1)} f.
\end{equation}
\end{lemma}

\begin{remark}
The fact that the center of $\SL_2(\Z/c\Z)$ appears in \cref{lem:perm-rep-pointwise-ub} is natural, since the permutation representation $\rho_c$ (corresponding to the action on $\P^1(\Z/c\Z)$) factors through $\PSL_2(\Z/c\Z)$. This is ultimately why the counting problems in \cref{sec:counting} take place in $\PSL_2(\Z/c\Z)$, but we found it more convenient to phrase \cref{sec:rep-and-kloost,sec:amplif} in terms of the group $\SL_2(\Z/c\Z)$.
\end{remark}

\begin{proof}[Proof of \cref{lem:perm-rep-pointwise-ub}]
From \cref{eq:CRT-perm-rep}, it follows that $\chi_c(g) = \prod_{p^k \| c} \chi_{p^k}(\pi_{c,p^k}(g))$. Working locally at a prime $p | c$, with say $p^k \| c$ and $p^j \| d$, we will establish the bound
\begin{equation} \label{eq:char-pwise-bound-local}
    \chi_{p^k}(g) \ll p^{\lf \frac{k+j}{2} \rf},
\end{equation}
for all $g \in \SL_2(\Z/p^k\Z)$ such that $p^j$ is the largest divisor of $p^k$ for which $g \in Z(\SL_2(\Z/p^k\Z)) \cdot \Gamma_{p^k}(p^j)$. Given \cref{eq:char-pwise-bound-local}, the desired bound in \cref{eq:char-pwise-bound} follows from the divisor bound.

Since $\rho_{p^k}$ is a permutation representation, $\chi_{p^k}(g) = \Tr\, \rho_{p^k}(g)$ equals the number of fixed points of $g$ in $\P^1(\Z/p^k\Z)$, i.e., the number of solutions in $u \in \P^1(\Z/p^k\Z)$ to $gu = u$. Let us write $g = \begin{psmall} q & r \\ s & t \end{psmall}$ and $u = [x : y]$ for some integers $q, r, s, t, x, y$ with $qt - rs \equiv 1 \pmod{p^k}$ and $(x, y, p) = 1$. Scaling both entries of $u$ by a unit in $(\Z/p^k\Z)^\times$, we can assume without loss of generality that $x = 1$ or $y = 1$; in fact, replacing $\begin{psmall} q & r \\ s & t \end{psmall} \leftrightarrow \begin{psmall} t & s \\ r & q \end{psmall}$ if necessary, we may assume that $y = 1$.
Then the equality $gu = u$ means that for some $\alpha \in \Z$, one has
\[
    \begin{pmatrix} q & r \\ s & t \end{pmatrix} \begin{pmatrix} x \\ 1 \end{pmatrix} \equiv \alpha \begin{pmatrix} x \\ 1 \end{pmatrix} \pmod{p^k}
    \qquad \Rightarrow \qquad 
    qx+r \equiv (sx + t)x \pmod{p^k}.
\]
This gives the quadratic congruence
\[
    sx^2 + (t-q)x - r \equiv 0 \pmod{p^k}.
\]
Now recall the explicit description of $Z(\SL_2(\Z/p^k\Z))$ from \cref{eq:centers}, so we have $g \in \gamma\Gamma_{p^k}(p^j)$ for some $\gamma \in \Z/p^k\Z$ with $\gamma^2 = 1$. Therefore, we know $p^j \mid s$, $p^j \mid r$, and $p^j \mid t - q$ (since $t \equiv \gamma \equiv q \pmod{p^j}$). 

We first assume that $p$ is odd; we will comment on the case $p = 2$ at the end of the proof. Then $\gamma \in \{\pm 1\}$, and $p^j$ is the largest power of $p$ dividing $(s, r, t-q, p^k)$ (otherwise, $q^2 \equiv 1 \pmod{p^{j+1}}$ forces $q \equiv \pm 1 \pmod{p^{j+1}}$, so $g \in \pm \Gamma_{p^k}(p^{j+1})$). Therefore, letting $a_2 := sp^{-j}$, $a_1 := (t-q)p^{-j}$ and $a_0 := -rp^{-j}$, we find that
\begin{equation} \label{eq:quadr-mod-p}
    a_2 x^2 + a_1 x + a_0 \equiv 0 \pmod{p^{k-j}},
\end{equation}
where $a_0, a_1, a_2$ are not all divisible by $p$. It now remains to show that this equation has 
\[
    O\left(p^{\lf \frac{k-j}{2} \rf} \right)
\] 
solutions in $x \pmod{p^{k-j}}$; every such solution will have $p^j$ lifts to $\Z/p^k\Z$, inducing a total of $O(p^{\lf (k-j)/2 \rf + j}) = O(p^{\lf (k+j)/2 \rf})$ fixed points $u = [x : 1]$ of $g$. We may assume that $j < k$, since otherwise this claim is trivial.

The number of solutions to quadratic congruences like \cref{eq:quadr-mod-p} has of course been studied before (see, e.g., \cite[Lemma 3]{kloosterman1946behaviourI} for a related problem). We derive an upper bound using three quick cases.

\textbf{Case 1:} $p \nmid a_2$. Then given any two solutions $x_0, x$ of \cref{eq:quadr-mod-p}, we can subtract the two equalities to obtain
\begin{equation} \label{eq:after-subtract-eqn}
    p^{k-j} \mid a_2(x^2 - x_0^2) + a_1(x - x_0) = (x - x_0)(a_2(x + x_0) + a_1).
\end{equation}
Let $\ell := \lc (k-j)/2 \rc$. By the pigeonhole principle, we must have $p^\ell \mid x - x_0$ or $p^\ell \mid a_2(x + x_0) + a_1$. Since $p \nmid a_2$, either option uniquely determines $x \pmod{p^\ell}$ in terms of $x_0$. So there can be at most
\[
    \ll \frac{p^{k-j}}{p^\ell} = p^{(k-j) - \lc \frac{k-j}{2} \rc} = p^{\lf \frac{k-j}{2} \rf}
\]
solutions in $x \pmod{p^{k-j}}$.

\textbf{Case 2:} $p \nmid a_1$. Given the previous case, we can assume $p \mid a_2$. Then for any two solutions $x_0, x$ of \cref{eq:quadr-mod-p}, we have $p \nmid a_2(x+x_0) + a_1$, so from \cref{eq:after-subtract-eqn} we find that $p^{k-j} \mid x-x_0$. This leaves at most one solution $x \pmod{p^{k-j}}$.

\textbf{Case 3:} $p \nmid a_0$. Then \cref{eq:quadr-mod-p} implies $p \nmid x$, and by substituting $x \leftrightarrow \bar{x} \pmod{p^{k-j}}$, we reduce to the case $p \nmid a_2$.

Finally, if $p = 2$, then the largest power of $2$ dividing $(s, r, t-q, 2^k)$ is either $2^j$ or $2^{j+1}$ (indeed, if $2^{j+2} \mid (s, r, t-q, 2^k)$, then $q^2 \equiv 1 \pmod{2^{j+2}}$ forces $q \equiv \pm 1 \pmod{2^{j+1}}$, so $g \in \pm \Gamma_{2^k}(2^{j+1})$, contradicting the definition of $j$). The rest of the proof goes through, possibly with $j$ replaced by $j+1$ (but this can only affect the final bound by an absolute constant factor when $p = 2$).
\end{proof}

\begin{remark}
The bound in \cref{eq:char-pwise-bound} is sharp in terms of $c$ and $d$, as can be seen by taking $g = T^d$. In particular, if $c = p^k$ is a prime power, then $g = T^{p^j}$ has $p^{\lf (k+j)/2 \rf}$ fixed points of the shape $[1 : x]$, given by the solutions to $p^j x^2 \equiv 0 \pmod{p^k}$.
\end{remark}

\begin{proposition}[Lower bound for squared character sums] \label{prop:squared-char-lb}
Let $c \in \Z_+$ and $\chi$ be an irreducible character inside $\chi_c^\circ$. Let $d, d', e \in \Z_+$ be such that $d' \mid d$, $(d, e) = 1$, and $c = d d' e$. Then one has
\[
    \sum_{n \in \Gamma_c(d)} |\chi(n)|^2
    \gg 
    \frac{c^{3-o(1)}}{d}.
\]
\end{proposition}

\begin{remark}
The lower bound in \cref{prop:squared-char-lb} wins a factor of about $d^2$ over the `trivial' bound of $|\Gamma_c(d)| = \tfrac{c^3}{d^3}$ due to \cref{eq:sum-square-char}. This is because $|\chi(n)|$ typically has size $\gtrsim d$ when $n \in \Gamma_c(d)$ (note that when $d = c$, one has $\Gamma_c(c) = \{I\}$ and $\chi(I) = \dim \chi$).
\end{remark}

The proof of \cref{prop:squared-char-lb} reduces to a local computation (i.e., for $c = p^k$ a prime power) given below, which builds on \cref{prop:prim-rep-large}. Indeed, one can rephrase \cref{prop:prim-rep-large} as follows: if $\chi \in \Irr(\SL_2(\Z/p^k\Z))$ is primitive, then $|\chi(I)|^2 \gg p^{2k}$.
Using a bit of Clifford theory, we can generalize this to averages of $|\chi|^2$ over the congruence subgroups $\Gamma_{p^k}(p^j)$, for some values of $j$.

\begin{lemma} \label{lem:squared-char-primitive}
Let $p^k$ be a prime power and $j \in \Z$ be such that either $j = 0$ or $\tfrac{k}{2} \le j \le k$. Let $\chi$ be a primitive irreducible character of $\SL_2(\Z/p^k\Z)$. Then 
\[
    \sum_{n \in \Gamma_{p^k}(p^j)} |\chi(n)|^2 \gg (k-j+1)^{-1} p^{3k-j}.
\]
\end{lemma}

\begin{proof}
Set $G := \SL_2(\Z/p^k\Z)$ and $N := \Gamma_{p^k}(p^j)$, so $N \triangleleft G$. Say $\chi = \Tr\, \rho$ where $\rho \in \hat G$ is primitive.
By \cref{eq:sum-square-char}, we have
\[
    \frac{1}{|N|} \sum_{n \in N} |\chi(n)|^2 
    = 
    \sum_{\rho_0 \in \hat N} \Mult(\rho_0, \rho\vert_N)^2.
\]
By \cref{lem:clifford}, $\rho\vert_N$ contains $L$ irreducible representations of $N$, each with multiplicity $m$, for some positive integers $L, m$. Thus
\[
    \sum_{\rho_0 \in \hat N} \Mult(\rho_0, \rho\vert_N)^2 = Lm^2,
\]
and in light of \cref{eq:size-nd}, it remains to show that
\begin{equation} \label{eq:Lm2-to-show}
    Lm^2 \gg (k-j+1)^{-1}p^{2j}.
\end{equation}
If $j = 0$, this is a trivial statement. Suppose now that $\tfrac{k}{2} \le j \le k$, so $N = \Gamma_{p^k}(p^j)$ is abelian by \cref{lem:abelian-subgroups}, and all of its irreducible representations are $1$-dimensional. 
By equating dimensions and using \cref{prop:prim-rep-large}, we find that $Lm = \dim \rho\vert_N = \dim \rho \gg p^k$, and therefore
\begin{equation} \label{eq:Lm-expr}
    Lm^2 \gg \frac{p^{2k}}{L}.
\end{equation}
So to prove \cref{eq:Lm2-to-show}, it suffices to establish a suitable upper bound for $L$. By the conclusion of \cref{lem:clifford}, all $L$ non-isomorphic representations in the decomposition of $\rho\vert_N$ lie in the same orbit of $G$'s action by conjugation. By the characterization of irreducible representations of $N$ from \cref{lem:abelian-subgroups}, it follows that $L$ is at most the maximal size of an orbit in the set
\[
    R^{2 \times 2}/RI,
    \qquad\qquad \text{where }
    R := \Z/p^{k-j}\Z,
\]
under conjugation by $\SL_2(\Z/p^k\Z)$, or equivalently by $\SL_2(R)$. 
Now for $B = \begin{psmall} a & b \\ c & d \end{psmall} \in R^{2 \times 2}$, the discriminant 
\[
    \Disc(B) := (a+d)^2 - 4(ad-bc) = (a-d)^2 + 4bc
\]
gives an invariant under both $\SL_2(R)$-conjugation and $RI$-translation. We therefore find that
\begin{equation} \label{eq:L-bound}
\begin{aligned}
    L &\le \max_{\Delta \in R} \#\left\{B + RI \in R^{2 \times 2}/RI : \Disc(B) = \Delta \right\}
    \\
    &=
    \max_{\Delta \in R} \#\left\{\begin{psmall} a & b \\ c & 0 \end{psmall} \in R^{2 \times 2} : a^2 + 4bc = \Delta \right\}.
\end{aligned}
\end{equation}
Now let $\Delta, a \in R$ and write $\Delta - a^2 = p^\ell z$ for some $0 \le \ell \le k-j$ and $z \in (\Z/p^{k-j-\ell}\Z)^\times$. The equation $4bc = p^\ell z$ in $\Z/p^{k-j}\Z$ then implies
\begin{equation} \label{eq:bc-prime}
    4b = p^{\ell_b} b', 
    \qquad c = p^{\ell_c} c',
    \qquad 
    b' c' \equiv z \pmod{p^{k-j-\ell}},
\end{equation}
for some $\ell_b, \ell_c \ge 0$ with $\ell_b + \ell_c = \ell$, and some $b' \in (\Z/p^{k-j-\ell_b}\Z)^\times$, $c' \in (\Z/p^{k-j-\ell_c}\Z)^\times$. 

Given $\Delta \in R$ and one of $p^{k-j}$ possible choices of $a$ in \cref{eq:L-bound}, $\ell$ and $z$ are determined. Then there are at most $\ell+1 \le k-j+1$ choices of $\ell_b, \ell_c$, and at most $p^{k-j-\ell_b}$ choices of $b'$. Given $\ell_b, \ell_c, b'$, there are at most $p^{k-j-\ell_c} / p^{k-j-\ell} = p^{\ell_b}$ choices of $c'$, since $c' \pmod{p^{k-j-\ell}}$ is fixed. Given $\ell_b, \ell_c, b', c'$, there are $O(1)$ choices of $b, c$ (there may be up to $4$ choices of $b$ if $p = 2$, due to the factor of $4$ in \cref{eq:bc-prime}) . Putting these counts together, we obtain that
\[
    L \ll p^{k-j} (k-j+1) \max_{\ell_b + \ell_c \le k} p^{k-j-\ell_b} p^{\ell_b}
    =
    (k-j+1) p^{2(k-j)}.
\]
Combining this with \cref{eq:Lm-expr} establishes the desired bound from \cref{eq:Lm2-to-show}.
\end{proof}

\begin{remark}
One may expect the bound in \cref{lem:squared-char-primitive} to be sharp up to $p^{o(k)}$ factors, and to actually hold for all $0 \le j \le k$. This might follow from a more careful study of $\hat\Gamma_{p^k}(p^j)$ for $1 \le j < \tfrac{k}{2}$, and it would imply \cref{prop:squared-char-lb} (as well as \cref{thm:bilinear-forms-composite}) for more flexible factorizations $c = de$.
\end{remark}

\begin{proof}[Proof of \cref{prop:squared-char-lb}]
By \cref{prop:sifted-rep}, $\rho_c^\circ$ decomposes as a direct sum of representations of the form 
\begin{equation} \label{eq:irred-subrep}
    \rho = \bigboxtimes_{p^k \| c} \rho_{p,k},
\end{equation}
where each $\rho_{p,k}$ is a primitive irreducible representation of $\SL_2(\Z/p^k\Z)$.
This gives the decomposition of $\rho_c^\circ$ into irreducible representations, so any irreducible representation occurring inside $\rho_c^\circ$ must be (isomorphic to one) of the form in \cref{eq:irred-subrep}. Now let $\chi := \Tr\, \rho$ and $\chi_{p,k} := \Tr\, \rho_{p,k}$ for such a representation. 
From \cref{eq:CRT-isom-groups} and fact that $\chi(n) = \prod_{p^k \| n} \chi_{p,k}(\pi_{c,p^k}(n))$ for $n \in \SL_2(\Z/c\Z)$, it follows that
\[
    \sum_{n \in \Gamma_c(d)} |\chi(n)|^2 =
    \prod_{\substack{p^k \| c \\ p^j \| d}}\, \sum_{n \in \Gamma_{p^k}(p^j)} |\chi_{p,k}(n)|^2.
\]
One can then apply \cref{lem:squared-char-primitive} to obtain 
\[
    \sum_{n \in \Gamma_c(d)} |\chi(n)|^2 \gg \prod_{p^k \| c} (k+1)^{-1} p^{3k-j},
\] 
Note that the hypothesis on $j$ from \cref{lem:squared-char-primitive} is satisfied because whenever $p$ is a prime dividing $c$ with $p^k \| c$ and $p^j \| d$, one of the following holds:
\begin{itemize} 
\item[$(i).$] One has $p \mid e$ and $p \nmid d$, so $j = 0$, or
\item[$(ii).$] One has $p \mid d$ and $p \nmid e$, so $k = v_p(dd') \le 2j$.
\end{itemize}
The desired conclusion then follows from the divisor bound.
\end{proof}

\subsection{Passing to a counting problem}
We can now state the result of our amplification argument.

\begin{proposition}[From Fourier coefficients to a counting problem] \label{prop:fourier-to-counting}
Let $c \in \Z_+$, $a \in (\Z/c\Z)^\times$, $H_1, H_2 \gg 1$, and $F_{c,a}^{H_1,H_2}$ be as in \cref{eq:func-H12}. Let $d, d', e \in \Z_+$ be such that $(d, e) = 1$, $d' \mid d$, and $c = dd'e$. Then for any even positive integer $q$, one has
\[
    \|\hat F_{c,a}^{H_1,H_2}(\rho_c^\circ)\|_{S^q}^q \ll
    \frac{c^{o(1)}d}{(H_1 H_2)^{q/2}} \max_{\substack{\tilde{d}, \tilde{f} \in \Z_+ \\ d \mid \tilde{d} \mid c \\ \tilde{f}^2 \mid c\tilde{d}}} \tilde{f} \sum_{\substack{h_1, \ldots, h_q \in \Z \\ |h_i| \le 2H_j \\ \forall i \equiv j \pmod{2}}} 
    \one_{T^{\bar{a}h_1}ST^{h_2}S \cdots T^{\bar{a}h_{q-1}}ST^{h_q}S = I \text{ in } \PSL_2(\Z/\tilde{d}\Z)}.
\]
\end{proposition}

\begin{proof}
By \cref{prop:sifted-rep}, $\rho_c^\circ$ is a sum of $c^{o(1)}$ irreducible representations $\rho$. By \cref{lem:break-down-schatten}, it suffices to prove the desired upper bound for each $\|\hat F_{c,a}^{H_1,H_2}(\rho)\|_{S^q}^q$; the loss factor of $c^{o(1)}$ is acceptable here. For each irreducible representation $\rho$ with $\Mult(\rho, \rho_c^\circ) > 0$, we apply \cref{prop:amplification,prop:squared-char-lb} with $F := F_{c,a}^{H_1,H_2}$ to obtain
\[
    \|\hat F(\rho)\|_{S^q}^q \ll 
    \frac{c^{3+o(1)}}{c^3/d} \sum_{\substack{g_1,\ldots,g_q \in \SL_2(\Z/c\Z) \\ g_1\cdots g_q \in \Gamma_c(d)}} F(g_1) \bar F(g_2^{-1}) \cdots F(g_{q-1}) \bar F(g_q^{-1}) \chi(g_1 \cdots g_q),
\]
where $\chi := \Tr\, \rho$. In fact, by \cref{prop:amplification}, the sum above is nonnegative if one replaces $\chi$ with any irreducible character of $\SL_2(\Z/c\Z)$. Summing over all such characters $\chi' = \Tr\, \rho'$ with weight $\Mult(\rho', \rho_c)$ (which is at least $1$ when $\rho' = \rho$), we find that
\[
    \|\hat F(\rho)\|_{S^q}^q \ll 
    c^{o(1)}d \sum_{\substack{g_1,\ldots,g_q \in \SL_2(\Z/c\Z) \\ g_1\cdots g_q \in \Gamma_c(d)}} F(g_1) \bar F(g_2^{-1}) \cdots F(g_{q-1}) \bar F(g_q^{-1}) \chi_c(g_1 \cdots g_q).
\]
Here $\rho_c$ is the original permutation representation from \cref{def:perm-rep}. In light of \cref{lem:perm-rep-pointwise-ub}, we ought to split the sum above based on the largest $\tilde{d} \mid c$ such that $g_1\cdots g_q \in Z(\SL_2(\Z/c\Z)) \cdot \Gamma_c(\tilde{d})$; note that by \cref{eq:def-PSL2}, this is equivalent to the equation $g_1\cdots g_q = I$ in $\PSL_2(\Z/\tilde{d}\Z)$. Then from the triangle inequality, \cref{lem:perm-rep-pointwise-ub}, and the divisor bound for $\tilde{d}$, we find that
\begin{equation} \label{eq:before-plug-in-FH12}
    \|\hat F(\rho)\|_{S^q}^q \ll c^{o(1)}d \max_{\substack{\tilde{d}, \tilde{f} \in \Z_+ \\ d \mid \tilde{d} \mid c \\ \tilde{f}^2 \mid c\tilde{d}}} \tilde{f} \sum_{g_1,\ldots,g_q \in \SL_2(\Z/c\Z)} |F(g_1) F(g_2^{-1}) \cdots F(g_{q-1}) F(g_q^{-1})| \one_{g_1 \cdots g_q = I \text{ in } \PSL_2(\Z/\tilde{d}\Z)}.
\end{equation}
Now recalling \cref{eq:func-H12}, we can expand 
\[
    |F(g)| = |F_{c,a}^{H_1,H_2}(g)| \ll \frac{1}{H_1H_2} \sum_{\substack{|h_1| \le H_1 \\ |h_2| \le H_2}} \one_{g = T^{\bar{a} h_1} S T^{h_2}}.
\]
The conclusion follows by plugging this into \cref{eq:before-plug-in-FH12}, and noting that as $h, h'$ vary in $[-H, H] \cap \Z$, the difference $h-h'$ varies in $[-2H, 2H] \cap \Z$, each value being attained $O(H)$ times.
\end{proof}

\section{Counting solutions in \texorpdfstring{$\PSL_2(\Z/c\Z)$}{PSL2(Z/cZ)}} \label{sec:counting}

We now develop the final ingredient towards \cref{thm:bilinear-forms-composite}, as outlined in \cref{subsec:outline-counting}. Given $c, q \in \Z_+$ with $q$ even, $a_1, a_2 \in (\Z/c\Z)^\times$, and $1 \le H_1 \le H_2 \ll c$, we will count solutions in $(h_1, \ldots, h_q) \in \Z$ to the system
\begin{equation} \label{eq:PSL2-eqn}
\begin{cases} 
    T^{a_1h_1} S T^{a_2h_2} S \cdots T^{a_1h_{q-1}} S T^{a_2h_q} S = I
    \text{ in } 
    \PSL_2(\Z/c\Z), 
    \\
    |h_i| \le H_j, \text{ for all } i, j \text{ with } i \equiv j \pmod{2}.
\end{cases}
\end{equation}
In particular, the ranges of $h_i$ alone produce the trivial bound $\ll_q (H_1 H_2)^{q/2}$ for the number of solutions. Focusing on the case $a_1 = a_2 = 1$, below are some classes of solutions to \cref{eq:PSL2-eqn}:
\begin{itemize}
    \item[$(i)$.] \emph{Integer solutions (1)}. If $h_1 = h_{\frac{q}{2} + 1} = 0$ (and similarly for cyclic permutations of this case), noting that $S^2 = I$ in $\PSL_2(\Z/c\Z)$, \cref{eq:PSL2-eqn} becomes
    \[
        T^{h_2} S \cdots S T^{h_{\frac{q}{2}}}
        =
        T^{-h_q} S T^{-h_{q-1}} S \cdots S T^{-h_{\frac{q}{2} + 2}}.
    \]
    This has $\asymp_q H_1^{\lf (q-2)/4 \rf} H_2^{\lc (q-2)/4 \rc}$ diagonal solutions with $h_i = h_{q-i+2}$, which actually give solutions to \cref{eq:PSL2-eqn} in $\PSL_2(\Z)$.
    \item[$(ii)$.] \emph{Integer solutions (2)}. If $h_1 = h_3 = \cdots = h_{q-1} = 0$, \cref{eq:PSL2-eqn} becomes
    \[
        T^{h_2 + h_4 + \cdots + h_q} = I
        \qquad 
        \iff 
        \qquad 
        h_2 + h_4 + \cdots + h_q \equiv 0 \pmod{c}.
    \]
    This has $\asymp_q H_2^{(q-2)/2}$ solutions, which supersedes the diagonal contribution from $(i)$.
    \item[$(iii)$.] \emph{Generic terms}. The product $T^{h_1} S T^{h_2} S \cdots T^{h_q} S$ can take $\approx c^3$ values in $\PSL_2(\Z/c\Z)$. If each matrix is attained roughly the same number of times (and such equidistribution ought to happen for large enough $q$), this gives an expected number of $\approx c^{-3} (H_1 H_2)^{q/2}$ solutions.
\end{itemize}
These heuristics lead to the following conjectural upper bound.

\begin{conjecture} \label{conj:counting}
Let $c, q \in \Z_+$ with $q$ even. For all $a_1, a_2 \in (\Z/c\Z)^\times$ and $1 \le H_1 \le H_2 \ll c$, the number of solutions $(h_1, \ldots, h_q) \in \Z^q$ to \cref{eq:PSL2-eqn} is at most
\begin{equation} \label{eq:conj-counting}
    \ll_q c^{o(1)} \left(H_2^{(q-2)/2} + \frac{(H_1 H_2)^{q/2}}{c^3} \right).
\end{equation}
\end{conjecture}

\begin{remark}
It is straightforward to prove \cref{conj:counting} when $q \in \{2, 4\}$; see \cref{lem:counting-4}.
The corresponding lower bound (without the $c^{o(1)}$ factor) can be established unconditionally; see \cref{lem:counting-lb}.
\end{remark}

\begin{remark}
When $c$ is large enough in terms of $H_1, H_2, q$ (so in particular, the first term in \cref{eq:conj-counting} dominates), \cref{conj:counting} becomes a statement about $\PSL_2(\Z)$, which can be established; see e.g. \cite[Lemma 30]{shkredov2021modular}. When $q$ is large enough in terms of $H_1, H_2, c$ (so in particular, the second term in \cref{eq:conj-counting} dominates), \cref{conj:counting} can be attacked using $L^2$-flattening methods; see \cite[Lemma 53 and Theorem 50]{shkredov2018asymptotic} for the case when $c$ is prime. However, note that \cref{eq:conj-counting} saves at best $c^3$ over the trivial bound of $(H_1 H_2)^{q/2}$, and this saving becomes $c^{3/q}$ in our final bounds; therefore, using a large value of $q$ ultimately produces a quantitatively-weak power saving. On the other hand, if $q$ is too small, then combining \cref{conj:counting} with \cref{prop:fourier-to-counting} gives information about a small moment of singular values, which produces a weak bound for the top singular value.

Because of this, \cref{conj:counting} is most relevant in the median range when $q \asymp 1$, say $q \in \{6, 8, 10\}$, and $H_1, H_2 \in [\sqrt{c}, c]$. It seems very difficult to fully establish \cref{conj:counting} in these cases, but we can nevertheless make some partial progress towards it when $q = 6$. When $c$ is prime, the cases $q \in \{4, 6\}$ are also related to some computations of Shkredov \cite[Lemma 15]{shkredov2021modular}.
\end{remark}

\begin{proposition}[Combinatorial count for $q = 6$] \label{prop:counting-6}
Let $c \in \Z_+$, $a_1, a_2 \in (\Z/c\Z)^\times$, and $1 \le H_1 \le H_2 \ll c$. The number of solutions $(h_1, \ldots, h_6) \in \Z^6$ to \cref{eq:PSL2-eqn} with $q = 6$ is at most
\begin{equation} \label{eq:counting-6}
    \ll c^{o(1)} \left(H_2^2 + \frac{(H_1 H_2)^2}{c} \right).
\end{equation}
\end{proposition}

\begin{remark}
\cref{conj:counting} would replace the second term in \cref{eq:counting-6} with $\tfrac{(H_1 H_2)^3}{c^3}$. In particular, \cref{prop:counting-6} establishes \cref{conj:counting} when $q = 6$ and either $H_1^2 \ll c$ or $H_1, H_2 \asymp c$.
\end{remark}

\begin{proof}[Proof of \cref{prop:counting-6}]
To simplify the exposition, we focus on the case $a_1 = a_2 = 1$. The proof is almost completely unchanged when $a_1, a_2 \in (\Z/c\Z)^\times$ are arbitrary.
We may then write the equation in \cref{eq:PSL2-eqn} (with $q = 6$) as
\[
    T^{h_1} S T^{h_2} S T^{h_3} S T^{h_4} S = S T^{-h_6} S T^{-h_5} \text{ in } \PSL_2(\Z/c\Z).
\]
In light of \cref{eq:def-PSL2,eq:centers},
a short computation brings this to the entry-wise congruence
\[
    \begin{pmatrix} 
    h_1 h_2 h_3 h_4 - h_1 h_4 - h_3 h_4 - h_1 h_2 + 1
    & 
    - h_1 h_2 h_3 + h_1 + h_3
    \\
    h_2 h_3 h_4 - h_2 - h_4 
    & 
    - h_2 h_3 + 1
    \end{pmatrix}
    \equiv
    \gamma
    \begin{pmatrix} -1 & h_5 \\ - h_6 & h_5 h_6 - 1 \end{pmatrix}
    \pmod{c}
\]
for some $\gamma \in \Z/c\Z$ with $\gamma^2 = 1$. Since there are $c^{o(1)}$ possible values of $\gamma$ by \cref{eq:center-bound}, we may as well regard $\gamma$ as fixed. This implies the system of three congruences
\begin{equation} \label{eq:system-cong}
    \begin{cases}
        1 - h_2 h_3 \equiv \gamma (h_5 h_6 - 1) \pmod{c}, \\
        h_1(1 - h_2 h_3) + h_3 \equiv \gamma h_5 \pmod{c}, \\
        h_4(1 - h_2 h_3) + h_2 \equiv \gamma h_6 \pmod{c}.
    \end{cases}
\end{equation}
Our argument now requires some casework.

\textbf{Case 1:} One has $h_j = 0$ for some $j \in \{1, \ldots, 6\}$. Since the original equation can also be written as $T^{h_{j-1}} S T^{h_{j}} S T^{h_{j+1}} S T^{h_{j+2}} S T^{h_{j+3}} S T^{h_{j+4}} S = I$ in $\PSL_2(\Z/c\Z)$ (viewing indices modulo $6$), we may assume without loss of generality that $j = 2$, up to potentially swapping $H_1$ and $H_2$ in the final bound (so we momentarily forget that $H_1 \le H_2$). So let us say $h_2 = 0$, which reduces \cref{eq:system-cong} to
\begin{equation} \label{eq:system-cong-case-1}
    \begin{cases}
        1 \equiv \gamma (h_5 h_6 - 1) \pmod{c}, \\
        h_1 + h_3 \equiv \gamma h_5 \pmod{c}, \\
        h_4 \equiv \gamma h_6 \pmod{c}.
    \end{cases}
\end{equation}
\emph{Subcase 1.1:} One has $h_5 = 0$. Then for any values of $h_1$ and $h_6$, the system in \cref{eq:system-cong-case-1} leaves only $O(1)$ possibilities for $h_3, h_4$ (since $H_1, H_2 \ll c$). This gives $O(H_1 H_2)$ solutions.

\emph{Subcase 1.2:} One has $h_6 = 0$. Then for any values of $h_1$ and $h_5$, the system in \cref{eq:system-cong-case-1} leaves only $O(1)$ possibilities for $h_3, h_4$. This gives $O(H_1^2)$ solutions.

\emph{Subcase 1.3:} One has $h_5 h_6 \neq 0$. Then the first congruence in \cref{eq:system-cong-case-1} fixes $h_5h_6 \pmod{c}$, leaving $1 + \tfrac{H_1H_2}{c}$ possibilities for the nonzero integer $h_5 h_6$, each of which gives $O(c^{o(1)})$ possible values for $h_5, h_6$ by the divisor bound. Once $h_5$ and $h_6$ are fixed, each value of $h_1$ produces $O(1)$ final solutions. This gives $(1 + \tfrac{H_1H_2}{c})H_1$ solutions.

From Case 1, we obtain a total number of solutions of
\[
    \ll c^{o(1)} \left(H_1^2 + H_2^2 + \left(1 + \frac{H_1 H_2}{c}\right)(H_1 + H_2)\right),
\]
which is $O(c^{o(1)} H_2^2)$ once we remember that $H_1 \le H_2$ and $H_1 \ll c$. This is acceptable in \cref{eq:counting-6}.

\textbf{Case 2:} One has $h_j \neq 0$ for all $j \in \{1, \ldots, 6\}$. We fix $d := (1 - h_2h_3, c) = (h_5h_6 - 1, c)$ up to an acceptable $O(c^{o(1)})$ loss; note that $(h_2, d) = 1$. Since $h_5 h_6 \equiv 1 \pmod{d}$, there are $O(1 + \frac{H_1 H_2}{d})$ ways to pick the nonzero integer $h_5 h_6$, each of which gives $O(c^{o(1)})$ ways to pick $h_5, h_6$.

Once $h_5, h_6$ are fixed, we pick $h_2, h_3$ subject to the system
\begin{equation} \label{eq:system-cong-case-2}
\begin{cases}
    1 - h_2 h_3 \equiv \gamma(h_5h_6 - 1) \pmod{c}, \\ 
    h_2 \equiv \gamma h_6 =: r \pmod{d},
\end{cases}
\end{equation}
which follows from \cref{eq:system-cong}; note that $r$ is also fixed at this point. We can do this in two ways:
\begin{itemize}
    \item Pick the nonzero integer $h_2 h_3$ subject to its residue mod $c$ (due to the first congruence in \cref{eq:system-cong-case-2}) in $O(1 + \tfrac{H_1H_2}{c})$ ways, and then $h_2, h_3$ in $O(c^{o(1)})$ ways by the divisor bound.
    \item For each choice of $h_2$ with $|h_2| \le H_2$ and $h_2 \equiv r\pmod{d}$, pick $h_3$ subject to its residue mod $\tfrac{c}{(h_2, c)}$ (again, due to the first congruence in \cref{eq:system-cong-case-2}) in $O(1 + \frac{H_1}{c} (h_2, c))$ ways.
\end{itemize}

Finally, once $h_5, h_6, h_2, h_3$ are fixed, \cref{eq:system-cong} determines the residues of $h_1$ and $h_4$ modulo $\tfrac{c}{d}$, so there are $(1 + \tfrac{H_1}{c}d)(1 + \tfrac{H_2}{c}d)$ choices of $h_1, h_4$. From Case 2, we obtain a total number of solutions of
\begin{equation} \label{eq:case-2-final-count}
\begin{aligned}
    \ll c^{o(1)} \max_{d \mid c} \underbrace{\left(1 + \frac{H_1 H_2}{d}\right)}_{\text{from picking } h_5, h_6} \underbrace{\min\Bigg(1 + \frac{H_1H_2}{c}, \max_{r \in (\Z/d\Z)^\times} 
    \sum_{\substack{|h_2| \le H_2 \\ h_2 \equiv r \pmod{d}}}
    \left(1 + \frac{H_1}{c}(h_2, c) \right) \Bigg)}_{\text{from picking } h_2, h_3}
    \\ 
    \times \underbrace{\left(1 + \frac{H_1}{c}d\right)}_{\text{from picking } h_1}\
    \underbrace{\left(1 + \frac{H_2}{c}d\right)}_{\text{from picking } h_4}.
\end{aligned}
\end{equation}
To bound the sum over $h_2$, we note that $(h_2, d) = 1$ implies $(g, d) = 1$ for any $g \mid h_2$, and we write
\[
\begin{aligned}
    \sum_{\substack{|h_2| \le H_2 \\ h_2 \equiv r \pmod{d}}} (h_2, c)
    &\le 
    \sum_{\substack{g \mid c \\ (g, d) = 1}} g \sum_{\substack{|h_2| \le H_2 \\ h_2 \equiv r \pmod{d} \\ h_2 \equiv 0 \pmod{g}}} 1
    \\
    &=
    \sum_{\substack{g \mid \frac{c}{d} \\ (g, d) = 1}} g \sum_{\substack{|h_2'| \le \frac{H_2}{g} \\ h_2' \equiv \bar{g}r \pmod{d}}} 1
    \ll 
    \sum_{\substack{|g| \le \frac{c}{d} \\ (g, d) = 1}} g \left(1 + \frac{H_2}{gd}\right)
    \ll 
    c^{o(1)} \left(\frac{c}{d} + \frac{H_2}{d}\right)
    \ll 
    \frac{c^{1+o(1)}}{d}.
\end{aligned}
\]
Plugging this into \cref{eq:case-2-final-count} gives a total count of
\[
    \ll c^{o(1)} \max_{d \mid c}
    \left(1 + \frac{H_1 H_2}{d}\right) \left(1 + \frac{H_2}{c} d\right) \left(1 + \frac{H_1}{c}d\right)
    \min\left(1 + \frac{H_1 H_2}{c}, 1 + \frac{H_2}{d} + \frac{H_1}{c} \cdot \frac{c}{d}\right)
\]
Since $H_1 \le H_2$, the final term of $\tfrac{H_1}{c} \cdot \tfrac{c}{d} = \tfrac{H_1}{d}$ can be omitted. The bound above now becomes
\[
\begin{aligned}
    &\ll c^{o(1)} \max_{d \mid c}
    \left(1 + \frac{H_1 H_2}{d}\right) \left(1 + \frac{H_2}{c}d\right) \max\left(\frac{H_1}{c} d, 1\right)
    \left(1 + \frac{H_2}{d}\min\left(\frac{H_1}{c}d, 1\right)\right)
    \\
    &= c^{o(1)} \max_{d \mid c}
    \left(1 + \frac{H_1 H_2}{d}\right) \left(1 + \frac{H_2}{c}d\right)
    \left(\max\left(\frac{H_1}{c} d, 1\right) + \frac{H_2}{d} \left(\frac{H_1}{c}d \cdot 1\right)\right)
    \\
    &\le
    c^{o(1)} \max_{d \mid c}
    \left(1 + \frac{H_1 H_2}{d}\right) \left(1 + \frac{H_2}{c}d\right)
    \left(1 + \frac{H_1}{c} d + \frac{H_1 H_2}{c}\right).
\end{aligned}
\]
After expanding the expression inside the maximum, each term is either strictly increasing, constant, or strictly decreasing in $d \in [1, c]$, so each term is maximized when $d = 1$ or $d = c$. It follows that we can bound the maximum over $d \mid c$, up to a constant, by looking only at the extreme points $d = 1$ and $d = c$. This gives a total count of 
\[
\begin{aligned}
    \ll c^{o(1)}
    \max\left( 
    H_1 H_2 \left(1 + \frac{H_2}{c}\right)
    \left(1 + \frac{H_1}{c} + \frac{H_1 H_2}{c}\right),
    \left(1 + \frac{H_1 H_2}{c}\right) H_2
    \left(H_1 + \frac{H_1 H_2}{c}\right)
    \right)
    \\
    \ll c^{o(1)}
    \max\left( 
    H_1 H_2
    \left(1 + \frac{H_1 H_2}{c}\right),
    \left(1 + \frac{H_1 H_2}{c}\right) H_2 H_1
    \right),
\end{aligned}
\]
where, to reach the last line, we used that $H_1, H_2 \ll c$. This establishes the desired bound.
\end{proof}

We also remove the restriction that $H_1, H_2 \ll c$ up to some additional factors.

\begin{corollary} \label{cor:counting-remove-ub-H12}
Let $c \in \Z_+$, $a_1, a_2 \in (\Z/c\Z)^\times$, and $1 \le H_1 \le H_2$. Then the number of solutions to \cref{eq:PSL2-eqn} with $q = 6$ is at most
\begin{equation} \label{eq:counting-6-refined}
    \ll c^{o(1)} \left( 1 + \frac{H_1 H_2}{c} + \frac{H_1^3 H_2}{c^3} \right) H_2^2.
\end{equation}
\end{corollary}

\begin{proof}
Since the equation in \cref{eq:PSL2-eqn} only depends on the residues of $h_1, \ldots, h_q \pmod{c}$, we may as well count solutions to the system
\begin{equation} \label{eq:PSL2-eqn-mod}
\begin{cases} 
    T^{a_1h_1'} S T^{a_2h_2'} S \cdots T^{a_1h_{q-1}'} S T^{a_2h_q'} S = I
    \text{ in } 
    \PSL_2(\Z/c\Z), 
    \\
    |h_i'| \le \min(H_j, c), \text{ for all } i, j \text{ with } i \equiv j \pmod{2},
\end{cases}
\end{equation}
and multiply the final count by a factor of $\ll_q (1 + \tfrac{H_1}{c})^{q/2}(1 + \tfrac{H_2}{c})^{q/2}$ (indeed, each solution $(h_1', \ldots, h_q')$ to \cref{eq:PSL2-eqn-mod} induces at most this many solutions $(h_1, \ldots, h_q)$ to \cref{eq:PSL2-eqn} with the same residues modulo $c$, and all solutions to \cref{eq:PSL2-eqn} can be obtained this way).

Taking $q = 6$, we apply \cref{prop:counting-6} for $\min(H_1, c)$ and $\min(H_2, c)$ to obtain a total count of
\[
\begin{aligned}
    &\ll c^{o(1)} \left(1 + \frac{H_1}{c}\right)^3 \left(1 + \frac{H_2}{c}\right)^3 \left(\min(H_2, c)^2 + \frac{\min(H_1,c)^2 \min(H_2,c)^2}{c} \right)
    \\
    &\ll c^{o(1)} \left(1 + \frac{H_1^3}{c^3}\right) \left(1 + \frac{H_2}{c}\right) H_2^2 + 
    c^{o(1)}\left(1 + \frac{H_1}{c}\right) \left(1 + \frac{H_2}{c}\right) \frac{H_1^2 H_2^2}{c}
    \\
    &\ll c^{o(1)} \left(1 + \frac{H_2}{c} + \frac{H_1^3}{c^3} + \frac{H_1^3 H_2}{c^4}\right)H_2^2 + 
    c^{o(1)} \left(1 + \frac{H_2}{c} + \frac{H_1 H_2}{c^2}\right) \frac{H_1^2 H_2^2}{c}.
\end{aligned}
\]
We note that the third and fourth terms in the first parenthesis above can be ignored: their contribution to the final bound is $\tfrac{H_1^3 H_2^2}{c^3} + \tfrac{H_1^3 H_2^3}{c^4}$, which is superseded by the contribution of $\tfrac{H_1^3 H_2^3}{c^3}$ from the third term in the second parenthesis. This gives a total count of
\begin{equation} \label{eq:counting-6-refined-2}
    \ll c^{o(1)} \left( 1 + \frac{H_2}{c} + \frac{H_1^2}{c} + \frac{H_1^2 H_2}{c^2} + \frac{H_1^3 H_2}{c^3} \right) H_2^2.
\end{equation}
This is bounded by \cref{eq:counting-6-refined}, noting that $\tfrac{H_1^2H_2}{c^2}$ is the geometric mean of $\tfrac{H_1H_2}{c}$ and $\tfrac{H_1^3 H_2}{c^3}$. 
\end{proof}

\section{Bilinear forms with Kloosterman sums} \label{sec:bil-forms}

We now combine the work in \cref{sec:rep-and-kloost,sec:amplif,sec:counting}, to deduce our main results from \cref{thm:bilinear-forms-composite,thm:bilinear-forms-general}.

\subsection{Composite moduli} \label{subsec:comp-moduli}
Here we prove a generalization of \cref{thm:bilinear-forms-composite}, which allows for larger values of $M, N$. We state our upper bound in two ways, to facilitate comparison with \cref{eq:trivial-bound}.

\begin{theorem} \label{thm:MN-bilinear-forms-composite}
Let $c = dd'e$ for some $d, d', e \in \Z_+$ with $d' \mid d$ and $(d, e) = 1$, and $f \le \sqrt{cd}$ be the largest integer with $f^2 \mid cd$. Let $\mI, \mJ \subset \Z$ be intervals of lengths $|\mI| = M$, $|\mJ| = N$, with\footnote{The assumption that $N \le M$ is only included to shorten the statement of \cref{thm:MN-bilinear-forms-composite}; one can of course swap $m$ and $n$ in the bilinear sum, up to swapping $M$ and $N$ in the upper bound.} $1 \le N \le M \le c$. Then for any complex sequences $(\alpha_m)_{m \in \mI}$, $(\beta_n)_{n \in \mJ}$ and $a \in (\Z/c\Z)^\times$, one has
\[
\begin{aligned}
    \mathop{\sum\sum}_{\substack{m \in \mI, n \in \mJ \\ (m, n, c) = 1}} \alpha_m \beta_n S(am, n; c) 
    &\ll
    \|\alpha\| \|\beta\| c^{1+o(1)}
    \left(\frac{dM^3 N}{c^3} + \frac{fM^2}{c^2} + \frac{f}{d^2} \right)^{\frac{1}{6}}
    \\
    &= 
    \|\alpha\| \|\beta\| c^{o(1)} \sqrt{MNc} \left(\frac{d}{N^2} + \frac{fc}{M N^3} + \frac{f c^3}{d^2 M^3 N^3} \right)^{\frac{1}{6}}.
\end{aligned}
\]
\end{theorem}

\begin{example} \label{ex:best-case-range}
Suppose $M \asymp N$ and let $\eps > 0$ be small.
For suitable factorizations of $c$, \cref{thm:MN-bilinear-forms-composite} can beat the Weil bound in \cref{eq:trivial-bound} for $N$ as small as $c^{2/5 + \eps}$ (this is attained, e.g., when $c = pq$ for distinct primes $p, q$ with $p \asymp q^{3/2}$), and it can beat the second bound in \cref{eq:trivial-bound} for $N$ as large as $c^{3/4 - \eps}$ (attained when $c$ has a small divisor $d$ such that $c/d$ is square-free).
\end{example}

\begin{remark} 
Additional savings are possible in \cref{thm:MN-bilinear-forms-composite} in the unbalanced range $M > N$. Firstly, the bound in \cref{eq:counting-6-refined} can be refined to \cref{eq:counting-6-refined-2}, but we omit this optimization for the sake of simplicity. Secondly, bounding the largest singular value of an $M \times N$ matrix by the sixth moment of its singular values (as we do) can be particularly lossy if $M > N$, since then the singular values often exhibit concentration near their maximum; one can try to amend this by subtracting a suitable main term from the sixth moment, as in \cite[Lemma 4.2]{guth2026new}.
\end{remark}

Before proving \cref{thm:MN-bilinear-forms-composite}, we need the following quick fact.

\begin{lemma} \label{lem:tilde-bound}
Let $c, d \in \Z_+$ with $d \mid c$, and $f$ be the maximal positive integer with $f^2 \mid cd$. For any nonnegative integer $k$, one has 
\[
    \max_{\substack{\tilde{d}, \tilde{f} \in \Z_+ \\ d \mid \tilde{d} \mid c \\ \tilde{f}^2 \mid c\tilde{d}}} \frac{\tilde{f}}{\tilde{d}^k} = 
    \begin{cases} 
    \frac{f}{d^k}, & k \ge 1, \\
    c, & k = 0.
    \end{cases}
\]
\end{lemma}
\begin{proof}
When one appends a prime $p$ to $\tilde{d}$, the numerator $\tilde{f}$ can increase by at most $p$, so the expression $\tilde{f}/\tilde{d}^k$ cannot increase if $k \ge 1$; the maximum is attained when $\tilde{d} = d$. On the other hand, if $k = 0$, then clearly $\tilde{f} \le \sqrt{c \cdot c} = c$, and the maximum is attained when $\tilde{d} = c$.
\end{proof}

\begin{proof}[Proof of \cref{thm:MN-bilinear-forms-composite}]
Let us write $c = c_1c_2$ where $c_1$ is square-free, $c_2$ is square-full, and $(c_1, c_2) = 1$. We first work with the modified bilinear form
\[
    \sum_{\substack{m \in \mI \\ n \in \mJ}}
    \alpha_m \beta_n S(am, n; c) \nu_{(m,n,c_1)} \one_{(m,n,c_2) = 1},
\]
where $\nu_g := \prod_{\text{prime }p \mid g} \frac{-1}{p^2-1}$ for $g \in \Z_+$, as in \cref{cor:kloosterman-ap}. We will remove these extraneous weights at the end of the proof. 

Let $\eps > 0$ and $H_1 := c^{1+\eps} M^{-1}$, $H_2 := c^{1+\eps} N^{-1}$. Since $M \ge N$, we have $H_1 \le H_2$. Let $q$ be an even positive integer.
By the characterization of operator norms from \cref{eq:operator-norm}, \cref{cor:kloosterman-ap} (using the notation from \cref{eq:kl-mat-ap,eq:func-H12}), the fact that $\|A\| \le \|A\|_{S^q}$ for any linear map $A$, and then \cref{prop:fourier-to-counting}, we have that
\begin{equation} \label{eq:kloosterman-first-steps}
\begin{aligned}
    \left\vert \sum_{\substack{m \in \mI \\ n \in \mJ}}
    \alpha_m \beta_n S(am, n; c) \nu_{(m,n,c_1)} \one_{(m,n,c_2) = 1} \right\vert
    &\le 
    \|\alpha\| \|\beta\| \|K_{c,a}^{\mI,\mJ}\|
    \\
    &\le 
    \|\alpha\| \|\beta\| \left(c^{1+2\eps} \|\hat F_{c,a}^{H_1,H_2}(\rho_c^\circ)\| + O_\eps(c^{-100})\right)
    \\ 
    &\ll_\eps 
    \|\alpha\| \|\beta\| \left(c^{1+3\eps} \frac{d^{1/q}}{(H_1 H_2)^{1/2}} \mathscr{S}^{1/q} + O_\eps(c^{-100})\right),
\end{aligned}
\end{equation}
where 
\[
    \mathscr{S} := \max_{\substack{\tilde{d}, \tilde{f} \in \Z_+ \\ d \mid \tilde{d} \mid c \\ \tilde f^2 \mid c\tilde{d}}} \tilde f \sum_{\substack{h_1, \ldots, h_q \in \Z \\ |h_i| \le 2H_j \\ \forall i \equiv j \pmod{2}}} 
    \one_{T^{\bar{a}h_1}ST^{h_2}S \cdots T^{\bar{a}h_{q-1}}ST^{h_q}S = I \text{ in } \PSL_2(\Z/\tilde{d}\Z)}.
\]
Note that the inner sum is a count of solutions to an equation of type \cref{eq:PSL2-eqn}, with $c$ replaced by $\tilde{d}$. 

Now set $q = 6$. By \cref{lem:tilde-bound,cor:counting-remove-ub-H12}, we obtain that
\[
    \mathscr{S} 
    \ll_\eps 
    c^\eps \max_{\substack{\tilde{d}, \tilde{f} \in \Z_+ \\ d \mid \tilde{d} \mid c \\ \tilde f^2 \mid c\tilde{d}}} \tilde{f} \left( 1 + \frac{H_1 H_2}{\tilde{d}} + \frac{H_1^3 H_2}{\tilde{d}^3} \right) H_2^2
    \le
    c^\eps \left(c + \frac{fH_1 H_2}{d} + \frac{fH_1^3 H_2}{d^3} \right) H_2^2.
\]
Plugging this into \cref{eq:kloosterman-first-steps}, we obtain 
\[
\begin{aligned}
    \mathop{\sum\sum}_{\substack{m \in \mI, n \in \mJ \\ (m, n, c) = 1}} \alpha_m \beta_n S(am, n; c) &\ll_\eps 
    \|\alpha\| \|\beta\| c^{1+4\eps} \frac{d^{1/6}}{(H_1 H_2)^{1/2}} \left(c H_2^2 + \frac{fH_1 H_2^3}{d} + \frac{fH_1^3 H_2^3}{d^3} \right)^{1/6}.
    \\
    &=
    \|\alpha\| \|\beta\| c^{1+4\eps} \left(\frac{cd}{H_1^3 H_2} + \frac{f}{H_1^2} + \frac{f}{d^2} \right)^{1/6}.
\end{aligned}
\]
Recalling that $H_1 = c^{1+\eps} M^{-1}$, $H_2 = c^{1+\eps} N^{-1}$, and that $\eps > 0$ was arbitrary, we conclude that
\begin{equation} \label{eq:almost-final-bound}
    \sum_{\substack{m \in \mI \\ n \in \mJ}}
    \alpha_m \beta_n S(am, n; c) \nu_{(m,n,c_1)} \one_{(m,n,c_2) = 1}
     \ll 
    \|\alpha\| \|\beta\| c^{1+o(1)}
    \left(\frac{dM^3 N}{c^3} + \frac{fM^2}{c^2} + \frac{f}{d^2} \right)^{\frac{1}{6}}.
\end{equation}
This is almost in the required form, except that we would like to replace the weights $\nu_g = \prod_{\text{prime }p \mid g} \frac{-1}{p^2-1}$ with $\one_{g = 1}$. To this end, for any square-free $n \in \Z_+$ we expand
\[
    \one_{n = 1} = \sum_{g \mid n} \nu_{n/g}\, f(g),  
    \qquad\quad 
    f(g) := \prod_{\text{prime }p \mid g} \frac{1}{p^2-1} \asymp g^{-2}.
\]
Indeed, by multiplicativity it suffices to verify this identity at primes, which is immediate. Using that $S(am, n; c) = \tfrac{\phi(c)}{\phi(c/g)} S(a\tfrac{m}{g}, \tfrac{n}{g}; \tfrac{c}{g})$, we may therefore write
\begin{equation} \label{eq:bil-kl-dec-g}
\begin{aligned}
    \mathop{\sum\sum}_{\substack{m \in \mI, n \in \mJ \\ (m, n, c) = 1}} \alpha_m \beta_n S(am, n; c)
    &=
    \sum_{\substack{m \in \mI \\ n \in \mJ}}
    \alpha_m \beta_n S(am, n; c)  \one_{(m,n,c_2) = 1}
    \sum_{g \mid (m,n,c_1)} f(g) \nu_{(\frac{m}{g},\frac{n}{g},\frac{c_1}{g})}
    \\
    &=
    \sum_{g \mid c_1} f(g) \frac{\phi(c)}{\phi(c/g)} 
    \mathop{\sum\sum}_{\substack{m \in \mI, n \in \mJ \\ g \mid (m, n)}}
    \alpha_m \beta_n S(a\tfrac{m}{g}, \tfrac{n}{g}; \tfrac{c}{g}) \nu_{(\frac{m}{g},\frac{n}{g},\frac{c_1}{g})} \one_{(\frac{m}{g},\frac{n}{g},c_2) = 1}
    \\
    &\ll 
    c^{o(1)}
    \max_{g \mid c_1} \frac{1}{g} \left\vert \mathop{\sum\sum}_{\substack{m \in \mI, n \in \mJ \\ g \mid (m, n)}}
    \alpha_m \beta_n S(a\tfrac{m}{g}, \tfrac{n}{g}; \tfrac{c}{g}) \nu_{(\frac{m}{g},\frac{n}{g},\frac{c_1}{g})} \one_{(\frac{m}{g},\frac{n}{g},c_2) = 1} \right\vert.
\end{aligned}
\end{equation}
The last sum is a bilinear form with Kloosterman sums with entries in the intervals
\[
    \mI_g := \{m' \in \Z : gm' \in \mI\},
    \qquad\qquad 
    \mJ_g := \{n' \in \Z : gn' \in \mJ\},
\]
and modulus $\tfrac{c}{g}$. This modulus has a factorization given by
\[
    \frac{c}{g} = d_g d'_g e_g,
    \qquad 
    d_g := \frac{d}{(d, \frac{g}{(d',g)})} = \frac{d(d', g)}{(dd', g)},
    \qquad 
    d'_g := \frac{d'}{(d', g)},
    \qquad 
    e_g := \frac{e}{(e, g)},
\]
which still satisfies $d'_g \mid d_g$ and $(d_g, e) = 1$. Note that $\frac{d}{g} \le d_g \le d$. Furthermore, the greatest integer whose square divides $\tfrac{c}{g} d_g$, call it $f_g$, satisfies $f_g \le f$. 
We can therefore apply our bound from \cref{eq:almost-final-bound} with the choice of parameters
\[
    (c, c_1, c_2, d, d', e, f, \alpha_m, \beta_n, \mI, \mJ)
    \gets 
    (\tfrac{c}{g}, \tfrac{c_1}{g}, c_2, d_g, d'_g, e_g, f_g, \alpha_{gm'}, \beta_{gn'}, \mI_g, \mJ_g),
\]
together with the crude bounds $|\mI_g| \le |\mI| = M$, $|\mJ_g| \le |\mJ| = N$,
to obtain 
\[
\begin{aligned}
    \frac{1}{g} \sum_{\substack{m' \in \mI_g \\ n' \in \mJ_g}}
    \alpha_{gm'} \beta_{gn'} &S(am', n'; \tfrac{c}{g}) \nu_{(m',n',\frac{c_1}{g})} \one_{(m',n',c_2) = 1} 
    \\ 
    &\ll 
    \|\alpha\|\|\beta\| \frac{c^{1+o(1)}}{g^2} \left(\frac{dM^3N}{(c/g)^3} + \frac{fM^2}{(c/g)^2} + \frac{f}{(d/g)^2} \right)^{\frac{1}{6}}.
\end{aligned}
\]
(If one of the intervals $\mI_g$ or $\mJ_g$ is empty, this bound still holds trivially since the left-hand side vanishes.)
The last expression is decreasing in $g$, so combining this with \cref{eq:bil-kl-dec-g} completes our proof.
\end{proof}

\begin{remark}
Using \cref{eq:counting-4-refined} with $q = 4$ instead of \cref{eq:counting-6-refined} with $q = 6$ in the proof above leads to a final bound of
\[
\begin{aligned}
    \mathop{\sum\sum}_{\substack{m \in \mI, n \in \mJ \\ (m, n, c) = 1}} \alpha_m \beta_n S(am, n; c)
    &\ll
    c^{o(1)} \|\alpha\| \|\beta\| c
    \left(\frac{d M^2 N}{c^2} + \frac{f M^2}{c^2} + \frac{fN}{dc} + \frac{f}{d^2} \right)^{1/4}
    \\
    &=
    c^{o(1)} \|\alpha\| \|\beta\| \sqrt{MNc}
    \left(\frac{d}{N} + \frac{f}{N^2} + \frac{fc}{dM^2N} + \frac{fc^2}{d^2M^2N^2} \right)^{1/4},
\end{aligned}
\]
which is weaker than \cref{thm:MN-bilinear-forms-composite} in the main ranges of interest.
\end{remark}

\begin{proof}[Proof of \cref{thm:bilinear-forms-composite}]
Note that the result holds trivially if $c = O(1)$. Since $M, N \ll c^{1/2+o(1)}$ and the result is symmetric in $M, N$, we can assume without loss of generality that $N \le M \le c$. One can then apply \cref{thm:MN-bilinear-forms-composite}, and since $M, N \ll c^{1/2+o(1)}$, the upper bound becomes
\[
    \|\alpha\| \|\beta\| c^{1+o(1)}
    \left(\frac{d}{c} + \frac{f}{c} + \frac{f}{d^2} \right)^{\frac{1}{6}}.
\]
The first term can be omitted in light of the bound $d \le f$ (since $d^2 \mid cd$).
\end{proof}

For later convenience, we also state a quick consequence of \cref{thm:MN-bilinear-forms-composite}.

\begin{corollary} \label{cor:MN-bilinear-forms-large-d}
Assume the setup of \cref{thm:MN-bilinear-forms-composite} and suppose that $d \ge c^{1/2}$. Then
\[
    \mathop{\sum_{m \in \mI} \sum_{n \in \mJ}}_{(m, n, c) = 1} \alpha_m \beta_n S(am, n; c) 
    \ll
    \|\alpha\| \|\beta\| c^{o(1)} \sqrt{MNc} \left(\frac{d^{\frac{1}{6}}}{N^{\frac{1}{3}}} + \frac{c^{\frac{1}{4}} d^{\frac{1}{12}}}{M^{\frac{1}{6}} N^\frac{1}{2}} + \frac{c^{\frac{11}{24}}}{(MN)^{\frac{1}{2}}} \right).
\]
\end{corollary}

\begin{proof}
By applying \cref{thm:MN-bilinear-forms-composite} and the bound $f \le \sqrt{cd}$, we obtain
\[
    \mathop{\sum_{m \in \mI} \sum_{n \in \mJ}}_{(m, n, c) = 1} \alpha_m \beta_n S(am, n; c) 
    \ll
    \|\alpha\| \|\beta\| c^{o(1)} \sqrt{MNc} \left(\frac{d}{N^2} + \frac{c^{\frac{3}{2}} d^{\frac{1}{2}}}{M N^3} + \frac{c^{\frac{7}{2}}}{d^{\frac{3}{2}} M^3 N^3} \right)^{\frac{1}{6}}.
\]
Using that $d \ge c^{1/2}$ in the last term, we obtain the claimed upper bound.
\end{proof}

\subsection{Near-prime moduli}
Building towards an unconditional result for general moduli, we need to slightly develop the result of Kowalski--Michel--Sawin \cite{kowalski2017bilinear} from \cref{thm:kms} so that it applies for near-prime moduli.

\begin{lemma} \label{lem:bilinear-split-two-moduli}
Let $q, c, M, N \in \Z_+$ with $M, N \le [q, c]$. Let $(\alpha_m)_{m \le M}$ and $(\beta_n)_{n \le N}$ be complex sequences, and $A(m, n; q)$ (respectively $B(m, n; c)$) be complex numbers depending only on the residues of $m, n$ modulo $q$ (respectively, modulo $c$). Then one has
\[
    \left\vert \sum_{m=1}^M \sum_{n=1}^N \alpha_m \beta_n A(m, n; q) B(m, n; c) \right\vert \le \mA \mB\, \|\alpha\| \|\beta\|,
\]
where
\[
    \mA := \sqrt{\sum_{m=1}^{\min(M,q)} \sum_{n=1}^{\min(N,q)} |A(m,n;q)|^2},
    \qquad\quad 
    \mB := \sup_{\|\tilde\alpha\| = \|\tilde\beta\| = 1}
    \left\vert \sum_{m=1}^{\min(M,c)} \sum_{n=1}^{\min(N,c)} \tilde\alpha_m \tilde\beta_n B(m, n; c)\right\vert,
\]
the supremum being over all complex sequences $(\tilde\alpha_m)_{m \le \min(M,c)}$, $(\tilde\beta_n)_{n \le \min(N,c)}$ with unit $\ell^2$ norms.
\end{lemma}

\begin{proof}
For $m_q, m_c, n_q, n_c \in \Z$, let us denote
\[
    \alpha_{m_q,m_c} := \sum_{\substack{1 \le m \le M \\ m \equiv m_q \pmod{q} \\ m \equiv m_c \pmod{c}}} \alpha_m,
    \qquad\qquad
    \beta_{n_q,n_c} := \sum_{\substack{1 \le n \le N \\ n \equiv n_q \pmod{q} \\ n \equiv n_c \pmod{c}}} \beta_n.
\]
Note that since $M, N \le [q, c]$, each of these sums contains at most one term.
We can then rewrite our bilinear sum as
\[
    \sum_{m_q=1}^{\min(M,q)} \sum_{n_q=1}^{\min(N,q)} A(m_q, n_q; q)
    \sum_{m_c=1}^{\min(M,c)} \sum_{n_c=1}^{\min(N,c)} \alpha_{m_q,m_c} \beta_{n_q,n_c}
    B(m_c, n_c; c).
\]
We then apply Cauchy--Schwarz in $m_q, n_q$ and use the definitions of $\mA, \mB$ to obtain
\[
\begin{aligned}
    \left\vert\sum_{m=1}^M \sum_{n=1}^N \alpha_m \beta_n A(m, n; q) B(m, n; c)\right\vert^2 
    &\le \mA^2 \sum_{m_q=1}^q \sum_{n_q=1}^q \left\vert \sum_{m_c=1}^{\min(M,c)} \sum_{n_c=1}^{\min(N,c)} \alpha_{m_q,m_c} \beta_{n_q,n_c} B(m_c,n_c;c)\right\vert^2
    \\
    &\le \mA^2 \sum_{m_q=1}^q \sum_{n_q=1}^q \mB^2 \sum_{m_c=1}^{\min(M,c)} |\alpha_{m_q,m_c}|^2
    \sum_{n_c=1}^{\min(N,c)} |\beta_{n_q,n_c}|^2
    \\
    &= 
    \mA^2 \mB^2 \|\alpha\|^2 \|\beta\|^2,
\end{aligned}
\]
which proves the desired bound.
\end{proof}

\begin{corollary}[Kowalski--Michel--Sawin bounds for near-primes] \label{cor:kms}
Let $c = pq$ where $p$ is a prime, $q \in \Z_+$, and $p \nmid q$. Let $M, N \in \Z$ be integers such that $1 \le N \le M \le c$. Then for any complex sequences $(\alpha_m)_{m \le M}$ and $(\beta_n)_{n \le N}$, and any $a \in (\Z/c\Z)^\times$, one has
\[
\begin{aligned}
    \mathop{\sum_{m=1}^M \sum_{n=1}^N}_{(m, n, c) = 1}\alpha_m \beta_n S(am, n; c)
    \ll 
    \|\alpha\|\|\beta\|
    c^{o(1)}
    \sqrt{MNc}
    \left(N^{-\frac{1}{2}}q + (MN)^{-\frac{3}{16}} c^{\frac{11}{64}} q^{\frac{53}{64}}\right).
\end{aligned}
\]
\end{corollary}
\begin{proof} 
By the twisted multiplicativity of Kloosterman sums, we have
\[
    \mathop{\sum_{m=1}^M \sum_{n=1}^N}_{(m, n, pq) = 1}\alpha_m \beta_n S(am, n; pq)
    =
    \mathop{\sum_{m=1}^M \sum_{n=1}^N} \alpha_m \beta_n S(a\bar{q}^2m, n; p)\one_{(m,n,p)=1} S(a\bar{p}^2m, n; q) \one_{(m,n,q)=1}.
\]
We separately consider those terms $m, n$ with $p \mid m$ (and $p \nmid n$) or $p \mid n$ (and $p \nmid m$). By the Weil and Ramanujan bounds from \cref{lem:weil,lem:ram}, their contribution is
\[
    \ll 
    \sum_{m=1}^M \sum_{n=1}^N |\alpha_m \beta_n| q^{\frac{1}{2}+o(1)}
    \ll 
    \|\alpha\| \|\beta\| q^{o(1)} \sqrt{MNq}.
\]
It follows that
\begin{equation} \label{eq:bilinear-after-ramanujan}
\begin{aligned}
    \mathop{\sum_{m=1}^M \sum_{n=1}^N}_{(m, n, c) = 1}\alpha_m \beta_n S(am, n; c)
    &=
    \mathop{\sum_{m=1}^M \sum_{n=1}^N} \alpha_m \beta_n S(a\bar{q}^2m, n; p)\one_{p \nmid mn} S(a\bar{p}^2m, n; q) \one_{(m,n,q)=1}
    \\
    &+
    O\left(\|\alpha\| \|\beta\| q^{o(1)} \sqrt{MNq}\right),
\end{aligned}
\end{equation}
and we can estimate the bilinear sum in the right-hand side using \cref{lem:bilinear-split-two-moduli} with $c = p$. This gives
\[
    \mathop{\sum_{m=1}^M \sum_{n=1}^N} \alpha_m \beta_n S(a\bar{q}^2m, n; p)\one_{p \nmid mn} S(a\bar{p}^2m, n; q) \one_{(m,n,q)=1}
    \ll 
    \mA\mB\, \|\alpha\| \|\beta\|,
\]
where
\[
\begin{aligned}
    \mA := \sqrt{\mathop{\sum_{m=1}^{\min(M,q)} \sum_{n=1}^{\min(N,q)}}_{(m,n,q)=1} |S(a\bar{p}^2m, n; q)|^2},
    \ \
    \mB := \sup_{\|\tilde\alpha\| = \|\tilde\beta\| = 1}
    \left\vert \sum_{m=1}^{\min(M,p-1)} \sum_{n=1}^{\min(N,p-1)} \tilde\alpha_m \tilde\beta_n S(a\bar{q}^2m, n; p)\right\vert.
\end{aligned}
\]
By the Weil bound (\cref{lem:weil}), we have $\mA \ll q^{o(1)} \sqrt{q^3}$.
By \cref{thm:kms} (which gives a bound that increases with $M$ and $N$), we also have
\[
\begin{aligned}
    \mB
    \ll 
    p^{o(1)} \sqrt{MNp}
    \left(N^{-\frac{1}{2}} + (MN)^{-\frac{3}{16}} p^{\frac{11}{64}}\right).
\end{aligned}
\]
Plugging these bounds into \cref{eq:bilinear-after-ramanujan}, we obtain
\[
\begin{aligned}
    \mathop{\sum_{m=1}^M \sum_{n=1}^N}_{(m, n, c) = 1}\alpha_m \beta_n S(am, n; c)
    \ll 
    c^{o(1)}
    q^{\frac{3}{2}}
    \|\alpha\|\|\beta\|
    \sqrt{MNp}
    \left(N^{-\frac{1}{2}} + (MN)^{-\frac{3}{16}} p^{\frac{11}{64}}\right)
    \\
    +
    \|\alpha\| \|\beta\| c^{o(1)} \sqrt{MNq}.
\end{aligned}
\]
Finally, recalling that $pq  = c$ and $M, N \le c$ (which imply $\sqrt{MNq} \le \sqrt{Mcq} \le q^{3/2}\sqrt{Mp}$), the last term can be omitted, and we arrive at the desired bound.
\end{proof}

\subsection{General moduli}
Finally, we prove a generalization of \cref{thm:bilinear-forms-general}, by combining previous results for various factorizations of the modulus $c$. In some ranges, we will use the following lemma.

\begin{lemma} \label{lem:greedy-fact}
Suppose $c \in \Z_+$ such that all prime powers $p^k \mid c$ have $p^k < c^{1/2}$. Then $c$ has a factorization $c = de$ with $(d, e) = 1$ and 
\[
    d \in [c^{1/2}, c^{3/4}].
\] 
\end{lemma}

\begin{proof}
We construct $d, e$ by a greedy algorithm. Initially, we take $d = e := 1$. For each prime power $p^k \| c$, we append $p^k$ to the smaller of $d$ and $e$. Note that throughout this process, $d$ and $e$ cannot differ by a factor larger than $c^{1/2}$. In the end, if $d < e$, we swap $d$ and $e$. We therefore obtain a factorization $c = de$ with $(d, e) = 1$ and $e \le d \le c^{1/2}e$, which implies $c^{1/2} \le d \le c^{3/4}$.
\end{proof}

\begin{theorem} \label{thm:MN-bilinear-forms-general}
Let $\delta \in [0, \tfrac{1}{24}]$, $c, M, N \in \Z_+$, $\tilde M := \max(M, N)$, and $\tilde N := \min(M, N)$. Let $(\alpha_m)_{m \le M}$, $(\beta_n)_{n 
\le N}$ be arbitrary complex sequences and $a \in (\Z/c\Z)^\times$. 
\begin{itemize} 
\item[$(i)$.] If $M, N \le c$, one has
\begin{equation} \label{eq:MN-bilinear-forms-general-1}
\begin{aligned}
    \mathop{\sum_{m=1}^M\sum_{n=1}^N}_{(m,n,c)=1} \alpha_m \beta_n S(am, n; c)
    &\ll 
    \|\alpha\|\|\beta\| c^{o(1)} \sqrt{MNc}
    \\
    &\times \left(\frac{c^{\frac{11+53\delta}{64}}}{(MN)^{\frac{3}{16}}}
    +
    \frac{c^{\frac{1-\delta}{6}}}{\tilde N ^{\frac{1}{3}}} +
    \frac{c^{\frac{4-\delta}{12}}}{\tilde M^{\frac{1}{6}} \tilde N^{\frac{1}{2}}} + \frac{c^{\frac{11}{24}}}{(M N)^\frac{1}{2}}
    \right).
\end{aligned}
\end{equation}
\item[$(ii)$.] If $|\alpha_m| \le 1$ for all $m$ (so $\|\alpha\| \le \sqrt{M}$), then
\begin{equation} \label{eq:MN-bilinear-forms-general-2}
\begin{aligned}
    \mathop{\sum_{m=1}^M\sum_{n=1}^N}_{(n,c)=1} \alpha_m \beta_n S(am, n; c)
    &\ll 
    \sqrt{M}\|\beta\| c^{o(1)} \sqrt{MNc}
    \\
    &\times \left(\frac{c^{\frac{11+53\delta}{64}}}{(MN)^{\frac{3}{16}}}
     + \frac{c^{\frac{1-\delta}{4}}}{\tilde N^{\frac{1}{2}}} + \frac{1}{c^{\frac{3}{16}}} + \frac{c^{\frac{1}{8}}}{\tilde N^{\frac{1}{3}}} + \frac{c^{\frac{11}{24}}}{(M N)^\frac{1}{2}}
    \right).
\end{aligned}
\end{equation}
\end{itemize}
\end{theorem}

\begin{proof}[Proof of \cref{thm:MN-bilinear-forms-general}]
We first assume that $M, N \le c$ for both \cref{eq:MN-bilinear-forms-general-1,eq:MN-bilinear-forms-general-2}; we will remove this assumption for \cref{eq:MN-bilinear-forms-general-2} at the end of the proof.

If $c$ has a factorization $c = dd'e$ with $d' \mid d$ and $(d, e) = 1$ such that $d \ge c^{1/2}$, then \cref{cor:MN-bilinear-forms-large-d} (applied for $\tilde M$, $\tilde N$ instead of $M, N$) gives
\begin{equation} \label{eq:modified-bilinear-composite}
    \mathop{\sum_{m=1}^M\sum_{n=1}^N}_{(m, n, c) = 1} \alpha_m \beta_n S(am, n; c) 
    \\ 
    \ll
    c^{o(1)} \underbrace{\|\alpha\| \|\beta\| \sqrt{MNc} \left(\frac{d^{\frac{1}{6}}}{\tilde N^{\frac{1}{3}}} + \frac{c^{\frac{1}{4}} d^{\frac{1}{12}}}{\tilde M^{\frac{1}{6}} \tilde N^\frac{1}{2}} + \frac{c^{\frac{11}{24}}}{(MN)^{\frac{1}{2}}} \right)}_{=:\, \mB(d)}.
\end{equation}
We note that the bound $\mB(d)$ is increasing with $d \in [c^{1/2}, c]$, and that:
\begin{itemize} 
    \item[$(i)$.] The right-hand side of \cref{eq:MN-bilinear-forms-general-1} supersedes $\mB(c^{1-\delta})$ (see the last three terms in \cref{eq:MN-bilinear-forms-general-1});
    \item[$(ii)$.] The right-hand side of \cref{eq:MN-bilinear-forms-general-2} supersedes $\mB(c^{3/4})$. Indeed, we have
\[
\begin{aligned}
    \mB(c^{3/4}) &= \|\alpha\| \|\beta\| \sqrt{MNc} \left(\frac{c^{\frac{1}{8}}}{\tilde N^{\frac{1}{3}}} + \frac{c^{\frac{5}{16}}}{\tilde M^{\frac{1}{6}} \tilde N^\frac{1}{2}} + \frac{c^{\frac{11}{24}}}{(MN)^{\frac{1}{2}}} \right),
\end{aligned}
\]
and a quick computation shows that since $\delta \le \tfrac{1}{24}$, 
\[
    \frac{c^{\frac{5}{16}}}{\tilde M^{\frac{1}{6}} \tilde N^\frac{1}{2}}
    \le 
    \max\left(\frac{c^{\frac{11}{24}}}{(\tilde M \tilde N)^{\frac{1}{2}}}, \frac{c^{\frac{1-\delta}{4}}}{\tilde N^{\frac{1}{2}}}\right).
\]
\end{itemize} 
We now split into cases depending on the factorization of the modulus $c$. 

\textbf{Case 1:} $c$ is divisible by a maximal prime power $p^k \ge c^{1-\delta}$. Then let us write $c = p^kq$, where $q$ is not necessarily a prime, but $q \le c^\delta$ and $(p, q) = 1$.

\emph{Subcase 1.1:} One has $k = 1$. Then we can apply \cref{cor:kms} (with $M, N$ replaced by $\tilde M, \tilde N$), which gives the bound
\[
\begin{aligned}
    \mathop{\sum_{m=1}^M\sum_{n=1}^N}_{(m,n,c)=1} \alpha_m \beta_n S(am, n; c)
    \ll 
    \|\alpha\|\|\beta\| c^{o(1)} \sqrt{MNc}
    \left(\tilde N^{-\frac{1}{2}}c^{\delta} + (MN)^{-\frac{3}{16}} c^{\frac{11+53\delta}{64}}\right).
\end{aligned}
\]
The first term here is superseded by the third term in \cref{eq:MN-bilinear-forms-general-1} and the second term in \cref{eq:MN-bilinear-forms-general-2}, since $\tilde M \le c$ and $\delta \le \tfrac{1}{24}$. The second term here appears directly in both \cref{eq:MN-bilinear-forms-general-1,eq:MN-bilinear-forms-general-2}.

\emph{Subcase 1.2:} One has $k \ge 2$. Then we let $d := p^{\lc k/2 \rc}q$, $d' := p^{\lf k/2 \rf}$, and $e := 1$, which gives a valid decomposition $c = dd'e$ to use in our \cref{thm:MN-bilinear-forms-composite}. Moreover, since $k \ge 2$, we have $\lc k/2 \rc \le 2k/3$, so $d \le p^{2k/3}q = c^{2/3} q^{1/3}$, and thus
\[
    d \in [c^{1/2}, c^{(2+\delta)/3}].
\]
From \cref{eq:modified-bilinear-composite} we thus obtain an upper bound of $c^{o(1)} \mB(c^{(2+\delta)/3})$, which is acceptable in both \cref{eq:MN-bilinear-forms-general-1,eq:MN-bilinear-forms-general-2} since $\tfrac{2+\delta}{3} \le \tfrac{3}{4}$.

\textbf{Case 2:} All prime powers $p^k \mid c$ have $p^k < c^{1-\delta}$.

\emph{Subcase 2.1:} All prime powers $p^k \mid c$ have $p^k < c^{1/2}$. Then we set $d' := 1$, and use \cref{lem:greedy-fact} to obtain a factorization $c = dd'e$ with $(d, e) = 1$ and $d \in [c^{1/2}, c^{3/4}]$.
Then \cref{eq:modified-bilinear-composite} gives a bound of $c^{o(1)} \mB(c^{3/4})$, 
which is acceptable in both \cref{eq:MN-bilinear-forms-general-1,eq:MN-bilinear-forms-general-2}.

\emph{Subcase 2.2:} The largest prime power dividing $c$ is some $p^k \in [c^{1/2}, c^{3/4})$. Then we let $d := p^k$, $d' := 1$, and $e := cp^{-k}$, and \cref{eq:modified-bilinear-composite} gives an acceptable bound of $c^{o(1)} \mB(c^{3/4})$ once again.

\emph{Subcase 2.3:} The largest prime power dividing $c$ is some $p^k \in [c^{3/4}, c^{1-\delta})$. On the one hand, \cref{eq:modified-bilinear-composite} gives a bound of $c^{o(1)} \mB(c^{1-\delta})$, which is acceptable in \cref{eq:MN-bilinear-forms-general-1}; this completes the proof of \cref{eq:MN-bilinear-forms-general-1}.

Now assume (still within Subcase 2.3) that $|\alpha_m| \le 1$ for all $m$, and we aim to establish \cref{eq:MN-bilinear-forms-general-2}.
\begin{itemize}
    \item If $p = 2$, then writing $c = 2^k q$, we can factorize $c = dd'e$ with $d := 2^{\lc k/2 \rc} q$, $d' := 2^{\lf k/2 \rf}$, and $e := 1$. Here $d \ll \sqrt{2^k q^2} = \sqrt{cq} \le c^{3/4}$, since $2^k \ge c^{1/2}$ implies $q \le c^{1/2}$. But then \cref{eq:modified-bilinear-composite} gives an acceptable bound of $c^{o(1)} \mB(c^{3/4})$. 
    \item If $p > 2$, then we can use $d := p^k$ in \cref{thm:bm}. Since $d \in [c^{3/4}, c^{1-\delta})$, this gives the bound 
    \[
        \mathop{\sum_{m=1}^M\sum_{n=1}^N}_{(n,c)=1} \alpha_m \beta_n S(am, n; c)
        \ll 
        \sqrt{M}\|\beta\| c^{o(1)} \sqrt{MNc}
        \left(\frac{c^{\frac{1}{8}}}{M^{\frac{1}{2}}} + \frac{1}{c^{\frac{3}{16}}} + \frac{c^{\frac{1-\delta}{4}}}{N^{\frac{1}{2}}}\right).
    \]
    Since $M, N \ge \tilde{N}$ and $\tfrac{1}{8} \le \tfrac{1-\delta}{4}$, the first and the last terms in the parenthesis above are superseded by the term $c^{(1-\delta)/4} \tilde N^{-1/2}$ from \cref{eq:MN-bilinear-forms-general-2}. The second term appears directly in \cref{eq:MN-bilinear-forms-general-2}.
\end{itemize}
This covers all cases assuming $M, N \le c$. Our last step is to remove this assumption for \cref{eq:MN-bilinear-forms-general-2}. 

First, if $M, N > c$, then applying the (first) bound from \cref{eq:trivial-bound} for the sequences $(\alpha'_{m'})_{m' \le c}$, $(\beta'_n)_{n \le N}$ given by $\alpha'_{m'} := \sum_{m \equiv m' \pmod{c}} \alpha_m$ and $\beta'_n := \one_{(n, c) = 1} \sum_{n \equiv n' \pmod{c}} \beta_n$ leads to the bound
\[
\begin{aligned}
    \mathop{\sum_{m=1}^M\sum_{n=1}^N}_{(n,c)=1} \alpha_m \beta_n S(am, n; c) \ll \|\alpha'\| \|\beta'\| c^{1+o(1)} 
    &\ll \frac{M}{c} \sqrt{c} \sqrt{\frac{N}{c}} \|\beta\| c^{1+o(1)}
    \\
    &\ll \sqrt{M} \|\beta\| c^{o(1)} \sqrt{MNc} \cdot c^{-3/16},
\end{aligned}
\]
so \cref{eq:MN-bilinear-forms-general-2} still holds. Secondly, if $N \le c < M$, then applying the (first) bound from \cref{eq:trivial-bound} for the sequences $(\alpha'_{m'})_{m' \le c}$, $(\beta'_n)_{n \le N}$ given by $\alpha'_{m'} := \sum_{m \equiv m' \pmod{c}} \alpha_m$ and $\beta'_n := \beta_n \one_{(n, c) = 1}$ leads to the bound
\[
\begin{aligned}
    \mathop{\sum_{m=1}^M\sum_{n=1}^N}_{(n,c)=1} \alpha_m \beta_n S(am, n; c) \ll \|\alpha'\| \|\beta'\| c^{1+o(1)} 
    &\ll \frac{M}{c} \sqrt{c} \|\beta\| c^{1+o(1)}
    \\
    &\ll \sqrt{M} \|\beta\| c^{o(1)} \sqrt{MNc} \cdot \frac{c^{\frac{1-\delta}{4}}}{N^{\frac{1}{2}}},
\end{aligned}
\]
so \cref{eq:MN-bilinear-forms-general-2} still holds. An analogous argument covers the remaining case $M \le c < N$.
\end{proof}

\begin{proof}[Proof of \cref{thm:bilinear-forms-general}]
Since $M, N \ll c^{\frac{1}{2}+o(1)}$ and the desired result is trivial when $c = O(1)$, we can assume without loss of generality that $M, N \le c$. We then apply \cref{thm:MN-bilinear-forms-general} with $M, N \ll c^{1/2+o(1)}$, using the optimal choices $\delta = \frac{3}{175}$ in \cref{eq:MN-bilinear-forms-general-1}, respectively $\delta = \frac{1}{69}$ in \cref{eq:MN-bilinear-forms-general-2}.
\end{proof}

\subsection{Averaging over moduli}

Finally, let us prove a generalization of \cref{cor:bilinear-forms-avg-c}.

\begin{corollary} \label{cor:MN-bilinear-forms-avg-c}
Let $q = dd'e$ for some $d, d', e \in \Z_+$ with $d' \mid d$ and $(d, e) = 1$, and $f \le \sqrt{qd}$ be the largest integer with $f^2 \mid qd$. Let $C \ge \tfrac{1}{2}$ and $\mI, \mJ \subset \Z_+$ be intervals of lengths $|\mI| = M$, $|\mJ| = N$, with $1 \le N \le M \le C$ and $\max(\mI \cup \mJ) \ll C^{O(1)}$. Let $(\alpha_m)_{m \in \mI}, (\beta_n)_{n \in \mJ}$ be complex sequences, and for each $c \sim C$, let $(\alpha_m(c))_{m \in \mI}$, $(\beta_n(c))_{n \in \mJ}$ be such that $|\alpha_m(c)| \le |\alpha_m|$, $|\beta_n(c)| \le |\beta_n|$ for all $m \in \mI, n \in \mJ$. Then one has
\[
\begin{aligned}
    \sum_{\substack{c \sim C \\ q \mid c}} \left\vert
    \mathop{\sum\sum}_{\substack{m \in \mI, n \in \mJ \\ (m, n, q) = 1}} \alpha_m(c) \beta_n(c) S(m, n; c) \right\vert 
    &\ll
    \|\alpha\| \|\beta\|
    \frac{C^{2+o(1)}}{q}
    \min 
    \begin{cases} 
    \left(\frac{dM^3 N}{C^3} + \frac{fM^2}{C^2} + \frac{f}{d^2} \right)^{\frac{1}{6}},
    \\ 
    \left(\frac{dM^3 N}{qC^2} + \frac{fM^2}{qC} + \frac{f q}{d^2 C} \right)^{\frac{1}{6}}.
    \end{cases}
    \hspace{-0.8cm}
\end{aligned}
\]
\end{corollary}

\begin{proof}[Proof of \cref{cor:MN-bilinear-forms-avg-c}]
Throughout this proof, we will use the notation
\[
    f_a := \max_{\tilde{f}^2 \mid a} \tilde{f},
    \qquad\qquad 
    \|\alpha_{a*}\| := \sqrt{\sum_{\substack{m \in \mI \\ a \mid m}} |\alpha_m|^2 },
    \qquad\qquad 
    \|\beta_{a*}\| := \sqrt{\sum_{\substack{n \in \mJ \\ a \mid n}} |\beta_n|^2 },
\]
for any $a \in \Z_+$. In particular, the assumption of the present \cref{cor:MN-bilinear-forms-avg-c} takes $f = f_{qd}$. Note that $f_a \mid f_{ab}$ and $f_{a^2b} = a f_b$ for any $a, b \in \Z_+$, and that $f_{ab} = f_a f_b$ when $(a, b) = 1$.

We can of course assume without loss of generality that $C \gg q$, since otherwise the sum over $c$ is empty. For each $c \sim C$ with $q \mid c$, we consider the sum
\[
\begin{aligned}
    \mS(c) &:= \mathop{\sum\sum}_{\substack{m \in \mI, n \in \mJ \\ (m, n, q) = 1}} \alpha_m(c) \beta_n(c) S(m, n; c) 
    \\
    &=
    \sum_{\substack{g \mid c \\ (g, q) = 1}}
    \mathop{\sum\sum}_{\substack{m \in \mI, n \in \mJ \\ (m, n, c) = g}} \alpha_m(c) \beta_n(c) S(m, n; c)
    =
    \sum_{\substack{g \mid c \\ (g, q) = 1}}
    \frac{\phi(c)}{\phi(c/g)}
    \mathop{\sum\sum}_{\substack{m \in \mI, n \in \mJ \\ (m, n, c) = g}} \alpha_{m}(c) \beta_{n}(c) S(\tfrac{m}{g}, \tfrac{n}{g}; \tfrac{c}{g}),
\end{aligned}
\]
where the last equality follows from the identity $S(m, n; c) = \tfrac{\phi(c)}{\phi(c/g)}S(\tfrac{m}{g},\tfrac{n}{g};\tfrac{c}{g})$. From the triangle inequality and the bound $\tfrac{\phi(c)}{\phi(c/g)} \le g$, we find that
\begin{equation} \label{eq:break-up-c-sum}
    \sum_{\substack{c \sim C \\ q \mid c}} |\mS(c)| \le 
    \sum_{\substack{g \le 2C/q \\ (g, q) = 1}} g \sum_{\substack{c \sim C \\ gq \mid c}} |\mS(c; g)|,
\end{equation}
where
\[
    \mS(c; g) := \mathop{\sum\sum}_{\substack{m \in \mI, n \in \mJ \\ g \mid (m, n) \\  (\frac{m}{g},\frac{n}{g},\frac{c}{g}) = 1}} \alpha_{m}(c) \beta_{n}(c) S(\tfrac{m}{g}, \tfrac{n}{g}; \tfrac{c}{g}).
\]
We aim to apply \cref{thm:MN-bilinear-forms-composite} (with $M, N, c \gets \tfrac{M}{g}, \tfrac{N}{g}, \tfrac{c}{g}$) to bound each sum $\mS(c; g)$, and this requires a suitable factorization of the modulus $\tfrac{c}{g}$. There are two ways to construct this from the assumed factorization $q = dd'e$, which correspond to placing `most' of the factor $\tfrac{c}{gq}$ into $e$ or into $d$.

\textbf{Method 1.} For each $c \sim C$ with $gq \mid c$, consider the factorization
\[
    \frac{c}{g} =: c'q = \tilde d d' \tilde e, \qquad \tilde d := (c', d^\infty) d, \qquad \tilde e := \frac{ec'}{(c',d^\infty)},
\]
which has $d' \mid \tilde d$ and $(\tilde d, \tilde e) = 1$. We find that
\[
    f_{(c/g)\tilde{d}}^2 \mid c'q(c', q^\infty) d = (c', q^\infty)^2 \frac{c'}{(c',q^\infty)} qd 
    \qquad 
    \Rightarrow 
    \qquad 
    f_{(c/g)\tilde{d}} \le (c', q^\infty) f_{c'} f_{qd} = (c', q^\infty) f_{c'} f,
\]
so \cref{thm:MN-bilinear-forms-composite} gives 
\[
\begin{aligned}
    \mS(c; g) = \mS(c'gq; g) 
    &\ll \|\alpha_{g*}\| \|\beta_{g*}\| \frac{c^{1+o(1)}}{g}
    \left(\frac{\tilde d M^3 N}{c^3} + \frac{f_{(c/g)\tilde{d}} M^2}{c^2} + \frac{f_{(c/g)\tilde{d}}}{\tilde d^2} \right)^{\frac{1}{6}} g^{\frac{1}{6}}
    \\
    &\ll 
    \|\alpha_{g*}\| \|\beta_{g*}\| C^{1+o(1)}
    \left(\frac{dM^3 N}{C^3} + \frac{fM^2}{C^2} + \frac{f}{d^2} \right)^{\frac{1}{6}} \left((c', q^\infty) f_{c'}\right)^{\frac{1}{6}}.
\end{aligned}
\]
Therefore,
\[
\begin{aligned}
    \sum_{\substack{c \sim C \\ gq \mid c}} |\mS(c; g)| 
    &= 
    \sum_{c' \sim \frac{C}{gq}} |\mS(c; g)|
    \\
    &\ll 
    \|\alpha_{g*}\| \|\beta_{g*}\| C^{1+o(1)}
    \left(\frac{dM^3 N}{C^3} + \frac{fM^2}{C^2} + \frac{f}{d^2} \right)^{\frac{1}{6}} 
    \sum_{c' \sim \frac{C}{gq}} \left((c', q^\infty) f_{c'}\right)^{\frac{1}{6}}.
\end{aligned}
\]
After applying Cauchy--Schwarz to the last sum, it remains to bound the sums $\sum_{c' \sim C/(gq)} (c',q^\infty)$ and $\sum_{c' \sim C/(gq)} f_{c'}$, both of which are $O(\tfrac{C^{1+o(1)}}{gq})$. In particular, for the second sum, we can write
\[
    \sum_{c' \sim \frac{C}{gq}} f_{c'} \le \sum_{f \ll \sqrt{\frac{C}{gq}}} f \sum_{c' \sim \frac{C}{gq}} \one_{f^2 \mid c'}
    \ll 
    \sum_{f \le \sqrt{\frac{C}{gq}}} \frac{C}{gqf} \ll \frac{C^{1+o(1)}}{gq}.
\]
From this and \cref{eq:break-up-c-sum}, we conclude that
\[
    \sum_{\substack{c \sim C \\ q \mid c}} |\mS(c)| \ll  \frac{C^{2+o(1)}}{q}
    \left(\frac{dM^3 N}{C^3} + \frac{fM^2}{C^2} + \frac{f}{d^2} \right)^{\frac{1}{6}} 
    \sum_{\substack{g \le 2C/q \\ (g, q) = 1}} \|\alpha_{g*}\| \|\beta_{g*}\|.
\]
Finally, the last sum is easily bounded by $C^{o(1)} \|\alpha\|\|\beta\|$ using Cauchy--Schwarz and the divisor bound.
This establishes the bound from \cref{cor:MN-bilinear-forms-avg-c} with the first term from the minimum.

\textbf{Method 2.}
For each $c \sim C$ with $gq \mid c$, consider the factorization
\[
    \frac{c}{g} =: c'q = \tilde d d' \tilde e, \qquad \tilde d := \frac{c'd}{(c',e^\infty)}, \qquad \tilde e := e(c',e^\infty),
\]
which satisfies $d' \mid \tilde d$ and $(\tilde d, \tilde e) = 1$. We find that
\[
    f_{(c/g)\tilde{d}}^2 \mid (c')^2 qd 
    \qquad \Rightarrow 
    \qquad 
    f_{(c/g)\tilde{d}} \le c'f_{qd} \ll \frac{fC}{q},
\]
so \cref{thm:MN-bilinear-forms-composite} gives
\[
\begin{aligned}
    \mS(c; g) = \mS(c'gq; g)
    &\ll 
    \|\alpha_{g*}\| \|\beta_{g*}\| C^{1+o(1)}
    \left(\frac{dM^3 N}{qC^2} + \frac{fM^2}{qC} + \frac{fq}{d^2C} \right)^{\frac{1}{6}} (c', e^\infty)^{\frac{1}{3}}.
\end{aligned}
\]
The second bound from \cref{cor:MN-bilinear-forms-avg-c} now follows from \cref{eq:break-up-c-sum} similarly as before, since the sum over $c' \sim \tfrac{C}{gq}$ `washes out' the factor $(c', e^\infty)$.
\end{proof}

\begin{proof}[Proof of \cref{cor:bilinear-forms-avg-c}]
This follows from \cref{cor:MN-bilinear-forms-avg-c} analogously to how \cref{thm:bilinear-forms-composite} follows from \cref{thm:MN-bilinear-forms-composite}.
\end{proof}

\section{Moments of twisted modular \texorpdfstring{$L$}{L}-functions} \label{sec:mom-l-functions}

Here we prove \cref{thm:mom-l-fns}, by inserting our bounds for bilinear forms with Kloosterman sums into the proofs from \cite{blomer2015second}. 
We begin by restating \cref{eq:MN-bilinear-forms-general-2} in a shape more similar to \cite[Theorem 5]{blomer2015second}.

\begin{corollary} \label{cor:kloost-large-sieve}
Let $\delta \in [0, \tfrac{1}{24}]$, $r, q \in \Z_+$ with $r \mid q$. Let $K, M \ge 1$, $\tilde K := \max(K, M)$, $\tilde M := \min(K, M)$, and $(\lambda_k)_{K \le k \le 2K}$ be a sequence with $|\lambda_k| \le 1$ for all $k$. Then one has
\[
\begin{aligned}
    \sum_{\substack{M \le m \le 2M \\ (m, q) = 1}} \left\vert \sum_{K \le k \le 2K} \lambda_k S(k, m; r) \right\vert^2 
    &\ll 
    (qKM)^{o(1)} K^2Mr 
    \\
    &\times \left( \frac{r^{\frac{11+53\delta}{32}}}{(KM)^{\frac{3}{8}}}
     + \frac{r^{\frac{1-\delta}{2}}}{\tilde M} + \frac{1}{r^{\frac{3}{8}}} + \frac{r^{\frac{1}{4}}}{\tilde M^{\frac{2}{3}}} + \frac{r^{\frac{11}{12}}}{KM} \right).
\end{aligned}
\]
\end{corollary}

\begin{proof}
One can of course assume without loss of generality that $M, K \in \Z_+$, and extend the sum over $m$ to include all $m \in [M, 2M]$ with $(m, r) = 1$. By duality, it suffices to establish the bound
\[
\begin{aligned}
    \sum_{\substack{M \le m \le 2M \\ (m, r) = 1}} \beta_m \sum_{K \le k \le 2K} \lambda_k S(k, m; r) 
    &\ll 
    (qKM)^{o(1)} \|\beta\| K\sqrt{Mr}
    \\
    &\times \left( \frac{r^{\frac{11+53\delta}{64}}}{(KM)^{\frac{3}{16}}}
     + \frac{r^{\frac{1-\delta}{4}}}{\tilde M^{\frac{1}{2}}} + \frac{1}{r^{\frac{3}{16}}} + \frac{r^{\frac{1}{8}}}{\tilde M^\frac{1}{3}} + \frac{r^{\frac{11}{24}}}{(K M)^\frac{1}{2}} \right),
\end{aligned}
\]
for any sequence $(\beta_m)_{M \le m \le 2M}$. But this is precisely the content of \cref{eq:MN-bilinear-forms-general-2} with $(M, N, c)$ replaced by $(K, M, r)$.
\end{proof}

We can now prove an analogue of \cite[Proposition 7]{blomer2015second}. We use the same normalization as in \cite[(2.3)]{blomer2015second} for the Hecke eigenvalues $\lambda_f(n)$ of a holomorphic cuspidal newform $f$ for $\SL_2(\Z)$, so that
\begin{equation} \label{eq:hecke-normaliz}
    \lambda_f(n) \rho_f(1) = \sqrt{n} \rho_f(n),
    \qquad 
    \text{where}
    \qquad 
    f(z) = \sum_{n = 1}^\infty \rho_f(n) (4\pi n)^{k/2} e(nz).
\end{equation}
In particular, the Deligne bound \cite{deligne1974conjecture} reads
\begin{equation} \label{eq:deligne}
    \lambda_f(n) \ll n^{o(1)}.
\end{equation}

\begin{proposition} \label{prop:S-bound}
Let $\eps > 0$, $q, d \in \Z_+$ with $d \mid q$, $\frac{1}{20} N \ge M \ge 1$ with $MN \le q^{2+\eps}$, and let $\lambda_1(m)$, $\lambda_2(n)$ be the Hecke eigenvalues of two (fixed) holomorphic cuspidal newforms for $\SL_2(\Z)$. Let $V_1, V_2 : \R \to \C$ be functions supported in $[1, 2]$ with derivatives $V_i^{(j)} \ll_{j,\eps} q^\eps$, and denote
\begin{equation} \label{eq:S-notation}
    S_{N,M,d,q} := \frac{d}{(NM)^{1/2}} \sum_{\substack{n \equiv m \pmod{d} \\ (nm, q) = 1 \\ n \neq m}} \lambda_1(m) \lambda_2(n) V_1\left(\frac{m}{M}\right) V_2\left(\frac{n}{N}\right).
\end{equation}
Then for any $\delta \in [0, \tfrac{1}{24}]$, one has
\begin{equation} \label{eq:S-bound}
    S_{N,M,d,q} \ll_\eps q^{O(\eps)} 
    \left(
    \frac{M^{\frac{5}{16}} q^{\frac{83+53\delta}{64}}}{N^{\frac{5}{16}}} + \frac{q^{\frac{7-\delta}{4}}}{\sqrt{N}}
    + 
    \frac{M^{\frac{1}{2}} q^{\frac{21}{16}}}{N^{\frac{1}{2}}}
    + 
    \frac{q^{\frac{13}{8}} M^{\frac{1}{6}}}{N^{\frac{1}{2}}}
    +
    q^{\frac{23}{24}}
    \right).
\end{equation}
\end{proposition}

\begin{proof}
We closely follow the proof in \cite[\S 4]{blomer2015second}. In particular, we decompose $q = q_d q'$ where $q'$ is maximal with $(q', d) = 1$.
The bound \cite[(4.2)]{blomer2015second} reads
\[
    S_{N,M,d,q} \ll \sqrt{N} 
    \sum_{\substack{g \mid f \mid q' \\ r \mid d}} \frac{\mu^2(f) |\lambda_2(f/g)|}{fgr} 
    \left(\sum_{\substack{m \asymp M \\ (m, q) = 1}} \Big\vert \sum_n S(\bar{fg}m, n; r) \lambda_2(n) V_2^\circ\left(\frac{nN}{fgr^2}\right) \Big\vert^2\right)^{1/2},
\]
where $V_2^\circ$ is a transform of $V_2$ as in \cite[(2.10)]{blomer2015second} (coming from an application of the Voronoi summation formula). Using the rapid decay of $V_2^\circ$, we may truncate the sum over $n$ at
\[
    n \le K_{f,g,r} := q^\eps \frac{fgr^2}{N},
\]
up to an acceptable loss. Note that the resulting sum over $n$ vanishes unless $K_{f,g,r} \ge 1$.
From \cref{cor:kloost-large-sieve}, \cref{eq:deligne}, and the divisor bound, we conclude that
\[
\begin{aligned}
    S_{N,M,d,q} \ll_\eps q^{O(\eps)} \sqrt{N} 
    \max_{\substack{g \mid f \mid q' \\ r \mid d}} \frac{1}{fgr}
    K_{f,g,r}\sqrt{Mr}
    \Bigg( \frac{r^{\frac{11+53\delta}{64}}}{(K_{f,g,r}M)^{\frac{3}{16}}}
    + \frac{r^{\frac{1-\delta}{4}}}{\min(K_{f,g,r},M)^{\frac{1}{2}}} + \frac{1}{r^{\frac{3}{16}}} 
    \\
    + \frac{r^{\frac{1}{8}}}{\min(K_{f,g,r},M)^\frac{1}{3}} + \frac{r^{\frac{11}{24}}}{(K_{f,g,r} M)^\frac{1}{2}} \Bigg).
\end{aligned}
\]
Plugging in the definition of $K_{f,g,r}$, we see that the expression inside the maximum is non-decreasing in $r$ and non-increasing in $f, g$. Writing 
\[
    K := K_{1,1,q} = \frac{q^{2+\eps}}{N} \ge M,
\]
we find that
\[
\begin{aligned}
    S_{N,M,d,q} &\ll_\eps q^{O(\eps)} \sqrt{N} 
    \frac{1}{q}
    K\sqrt{Mq}
    \left( \frac{q^{\frac{11+53\delta}{64}}}{(KM)^{\frac{3}{16}}}
    + \frac{q^{\frac{1-\delta}{4}}}{M^{\frac{1}{2}}} + \frac{1}{q^{\frac{3}{16}}} + \frac{q^{\frac{1}{8}}}{M^\frac{1}{3}} + \frac{q^{\frac{11}{24}}}{(K M)^\frac{1}{2}} \right)
    \\
    &
    \ll_\eps q^{O(\eps)} 
    \frac{q^{\frac{3}{2}} M^{\frac{1}{2}}}{N^{\frac{1}{2}}}
    \left( \frac{q^{\frac{11+53\delta}{64}}}{(q^2M/N)^{\frac{3}{16}}}
    + \frac{q^{\frac{1-\delta}{4}}}{M^{\frac{1}{2}}} + \frac{1}{q^{\frac{3}{16}}} + \frac{q^{\frac{1}{8}}}{M^\frac{1}{3}} + \frac{q^{\frac{11}{24}}}{(q^2 M/N)^\frac{1}{2}} \right),
\end{aligned}
\]
which reduces to the desired bound.
\end{proof}

We can now prove the desired asymptotic for twisted moments of modular $L$-functions.

\begin{proof}[Proof of \cref{thm:mom-l-fns}]
Let $\eps > 0$ and $\gamma := \tfrac{1}{674}$. We closely follow the proof in \cite[\S 3]{blomer2015second}, making no changes to the main term analysis from \cite[\S 3.1]{blomer2015second}. Treating the off-diagonal term as in \cite[\S 3.2]{blomer2015second}, it remains to establish the bound
\begin{equation} \label{eq:mom-des-bound}
    S_{N,M,d,q} \stackrel{?}{\ll_\eps} q^{1 - \gamma + O(\eps)},
\end{equation}
for all $d \mid q$ and $N \ge M \ge 1$ with $MN \le q^{2+\eps}$, using the notation from \cref{eq:S-notation}. As in \cite[\S 3.3]{blomer2015second}, we can easily discount the contribution of the range $M \le N < 20 M$ using \cite[(3.12)]{blomer2015second}, so let us assume that $N \ge 20 M$. We will rely on the bounds
\begin{align} \label{eq:known-bound-1}
    S_{N,M,d,q} &\ll_\eps q^{O(\eps)} (MN)^{\frac{1}{2}},
    \\ \label{eq:known-bound-2}
    S_{N,M,d,q} &\ll_\eps q^{O(\eps)} \left(\frac{(Nq)^{\frac{1}{2}}}{M^{\frac{1}{2}}} + \frac{N^{\frac{3}{4}}}{M^{\frac{1}{4}}} + 
    \frac{N^{\frac{1}{4}} q^{\frac{3}{4}}}{M^{\frac{1}{4}}} + N^{\frac{1}{2}} q^{\frac{1}{4}}\right),
\end{align}
from \cite[(3.6) and (3.11)]{blomer2015second}, as well as on our \cref{prop:S-bound} (instead of \cite[Proposition 7]{blomer2015second}). First, the trivial bound \cref{eq:known-bound-1} establishes \cref{eq:mom-des-bound} unless
\begin{equation} \label{eq:M-lower-bound}
    M > \frac{q^{2-2\gamma}}{N},
\end{equation}
so let us assume that we are in this range. We now split into cases depending on the size of $N$.

\textbf{Case 1:} One has $N \le q^{3/2-3\gamma}$. Then by plugging \cref{eq:M-lower-bound} into \cref{eq:known-bound-2}, we obtain \cref{eq:mom-des-bound}.

\textbf{Case 2:} One has $N \in (q^{3/2-3\gamma}, q^{3/2-2\gamma}]$. Then by plugging \cref{eq:M-lower-bound} into \cref{eq:known-bound-2}, we find that
\[
    S_{N,M,d,q} \ll_\eps q^{O(\eps)} \left(q^{1-\gamma} + \frac{N^{\frac{1}{4}} q^{\frac{3}{4}}}{M^{\frac{1}{4}}} \right),
\]
which is acceptable in \cref{eq:mom-des-bound} unless
\[
    \frac{N}{M} > q^{1-4\gamma}.
\]
Plugging this and $N > q^{3/2-3\gamma}$ into \cref{eq:S-bound}, we find that
\begin{equation} \label{eq:case-2-S-bound}
    S_{N,M,d,q} \ll_\eps q^{O(\eps)} \left(
    q^{\frac{63+53\delta}{64} + \frac{5\gamma}{4}} + q^{1-\frac{\delta}{4} + \frac{3\gamma}{2}} 
    + 
    q^{\frac{23}{24} + 2\gamma}
    \right),
\end{equation}
which is acceptable in \cref{eq:mom-des-bound} provided that
\[
    10 \gamma \le \delta \le \frac{1-144\gamma}{53}.
\]
This is precisely attained for our choice of $\gamma = \frac{1}{674}$ by taking $\delta := \frac{10}{674}$ in \cref{prop:S-bound}.

\textbf{Case 3:} One has $N \in (q^{3/2-2\gamma}, q^{3/2+\gamma})$. Then \cref{eq:known-bound-2} is useless because of the last term. We plug in $M \le q^{2+\eps}/N$ and then $N \ge q^{3/2-2\gamma}$ into \cref{eq:S-bound} to find that
\[
    S_{N,M,d,q} \ll_\eps q^{O(\eps)} \left(q^{\frac{63+53\delta}{64} + \frac{5\gamma}{4}}
    +
    q^{1-\frac{\delta}{4}+\gamma} + q^{\frac{23}{24} + 2\gamma}
    \right),
\]
which is a stronger bound than \cref{eq:case-2-S-bound}. This completes our proof.
\end{proof}

\section{Large sieve for exceptional cusp forms} \label{sec:exc-large-sieve}

Here we prove a generalization of \cref{cor:large-sieve}, which requires some background from the spectral theory of automorphic forms. We recall \cite{deshouillers1982kloosterman} that for $q \in \Z_+$, the congruence subgroup $\Gamma_0(q)$ contains those matrices in $\SL_2(\Z)$ with bottom-left entries divisible by $q$. Each cusp $\ma$ of the the fundamental domain $\Gamma_0(q)\backslash \H$ is equivalent to a fraction of the form $\tfrac{u}{w}$, where $u, w \in \Z_+$, $w \mid q$, $(u, w) = 1$, and $u \le (w, \tfrac{q}{w})$; in particular, the cusp at $\infty$ is equivalent to $\tfrac{1}{q}$. To such a cusp, one can associate a scaling matrix $\sigma_\ma \in \PSL_2(\R)$ with $\sigma_\ma \infty = \ma$, and via these scaling matrices, functions on $\Gamma_0(q)\backslash \H$ can be Fourier expanded around $\ma$. 

The discrete spectrum of the hyperbolic Laplacian $\Delta = -y^2(\partial_x^2 + \partial_y^2)$ is parametrized by Maass cusp forms: these are smooth functions $f : \Gamma_0(q) \backslash \H \to \C$ which are eigenfunctions of $\Delta$, vanish at all cusps of $\Gamma_0(q)\backslash \H$, and are square-integrable with respect to the Petersson inner product. Following the normalization of Deshouillers--Iwaniec \cite{deshouillers1982kloosterman}, we write the Fourier expansion of $f$ at $z = x+iy \in \H$ around a cusp $\ma$ (with scaling matrix $\sigma_\ma$) as
\[
    f(\sigma_{\ma} z) = y^{1/2} \sum_{n \neq 0} \rho_{\ma}(n) K_{i\kappa}(2\pi |n| y)\, e(nx),
\]
where $K$ is a Whittaker function as in \cite[p.\,264]{deshouillers1982kloosterman}. Altering the choice of scaling matrix $\sigma_\ma$ results in multiplying the Fourier coefficients $\rho_{\ma}(n)$ by an exponential phase $e(n\omega)$, for some uniform $\omega \in \R/\Z$. 

The Kuznetsov trace formula \cite{deshouillers1982kloosterman,kuznetsov1980petersson}, as well as the large sieve inequalities that derive from it, involve an orthonormal basis of Maass cusp forms. The following notation will therefore be useful.

\begin{notation} \label{not:orthonormal-basis}
Let $q \in \Z_+$, $\ma$ be a cusp of $\Gamma_0(q)$ equivalent\footnote{The assumption that $\ma$ is equivalent to $\tfrac{1}{s}$ is true in most applications (note that this includes the cusp at $\infty$), and only made for convenience; one can prove similar results at arbitrary cusps with small adjustments.} to $\tfrac{1}{s}$ for some $s \mid q$ with $(s, \tfrac{q}{s}) = 1$, and $\sigma_\ma \in \PSL_2(\R)$ be any scaling matrix for $\ma$. Consider an orthonormal basis $(f_j)_{j \ge 1}$ of Maass cusp forms for $\Gamma_0(q)$, with:
\begin{itemize}
    \item[$(i)$.] Laplacian eigenvalues $\lambda_j$ and spectral parameters $\theta_j := \max(0, \tfrac{1}{4}-\lambda_j)^{1/2}$;
    \item[$(ii)$.] Fourier coefficients $(\rho_{j\ma}(n))_{n \in \Z}$ around the cusp $\ma$, using the scaling matrix $\sigma_\ma$.
\end{itemize}
\end{notation}

\begin{proposition} \label{prop:prelim-bound-exc}
Assume \cref{not:orthonormal-basis}, let $X, N \ge 1/2$, and let $(\alpha_n)_{n \sim N}$ be a complex sequence. Let $\Phi : \R \to [0,\infty)$ be a smooth function supported in a fixed sub-interval of $(0, \infty)$, with $\int \Phi(t)\, dt \gg 1$ and $\Phi^{(j)}(t) \ll_j 1$. Then there exists $\omega \in \R/\Z$ (depending only on $\ma$, $\sigma_\ma$) such that 
\begin{equation} \label{eq:prelim-bound-exc}
\begin{aligned}
    \sum_{\lambda_j < 1/4} X^{2\theta_j} \left\vert \sum_{n \sim N} \alpha_n\, \rho_{j\ma}(n) \right\vert^2 
    &\ll 
    (qN)^{o(1)}\left(1 + \frac{N}{q}\right) \|\alpha\|^2
    \\
    &+
    \left\vert \sum_{c \in q\Z_+} \frac{1}{c} 
    \sum_{m,n \sim N} \bar{\alpha_m e(m\omega)} \, \alpha_n e(n\omega)\, S(m, n; c)\, \Phi\left(\frac{\sqrt{mn}}{c} X \right) \right\vert.
\end{aligned}
\end{equation}
\end{proposition}

\begin{proof}
This is \cite[Corollary 3.10]{pascadi2026large}, which follows from the Kuznetsov trace formula and the regular-spectrum large sieve inequalities of Deshouillers--Iwaniec \cite[Theorem 2]{deshouillers1982kloosterman}. We have implicitly used \cite[Lemma 3.2]{pascadi2026large} to write down the Kloosterman sums and $c$-supports for cusps $\ma$ equivalent to $\tfrac{1}{s}$ for some $s \mid q$ and $(s, \tfrac{q}{s}) = 1$ (the latter condition is written as $\mu(\ma) = q^{-1}$ in loc.\ cit.). Note that we incur factors of $e(m\omega)$ and $e(n\omega)$ since we do not assume a special scaling matrix $\sigma_\ma$ (as we may), but this will be irrelevant in our computations since the sequence $(\alpha_n)$ is arbitrary.
\end{proof}

In the right-hand side of \cref{eq:prelim-bound-exc}, the sum over $c$ is really supported on $c \asymp NX$ due to the $\Phi$-weight, and it vanishes if $q \gg NX$ with a large enough implied constant. Deshouillers--Iwaniec used this simple observation to deduce the following result, which combines \cite[Theorems 2 and 5]{deshouillers1982kloosterman}.

\begin{theorem}[Deshouillers--Iwaniec \cite{deshouillers1982kloosterman}] \label{thm:di}
Assume \cref{not:orthonormal-basis}, let $N \ge \tfrac{1}{2}$, and let $(\alpha_n)_{n \sim N}$ be a complex sequence. Then one has
\begin{equation} \label{eq:large-sieve-bound}
    \sum_{\lambda_j < 1/4}
    X^{2\theta_j} 
    \left\vert 
    \sum_{n \sim N} \alpha_n\, \rho_{j\ma}(n)
    \right\vert^2 
    \ll
    (qN)^{o(1)}
    \left(1 + \frac{N}{q}\right) \|\alpha\|^2,
\end{equation}
for any positive $X \ll 1 + \frac{q}{N}$.
\end{theorem}

Until now, if $\sqrt{q} \ll N \ll q$, \cref{thm:di} has been the state-of-the-art exceptional-spectrum large sieve bound for general sequences $(\alpha_n)$ and a single group $\Gamma_0(q)$; the same is true if one averages over levels $q \sim Q$ and allows the sequence $(\alpha_n)$ to depend on $q$. 

We can now achieve an improvement of \cref{thm:di} when $q$ has a factorization as in \cref{thm:MN-bilinear-forms-composite}, and similar results can be deduced for arbitrary levels $q$ using \cref{thm:MN-bilinear-forms-general}. We require a coprimality constraint $(n, q) = 1$ for technical reasons, but this is usually harmless in applications. The resulting power savings are relatively small, but serve as a proof of concept that \cref{thm:di} is not a fundamental barrier.

\begin{theorem}[Large sieve for composite levels] \label{thm:large-sieve-comp-lvl}
Assume \cref{not:orthonormal-basis}, let $N \ge \tfrac{1}{2}$, and let $(\alpha_n)_{n \sim N}$ be a complex sequence supported on $(n, q) = 1$. Suppose that $q = dd'e$ with $d' \mid d$ and $(d, e) = 1$, and let $f \le \sqrt{qd}$ be the largest integer with $f^2 \mid qd$. Then \cref{eq:large-sieve-bound} holds for any positive
\begin{equation} \label{eq:X-range-new}
    X \ll 1 + \frac{q}{N} + \min\left(\frac{q^2}{d^{1/3}N^{7/3}}, \frac{q^{3/2}}{f^{1/4}N^{3/2}}, \frac{qd^{1/3}}{f^{1/6}N}\right) + \min\left(\frac{q^{7/4}}{d^{1/4}N^2}, \frac{q^{7/5}}{f^{1/5}N^{7/5}}, \frac{qd^{2/5}}{f^{1/5}N} \right).
\end{equation}
\end{theorem}

\begin{proof}[Proof of \cref{thm:large-sieve-comp-lvl}]
We may assume without loss of generality that
\begin{equation} \label{eq:X-range-updated}
    1 + \frac{q}{N} < X \ll \min\left(\frac{q^2}{d^{1/3}N^{7/3}}, \frac{q^{3/2}}{f^{1/4}N^{3/2}}, \frac{qd^{1/3}}{f^{1/6}N}\right) + \min\left(\frac{q^{7/4}}{d^{1/4}N^2}, \frac{q^{7/5}}{f^{1/5}N^{7/5}}, \frac{qd^{2/5}}{f^{1/5}N} \right),
\end{equation}
since otherwise the result follows from \cref{thm:di}.
We apply \cref{prop:prelim-bound-exc} with a choice of $\Phi$ supported on $[2, 4]$, then separate variables in the smooth weight $\Phi(\cdot)$ via two-dimensional Fourier inversion, as in \cite[Proof of Theorem 5.2]{pascadi2026large}, to arrive at
\[
\begin{aligned} 
    \sum_{\lambda_j < 1/4}
    X^{2\theta_j} 
    \left\vert 
    \sum_{n \sim N} \alpha_n\, \rho_{j\ma}(n)
    \right\vert^2 
    &\ll
    (qN)^{o(1)}
    \left(1 + \frac{N}{q}\right) \|\alpha\|^2
    \\
    &+
    \sum_{\substack{\frac{NX}{4} < c \le NX \\ q \mid c}} \frac{1}{c}
    \sup_{\substack{(\beta_n)_{n \sim N} \\ |\beta_n| = |\alpha_n|}} 
    \sup_{\substack{(\gamma_n)_{n \sim N} \\ |\gamma_n| = |\alpha_n|}}
    \left\vert \sum_{m,n \sim N} \beta_m \gamma_n S(m, n; c) \right\vert.
\end{aligned}
\]
The sequences $(\beta_n)$, $(\gamma_n)$ in the supremum arise by incorporating exponential phases $e(n\omega)$ into $(\alpha_n)$, partly from the choice of the scaling matrix $\sigma_\ma$, and partly due to the separation of variables. The suprema are of course attained by some sequences $(\beta_n)$, $(\gamma_n)$ supported on $(n, q) = 1$, so we can apply \cref{cor:MN-bilinear-forms-avg-c} with $M = N$ and $C \asymp NX$, to obtain
\[
\begin{aligned} 
    \sum_{\lambda_j < 1/4}
    X^{2\theta_j} 
    \left\vert 
    \sum_{n \sim N} \alpha_n\, \rho_{j\ma}(n)
    \right\vert^2 
    &\ll
    (qN)^{o(1)}
    \left(1 + \frac{N}{q}\right) \|\alpha\|^2
    \\
    &+
    \|\alpha\|^2
    \frac{(NX)^{1+o(1)}}{q}
    \min 
    \begin{cases} 
    \left(\frac{dN}{X^3} + \frac{f}{X^2} + \frac{f}{d^2} \right)^{\frac{1}{6}},
    \\ 
    \left(\frac{dN^2}{qX^2} + \frac{fN}{qX} + \frac{f q}{d^2 NX} \right)^{\frac{1}{6}}.
    \end{cases}
\end{aligned}
\]
We conclude by noting that
\[
    \frac{NX}{q}
    \left(\frac{dN}{X^3} + \frac{f}{X^2} + \frac{f}{d^2} \right)^{\frac{1}{6}} \ll 1
    \qquad 
    \text{for} 
    \qquad 
    X \ll \min\left(\frac{q^2}{d^{1/3}N^{7/3}}, \frac{q^{3/2}}{f^{1/4}N^{3/2}}, \frac{qd^{1/3}}{f^{1/6}N}\right),
\]
and 
\[
    \frac{NX}{q}
    \left(\frac{dN^2}{qX^2} + \frac{fN}{qX} + \frac{f q}{d^2 NX} \right)^{\frac{1}{6}} \ll 1
    \qquad 
    \text{for} 
    \qquad 
    X \ll \min\left(\frac{q^{7/4}}{d^{1/4}N^2}, \frac{q^{7/5}}{f^{1/5}N^{7/5}}, \frac{qd^{2/5}}{f^{1/5}N} \right).
\]
This covers the range in \cref{eq:X-range-updated}.
\end{proof}

\begin{proof}[Proof of \cref{cor:large-sieve}]
If $q$ has a divisor $d \asymp \sqrt{q}$ such that $\tfrac{q}{d}$ is square-free, then we can take $d' = (d, \tfrac{q}{d})$ and $f = d \asymp \sqrt{q}$ in \cref{thm:large-sieve-comp-lvl}, so \cref{eq:large-sieve-bound} holds for any positive
\[
    X \ll 1 + \frac{q}{N} + \min\left(\frac{q^{11/6}}{N^{7/3}}, \frac{q^{11/8}}{N^{3/2}}, \frac{q^{13/12}}{N}\right) + \min\left(\frac{q^{13/8}}{N^2}, \frac{q^{13/10}}{N^{7/5}}, \frac{q^{11/10}}{N} \right).
\]
If additionally $N \ll q^{1/2+o(1)}$ (as \cref{cor:large-sieve} assumes), then we can take $X = q^{3/5}$, since this is only larger by a factor of $q^{o(1)}$ than the second minimum above.
\end{proof}

\appendix 

\section{On words in \texorpdfstring{$\PSL_2(\Z/c\Z)$}{PSL2(Z/cZ)}}

Here we give some additional context around \cref{conj:counting}, which is not strictly required in our main proofs (but which may be informative to our strategy). In particular, we prove a corresponding lower bound to \cref{conj:counting} in \cref{lem:counting-lb}.

\begin{lemma} \label{lem:counting-4}
\cref{conj:counting} holds if $q = 2$ or $q = 4$.
\end{lemma}

\begin{proof}
When $q = 2$, the equation in \cref{eq:PSL2-eqn} reads $T^{a_1h_1} S = S T^{-a_2h_2}$, which implies the entry-wise congruence
\[
    \begin{pmatrix} 
    a_1 h_1 & -1 \\ 
    1 & 0
    \end{pmatrix}
    \equiv
    \gamma
    \begin{pmatrix} 
    0 & -1 \\ 
    1 & - a_2 h_2
    \end{pmatrix}
    \pmod{c},
\]
for some $\gamma \in \Z/c\Z$ with $\gamma^2 = 1$. This actually forces $\gamma = 1$, and $h_1, h_2 \equiv 0 \pmod{c}$. Since $H_1, H_2 \ll c$, we obtain $O(1)$ choices of $h_1, h_2$, which matches the bound from \cref{eq:conj-counting}.

When $q = 4$, the equation in \cref{eq:PSL2-eqn} reads $T^{a_1h_1} S T^{a_2h_2} S = S T^{-a_2h_4} S T^{-a_1h_3}$, which translates to
\[
    \begin{pmatrix} 
    a_1a_2 h_1 h_2 - 1 & -a_1 h_1 \\ 
    a_2 h_2 & -1
    \end{pmatrix}
    \equiv
    \gamma
    \begin{pmatrix} 
    -1 & a_1 h_3 \\ 
    -a_1 h_4 & a_1 a_2 h_3 h_4 - 1
    \end{pmatrix}
    \pmod{c},
\]
for some $\gamma \in \Z/c\Z$ with $\gamma^2 = 1$. Since there are $c^{o(1)}$ such values of $\gamma$ (recall \cref{eq:center-bound}), we may fix $\gamma$ up to an acceptable loss. To establish the desired bound of $O(c^{o(1)} H_2)$ for the number of solutions $(h_1, h_2, h_3, h_4)$ with $|h_1|, |h_3| \le H_1$ and $|h_2|, |h_4| \le H_2$, we split into three cases.

\textbf{Case 1:} $h_1 = 0$. This forces $h_3 \equiv 0 \pmod{c}$, and each choice of $h_2$ induces a unique residue of $h_4 \pmod{c}$. Since $H_1, H_2 \ll c$, this gives a total of $O(c^{o(1)} H_2)$ solutions.

\textbf{Case 2:} $h_2 = 0$. This forces $h_4 \equiv 0 \pmod{c}$, and each choice of $h_1$ induces a unique residue of $h_3 \pmod{c}$. Since $H_1, H_2 \ll c$, this gives a total of $O(c^{o(1)} H_1)$ solutions, and recall $H_1 \le H_2$.

\textbf{Case 3:} $h_1 h_2 \neq 0$. Then the congruence $a_1 a_2 h_1 h_2 - 1 \equiv -\gamma \pmod{c}$ fixes the residue of $h_1 h_2 \pmod{c}$, leaving $O(1 + \tfrac{H_1 H_2}{c})$ possible values of $h_1 h_2$, each of which gives $O(c^{o(1)})$ choices of $h_1, h_2$ by the divisor bound. This gives a total of $\ll c^{o(1)}(1 + \tfrac{H_1H_2}{c}) \ll c^{o(1)} H_2$ solutions.
\end{proof}

We can also remove the constraint $H_1, H_2 \ll c$ from this result when $q = 4$, as in \cref{cor:counting-remove-ub-H12}.

\begin{corollary}
Let $c \in \Z_+$, $a_1, a_2 \in (\Z/c\Z)^\times$, and $1 \le H_1 \le H_2$. Then the number of solutions to \cref{eq:PSL2-eqn} with $q = 4$ is at most
\begin{equation} \label{eq:counting-4-refined}
    \ll c^{o(1)} \left(1 + \frac{H_1^2}{c^2}\right) \left(1 + \frac{H_2}{c}\right) H_2.
\end{equation}
\end{corollary}

\begin{proof}
Following the proof of \cref{cor:counting-remove-ub-H12}, we can count the number of solutions to the system \cref{eq:PSL2-eqn-mod} and then multiply the result by $(1 + \tfrac{H_1}{c})^{q/2}(1 + \tfrac{H_2}{c})^{q/2}$.

Taking $q = 4$ and applying \cref{lem:counting-4} for $\min(H_1, c)$ and $\min(H_2, c)$, we obtain a total number of solutions of
\[
    \ll c^{o(1)} \left(1 + \frac{H_1}{c}\right)^2 \left(1 + \frac{H_2}{c}\right)^2 \min(H_2, c).
\]
The bound in \cref{eq:counting-4-refined} follows by noting that $\min(H_2, c) \asymp H_2 \left(1 + \frac{H_2}{c}\right)^{-1}$.
\end{proof}

\begin{lemma}\label{lem:counting-lb}
Let $c, q \in \Z_+$ with $q$ even, $a_1, a_2 \in (\Z/c\Z)^\times$, and $1 \le H_1 \le H_2 \ll c$. The number of solutions $(h_1, \ldots, h_q) \in \Z^q$ to \cref{eq:PSL2-eqn} is at least
\begin{equation} \label{eq:counting-lb}
    \gg_q H_2^{(q-2)/2} + \frac{(H_1 H_2)^{q/2}}{c^3}.
\end{equation}
\end{lemma}
\begin{proof}
The lower bound by the first term in \cref{eq:counting-lb} follows by considering the aforementioned integer solutions with $h_1 = h_3 = \cdots = h_{q-1} = 0$ and $h_2 + h_4 + \cdots + h_q = 0$. To obtain a lower bound by the second term in \cref{eq:counting-lb}, we let $k := \tfrac{q+2}{2}$, write $(H_j, a_j)$ for $(H_1, a_1)$ or $(H_2, a_2)$ depending on whether $j$ is odd or even, and apply Cauchy--Schwarz to obtain
\[
\begin{aligned}
    \sum_{\substack{h_1, \ldots, h_k \in \Z \\ |h_i| \le \frac{1}{2} H_j \\ \forall i \equiv j \pmod{2}}} 1
    &=
    \sum_{g \in \PSL_2(\Z/c\Z)}
    \sum_{\substack{h_1, \ldots, h_k \in \Z \\ |h_i| \le \frac{1}{2} H_j \\ \forall i \equiv j \pmod{2}}} \one_{T^{a_1h_1}S \cdots T^{a_kh_k}S = g}
    \\
    &\ll 
    c^{3/2}
    \Bigg(\sum_{g \in \PSL_2(\Z/c\Z)}
    \Bigg( \sum_{\substack{h_1, \ldots, h_k \in \Z \\ |h_i| \le \frac{1}{2}H_j \\ \forall i \equiv j \pmod{2}}} \one_{T^{a_1h_1}S \cdots T^{a_kh_k}S = g}\Bigg)^2 \Bigg)^{1/2}
    \\
    &=
    c^{3/2}
    \Bigg(\sum_{\substack{h_1, \ldots, h_k \in \Z \\ h_1', \ldots, h_k' \in \Z \\ |h_i|, |h_i'| \le \frac{1}{2}H_j \\ \forall i \equiv j \pmod{2}}} \one_{T^{a_1h_1} S \cdots T^{a_kh_k} S = T^{a_1h_1'} S \cdots T^{a_k h_k'}S \text{ in } \PSL_2(\Z/c\Z)} \Bigg)^{1/2}.
\end{aligned}
\]
One can rewrite the last equation $T^{a_1h_1} S \cdots T^{a_kh_k} S = T^{a_1h_1'} S \cdots T^{a_k h_k'}S$ in $\PSL_2(\Z/c\Z)$ as 
\[
    T^{a_1(h_1-h_1')} S T^{a_2h_2} S \cdots T^{a_{k-1}h_{k-1}} S T^{a_k(h_k-h_k')} S T^{-a_{k-1} h_{k-1}'} S \cdots T^{-a_1 h_1'} S = I.
\]
Comparing this with \cref{eq:PSL2-eqn}, recalling that $k = \tfrac{q+2}{2}$, and noting that $h_1 - h_1'$ takes each value in $\Z \cap [-H_1, H_1]$ at most $O(H_1)$ times (and similarly for $h_k - h_k'$), we conclude that the desired count of solutions is at least
\[
    \gg_q \frac{1}{c^3 H_1 H_k} \Bigg(\sum_{\substack{h_1, \ldots, h_k \in \Z \\ |h_i| \le \frac{1}{2} H_j \\ \forall i \equiv j \pmod{2}}} 1\Bigg)^2
    \gg_q 
    \frac{1}{c^3 H_1 H_k} H_1^{2\lc k/2 \rc} H_2^{2\lf k/2 \rf}
    =
    \frac{H_1^{k-1} H_2^{k-1}}{c^3}.
\]
Since $k-1 = \tfrac{q}{2}$, this completes our proof.
\end{proof}

\bibliographystyle{plain}
\bibliography{main}

\end{document}